%% file: effBHC.tex
\documentclass[12pt,oneside]{amsart}
\usepackage{amscd,a4,amssymb,amsmath,mathrsfs}
\usepackage{typearea}
\usepackage{hyperref}
\usepackage{color}

\input{macros.tex}

\newcommand{\hilbfunc}[1]{{\mathscr H}_g({#1})}
\newcommand{\arithhilb}[1]{{\mathscr H}_a({#1})}
\newcommand{\hhilbpoly}[1]{{H}({#1})}
\newcommand{\bigO}[1]{{\mathcal O}({#1})}

\newcommand{\mahler}[2]{m_{#1}({#2})}

\begin{document}
\title{Effective Height Upper Bounds on Algebraic Tori}
\author[Philipp Habegger]{P. Habegger}
\date{Last update: \today}

\maketitle



This manuscript was written for the CIRM Workshop 
``La conjecture de Zilber-Pink''
held in Luminy, May 2011 and organized by Ga\"el R\'emond and Emmanuel Ullmo.

 The main emphasis will be on height upper bounds in the
algebraic torus $\IG_m^n$. By  height we will mean the absolute logarithmic Weil height. 
Section \ref{sec:heights} contains a precise definition of this and
other more general height functions.
The first appendix gives a short overview of known results in the
abelian case.  The second appendix
contains a few height bounds in Shimura
varieties.

\section{A Brief Historical Overview in the Toric Setting}
\label{sec:history}

In 1999, Bombieri, Masser, and Zannier 
 proved the following result.

 \begin{theorem}[Bombieri, Masser, and Zannier \cite{BMZ}]
\label{thm:BMZ}
Let $C$ be an irreducible algebraic curve inside $\IGm^n$ and defined
over $\IQbar$, an algebraic closure of $\IQ$. Suppose that $C$ is not contained in a proper coset, that is 
 the translate of a
proper algebraic subgroup.
Then the height of 
 points on $C$ that are contained in a proper algebraic subgroup
 is bounded from above uniformly.
 \end{theorem}

The adjective uniformly refers to the fact that the bound for the
height does not depend on the algebraic subgroup. It may and will
depend on the curve $C$.

Using this height upper bound together with height lower bounds in the
context of Lehmer's question they were able to prove the following
theorem. 

\begin{theorem}[Bombieri, Masser, and Zannier \cite{BMZ}]
\label{thm:BMZfinite}
Let $C$ be as in Theorem \ref{thm:BMZ}.  
There are only finitely many points on $C$ that are contained in an
algebraic subgroup of codimension at least $2$.
\end{theorem}

Heuristically, points considered in this theorem are
  rare since they are inside the
\emph{unlikely intersection} of a  curve and a subvariety of codimension $2$.
The difficulty in proving such a theorem 
arises from the fact that the codimension $2$
subvarieties varies over an infinite family. 

\begin{example*}
We exemplify this on the example of a line in $\IG_m^3$. 
The line is parametrized by
\begin{equation}
\label{eq:line}
  (\alpha t + \alpha', \beta t + \beta',\gamma t + \gamma')
\end{equation}
with 
$\alpha,\beta,\gamma,\alpha',\beta',\gamma'\in \IQbar$ fixed.
To guarantee that our line is not contained in a proper coset
of $\IG_m^3$ we ask that
\begin{equation*}
  \alpha\beta\gamma\not= 0
\quad\text{and}\quad 
\frac{\alpha'}{\alpha}, 
\frac{\beta'}{\beta},
\frac{\gamma'}{\gamma}
\text{ are pair-wise distinct.}
\end{equation*}

Theorem \ref{thm:BMZ} implies that there is a constant $B$ with the
following property. Suppose $t\in \IQbar$ such
that no coordinate of (\ref{eq:line}) vanishes and such that
there exist
$(a,b,c)\in\IZ^3$ with
\begin{equation}
\label{eq:multrel}
  (\alpha t + \alpha')^a (\beta t + \beta')^b (\gamma t + \gamma')^c =
  1
\end{equation}
then the  height of $t$ is at most $B$.

The exponent vector $(a,b,c)$ is allowed
to vary in (\ref{eq:multrel}). It determines a proper algebraic subgroup of
$\IG_m^3$. Conversely, any such subgroup arises in this way.

In order to obtain finiteness of the set of $t$ as in the second result of Bombieri,
Masser, and Zannier we must impose a second condition. More precisely,
there are only finitely many $t$ with (\ref{eq:multrel})
such that there is
$(a',b',c')\in\IZ^3$ linearly independent of $(a,b,c)$ with
\begin{equation*}
  (\alpha t + \alpha')^{a'} (\beta t + \beta')^{b'} (\gamma t + \gamma')^{c'} =
  1.
\end{equation*}

The two monomial equations determined by $(a,b,c)$ and $(a',b',c')$
define an algebraic subgroup of $\IG_m^3$ of codimension $2$. 
\end{example*}

Bombieri, Masser, and Zannier remarked that any curve 
contained in a proper coset invariable leads to unbounded height.
So this condition cannot be dropped from their theorem.
On the other hand, it remained unclear if the restriction was
necessary in order to achieve finiteness in the situation of 
unlikely intersections. The authors posed the following question.

\begin{question}
Suppose $C\subset\IGm^n$ is an irreducible algebraic curve defined
over $\IQbar$ that is not contained in a proper algebraic subgroup.
Is the set of points on $C$ that are contained in an algebraic
subgroup of codimension $2$ finite?  
\end{question}

First progress was made in 2006 when the same group of authors
  obtained a partial answer.
By making a detour to surfaces, Bombieri, Masser,
and Zannier gave a positive answer to their question in low dimension.

\begin{theorem}[Bombieri, Masser, and Zannier \cite{BMZ3}]
\label{thm:BMZneq5}
Suppose $n\le 5$.
Let $C$ be an irreducible algebraic curve inside $\IGm^n$ and defined
over $\IQbar$. Suppose that $C$ is not contained in a proper 
 algebraic subgroup.
  There are only finitely many points on $C$ that are contained in an
algebraic subgroup of codimension at least $2$.
\end{theorem}

In the most interesting case $n=5$ they constructed an algebraic  surface
 $S\subset\IG_m^3$ derived from the curve.
In order to answer their question for $C$, a bounded height result
 akin to Theorem
\ref{thm:BMZ} was needed for points on $S$ lying in an algebraic subgroup of
codimension $\dim S=2$. 
So the algebraic subgroups in question have dimension at most $1$. 
Luckily, height bounds on the 
intersection a fixed variety with varying algebraic subgroups
of dimension $1$ can be handled
 by  early work of Bombieri and Zannier on subvarieties
 $X\subset\IG_m^n$ of unrestricted dimension.

\begin{theorem}[Bombieri and Zannier \cite{ZannierAppendix}]
\label{thm:BZ}
Suppose $X\subset\IGm^n$ is an irreducible algebraic subvariety defined
over $\IQbar$. Let $X^o$ be the complement in $X$ of the union of all
cosets of positive dimension that are contained in $X$.
Then the  height of 
 points on $X$ that are contained in an algebraic subgroup
of dimension at most $1$ is uniformly bounded.
\end{theorem}

  The construction  used by the three authors works for $n>5$ too
   and always yields a surface $S$. Unfortunately, it will not  be a
  hypersurface  in general.
Thus the algebraic subgroups, having codimension $\dim S$,
are too large for Theorem \ref{thm:BZ}.

One interesting aspect of  Theorem
\ref{thm:BMZneq5}
 is that its  proof is effective, provided an
 effective height bound in Theorem \ref{thm:BZ} is available.
The author \cite{smallsbgrp} made a generalization of Theorem
\ref{thm:BZ}
effective and completely explicit.
He gave a height bound for the points considered in Theorem
\ref{thm:BZ} in terms of $n$ and the degree and height of $X$.

In 2008, and 
using a different approach, Maurin 
 gave a positive answer to the
question posed above for all $n$. 

\begin{theorem}[Maurin \cite{Maurin}]
Let $C$ be an irreducible algebraic curve inside $\IGm^n$ and defined
over $\IQbar$. Suppose that $C$ is not contained in a proper 
 algebraic subgroup.
  There are only finitely many points on $C$ that are contained in an
algebraic subgroup of codimension at least $2$.  
\end{theorem}

Maurin first proves that the height is bounded from above. 
But his method  differed substantially from the one used in the
proof of Theorem \ref{thm:BMZ}. It relied on 
generalization of Vojta's inequality by R\'emond
\cite{Remond:Vojtagen} which is able to account for the varying
algebraic subgroups. The original Vojta inequality appeared
prominently in a proof of Faltings' Theorem, the Mordell
Conjecture. 

Maurin's Theorem holds for curves defined over $\IQbar$. But its
statement makes no reference to algebraic numbers. Having it
 at their disposal  
 Bombieri, Masser, and Zannier 
applied specialization techniques to obtain the following result.

\begin{theorem}[Bombieri, Masser, and Zannier \cite{BMZUnlikely}]
Let $C$ be an irreducible algebraic curve inside $\IGm^n$ and defined
over $\IC$. Suppose that $C$ is not contained in a proper 
 algebraic subgroup.
  There are only finitely many points on $C$ that are contained in an
algebraic subgroup of codimension at least $2$.    
\end{theorem}

For curves this theorem implies Zilber's Conjecture 1 \cite{Zilber} on anomalous
intersections, also dubbed the \emph{Conjecture on Intersection with Tori}
or short CIT. It also resolves, in the case of
curves, the related Torsion Finiteness Conjecture stated
by Bombieri, Masser, and Zannier \cite{BMZGeometric}

In the mean time, progress was being made on boundedness of height for
surfaces by Bombieri, Masser, and Zannier.
A statement of their result requires the definition of the \emph{open
  anomalous} set $\oa{X}$ of an irreducible closed subvariety
$X\subset\IG_m^n$. We will give a proper definition further down
near (\ref{eq:defineoa}).
 
If $C$ is a curve then $\oa{C}=C$ if and only if $C$ is not contained
in a proper coset. Otherwise, we will have $\oa{C}=\emptyset$. 
So Theorem \ref{thm:BMZ} can be reformulated succinctly follows. Say
$C\subset\IG_m^n$ is an irreducible algebraic curve defined over
$\IQ$. Then any point on $\oa{C}$ contained in a proper algebraic subgroup
has height bounded from above uniformly.

\begin{theorem}[Bombieri, Masser, and Zannier \cite{BMZPlanes}]
Let $P$ be a plane inside $\IGm^n$ and defined
over $\IQbar$. The open anomalous set $\oa{P}$ is Zariski open in $P$.
The  height of 
 points on $\oa{P}$ that are contained in a  algebraic subgroup
of codimension at least $2$ is bounded from above uniformly.
\end{theorem}

Unfortunately, for $X$ of dimension strictly greater than $1$ 
the anomalous $\oa{X}$ cannot be described as easily as for curves.
Their complete result also contains a precise description of $\oa{P}$
and in particular a necessary and sufficienty condition for
$\oa{P}\not=\emptyset$. 


These three authors expected  boundedness of height to hold for
general subvarieties of $\IGm^n$ and thus formulated the Bounded Height
Conjecture  \cite{BMZGeometric}.
Heuristic observations dictate  the algebraic subgroups of have dimension
complementary to the dimension of the fixed variety. In other words,
if the subvariety of $\IGm^n$ has dimension $r$, one must intersect
with algebraic subgroups of codimension  $r$.

For arbitrary $X$, the following result of Bombieri, Masser, and
Zannier controls $\oa{X}$.

\begin{theorem}[Bombieri, Masser, and Zannier \cite{BMZGeometric}]
  Let $X\subset\IG_m^n$ be an irreducible closed subvariety defined
  over $\IC$. Then $\oa{X}$ is Zariski open in $X$.
\end{theorem}

In fact they prove a stronger structure theorem for $\oa{X}$.

Later in 2009, the author showed  boundedness of height for arbitrary
subvarieties of $\IGm^n$. 

\begin{theorem}[\cite{BHC}]
\label{thm:BHC}
  Let $X\subset\IG_m^n$ be an irreducible closed subvariety defined
  over $\IQbar$.
The  height of 
 points on $\oa{X}$ that are contained in an  algebraic subgroup
of codimension at least $\dim X$ is bounded from above uniformly.
\end{theorem}

This height bound is now sufficiently general to prove height bounds
on surfaces
arising from the strategy devised by Bombieri, Masser, and Zannier
in their proof of Theorem \ref{thm:BMZneq5}.
In joint work with this group, the author 
gave a new proof of Maurin's Theorem \cite{BHMZ}.

The purpose of this article is to provide an effective and fully explicit
version of Theorem \ref{thm:BHC}. In combination with the methods
presented in \cite{BHMZ} we obtain an effective version of Maurin's Theorem. 

In recent work,
 Maurin obtained a height bound when the algebraic subgroups in
question have codimension at least $1+\dim X$ but are enlarged by a
subgroup of finite type.

\begin{theorem}[Maurin  \cite{Maurin2}]
\label{thm:maurin2}
  Let $X\subset\IG_m^n$ be an irreducible closed subvariety defined
  over $\IQbar$ and let $\Gamma \subset \IG_m^n(\IQbar)$ be the division
  closure of a finitely generated subgroup.
  The set of points $ab\in \oa{X}$ with $a\in \Gamma$ and $b$ in an
  algebraic subgroup of codimension at least $1+\dim X$ has uniformly
  bounded height.
\end{theorem}

From this height bound, he is able to deduce the non-Zariski denseness
of the points $ab$ in question. 

\section{An Effective Height Bound}
\label{sec:effbhc}


In order to state our result we must define 
 the anomalous locus of an irreducible
closed subvariety $X\subset\IGm^n$ defined over $\IC$. 

Let $s \ge 0$ be an integer. An irreducible closed subvariety
$Z\subset X$ is called  $s$-anomalous if there exists a coset
$H\subset \IGm^n$ containing $Z$ such that
\begin{equation*}
  \dim Z \ge \max\{1, s + \dim H - n + 1\}.
\end{equation*}
If $s=\dim X$ then $Z$ is called anomalous. The
($s$-)anomalous locus of $X$ is the union of all 
($s$-)anomalous subvarieties of $X$.
We define $\oap{X}{s}$ to be the complement of the $s$-anomalous
locus and set
\begin{equation}
\label{eq:defineoa}
  \oa{X} = \oap{X}{\dim X}.
\end{equation}

We also set
\begin{equation*}
  \sgu{s} = \bigcup_{\atopx{H\subset \IGm^n}{\codim H \ge s}} H(\IQbar)
\end{equation*}
where the union is over all algebraic subgroups $H$ of $\IGm^n$ that
have codimension at least $s$. We remark that all algebraic subgroups
of $\IGm^n$ are defined over $\IQ$.

Our main result is the following explicit version of the Bounded
Height Theorem. 

\begin{theorem}
  \label{thm:effbhc}
Let $n\ge 1$ and
let $X\subset \IGm^n$ be an irreducible closed subvariety defined over
$\IQbar$ of
dimension $r$ and say $s \ge 0$ is an integer. 
We set 
$C = (600 n^7(2n)^{n^2})^r$.
There exists a Zariski closed subset $Z\subset X$ with the following
properties. 
\begin{enumerate}
\item [(i)] We have $\oap{X}{s}\subset X\ssm Z$. 
\item [(ii)] Each irreducible component $V$ of $Z$ satisfies
  \begin{equation*}
    \deg{V}\le (2\deg{X})^C
  \end{equation*}
 and 
 \begin{equation*}
   \height{V}\le
(2\deg{X})^C (1+\height{X}).
 \end{equation*}
 The number of irreducible components of $Z$ is at most
$(2\deg{X})^C$. 
\item[(iii)] If $p \in (X\ssm Z)(\IQbar)\cap \sgu{s}$, then 
\begin{equation*}
  \height{p}\le 
  (2\deg{X})^{C} (1+\height{X}).
\end{equation*}
\end{enumerate}
\end{theorem}

The general proof strategy is similar to the one used by the author to
prove the original Bounded Height Conjecture. 
For example, Ax's Theorem \cite{Ax} still plays a pivotal role.
Although there is one
significant different. The first proof  relied ultimately on a
compactness argument to bound from below a certain intersection
number. Therefore, no effective or even explicit result could be
attained. In the current proof we use instead a recent result of Sturmfels
and Tevelev \cite{SturmfelsTevelev} from 
 Tropical Geometry to make everything explicit.

Going the other way,  we obtain an amusing corollary on the
tropicalization $\trop{X} \subset \IQ^n$ of an irreducible closed
 subvariety of $\IGm^n$ defined over $\IC$. The
proof and necessary definitions are given in Section \ref{sec:tropical}. It
suffices to say at this point that
$\trop{X}$ is a finite union of rational polyhedral cones of dimension
$\dim X$. We let $\overline{\trop{X}}$ denote the topological closure
of $\trop{X}$ inside $\IR^n$. The following corollary 
 shows that the configuration of cones determining
 $\overline{\trop{X}}$ satisfy a rationality
condition with respect to projecting to $\IR^{\dim X}$. It will not play a role in the proof of our main result.

\begin{corollary}
\label{cor:tropical}
 Let $X\subset\IGm^n$ be an irreducible closed subvariety defined over $\IC$
of dimension $r$.
If there is $\varphi_0\in\mat{r,n}{\IR}$ of rank $r$ with 
$\varphi_0(\overline{\trop{X}})\not=\IR^r$
then there exists $\varphi\in\mat{r,n}{\IQ}$ of rank $r$
with $\varphi(\trop{X}) \not=\IQ^r$.
\end{corollary}

This rationality property is non-trivial 
as portrayed
 by an example in
 Section \ref{sec:tropical}. There we exhibit four
two-dimensional vector subspaces $L_{1,2,3,4}\subset\IQ^4$
and a surjective $\varphi_0\in \mat{2,4}{\IR}$ such that 
$\dim \varphi_0(\overline L_i) < 2$ for  $1\le i\le 4$ but with 
\begin{equation*}
  \varphi(L_1\cup L_2\cup L_3\cup L_4) = \IQ^2
\end{equation*}
for all rational surjective  $\varphi\in \mat{2,4}{\IQ}$.

Below $X$ denotes a subvariety of $\IG_m^n$ with dimension $r$
as in our theorem. 
We give a brief overview of the proof of Theorem \ref{thm:effbhc} in
the case $s=r$.
The following steps do not reflect the formal layout of the argument
as it is  given.  We rather  present a simplified account which stays
true to the general idea.

{\bf Step I:} 
Suppose $\varphi:\IG_m^n\rightarrow\IG_m^r$ is a homomorphism of
algebraic groups. 
We will see that $\varphi$ can be represented as an $n\times r$ matrix 
with integral coefficients. 
The morphism $\varphi$ is surjective if and only if the associated
matrix has rank $r$ and we shall assume this.
We let $|\cdot|$ denote any fixed norm on
the  vector space of $n\times r$-matrices in real coefficients. 
The restriction $\varphi|_X$ of $\varphi$ to $X$ is a 
morphism with a certain degree $\Delta(\varphi) \ge 0$. 

In Section \ref{sec:corr} we will  prove a general height inequality
on correspondences in $\IP^n\times\IP^r$. 
If we apply it to the
Zariski closure of the graph of $\varphi|_X$ in $\IP^n\times\IP^r$ we 
obtain two positive constants $c_1$ and $c_2(\varphi) >
0$ with the following properties. First, $c_1$ and $c_2(\varphi)$ are
completely explicit
and given in terms of the height of ${X}$, the degree of $X$, and $n$.  Actually, we allow  $c_2(\varphi)$ to
depend on $\varphi$.
Second, the lower bound
\begin{equation}
\label{eq:varphilb}
  \height{\varphi(p)} \ge c_1 |\varphi| \frac{\Delta(\varphi)}{|\varphi|^{r}}
  \height{p} - c_2(\varphi)
\end{equation}
holds for all algebraic  points $p$ in a Zariski open and dense subset
of $X$ (which depends on $\varphi$).
Moreover, the completement of said Zariski open set can be effectively
determined.
This means that we can bound the height and degree of each irreducible
component as well as their number  in terms of $X$ and
$\varphi$. 
Much of the work involved in deducing this inequality  goes into
making all estimates explicit. 

We observe that if $\varphi(p)$ is the unit element, then the left-hand
side of (\ref{eq:varphilb}) vanishes. 
This is will be the common situation as the kernel of $\varphi$ 
is an algebraic subgroup of codimension $r$. Moreover, any algebraic
subgroup of codimension $r$ arises as such a kernel. 
So if $p$ is inside the Zariski open subset and if $\Delta(\varphi) >
0$ we obtain
\begin{equation*}
 \height{p} \le \frac{c_2(\varphi)}{c_1} \frac{1}{|\varphi|} \frac{|\varphi|^{r}}{\Delta(\varphi)}. 
\end{equation*}

This is indeed a height upper bound. But it is not clear if it implies a
uniform height upper bound: the right-hand side seems to
depend heavily on $\varphi$. Even worse, the Zariski open set on which
it holds also depends on $\varphi$. 

{\bf Step II:} In the second step we get a hold on the degree
$\Delta(\varphi)$ of $\varphi|_X$.
This is an integer and it cannot be negative. It
 vanishes if and only if all fibers of
$\varphi|_X:X\rightarrow\IG_m^r$ have positive dimension. 
Each such fibers arises as the intersection of $X$ with the translate
of an algebraic subgroup of codimension $r$. From our definition we
find that $X$ is covered by anomalous subvarieties, so $\oap{X}{r} =
\emptyset$. In this case, Theorem \ref{thm:effbhc} is void.
So we may assume without loss of generality that $\Delta(\varphi)>
0$ for all surjective $\varphi$. 

What we want to do next is to transform this non-vanishing statement into
a sufficiently strong and 
 explicit lower bound in terms of $\varphi$. This was also a key
step in the first proof of the Bounded Height Conjecture. The main
tool there was Ax's Theorem combined with a compactness argument. In
the current approach
we will still need Ax's Theorem. But the compactness argument is
replaced by a result in 
Tropical Geometry in our Section \ref{sec:tropical}. 
The tropical variety associated to $X$ is a finite union of $r$-dimensional polyhedral
cones. This data can be determined effectively in
terms of $X$. The tropical variety associated to the kernel of
$\varphi$ is quite simply  a vector subspace of $\IQ^n$ of
dimension $n-r$. 
Determining $\Delta(\varphi)$ essentially amounts to intersecting the
tropical variety of $X$ with this vector subspace and counting
the points on the tropical intersection with  correct
 multiplicities.
This last feat is attained by the Theorem of Sturmfels and Tevelev.
We will be able to bound $\Delta(\varphi)$ from below using a certain
polynomial in the entries of $\varphi$.
Ax's Theorem will imply that this polynomial does not vanish when
evaluated on  \emph{real} $n\times r$ matrices of maximal rank.
This information in conjunction with a
compactness argument is enough to provide a qualitative lower bound
for this polynomial. Using an explicit version of \L ojasiewicz's
inequality due to  R\'emond \cite{Remond:Lojasiewicz}
we can transform this qualitative bound into a quantitative one. 
The overall effect will be 
\begin{equation}
\label{eq:deltavarphi}
 \Delta(\varphi) \ge c_3 |\varphi|^r 
\end{equation}
 with $c_3 >0$ independent of
$\varphi$,
at least for $\varphi$ not too close to a matrix of rank less than $r$.
From this inequality we see that the factor
$|\varphi|^r/\Delta(\varphi)$ is harmless. 

{\bf Step III:}
There remains the issue that $\varphi$ depends on the algebraic
subgroup
of $\IG_m^n$. Recall that $c_2(\varphi)$ and
 the Zariski open set on which inequality
(\ref{eq:varphilb}) holds both depend on $\varphi$. 
In order to ensure that our height inequality remains valid for all possible
algebraic subgroups we would have to restrict to the intersection of
 infinitely many Zarsiki
open sets. This is not a wise idea. Luckily,  this problem can be
resolved using a rather simple trick  found in Section \ref{sec:genbh}.
We relax the condition on $\varphi$.
Instead of asking that
 $\varphi(p)$ is the unit element and hence forcing
the left-hand side of (\ref{eq:varphilb}) to vanish,
we instead only want $\height{\varphi(p)}$ to be bounded above
linearly in terms of $\height{p}$. The constant involved will depend
only on $n$. This can be achieved by replacing the matrix whose kernel
defines the algebraic subgroup containing $p$ by a sufficiently good
approximation.
The payoff is that we may choose $\varphi$ from a fixed finite set
depending only on $X$. 
So we will only need to take a finite intersection to apply
(\ref{eq:varphilb}). Moreover, the constant $c_2(\varphi)$ in this
inequality becomes harmless. 
 The right-hand side of
(\ref{eq:varphilb}) involves $|\varphi|^{r+1} / \Delta(\varphi)$.
Compared with (\ref{eq:deltavarphi}) we get an extra factor
$|\varphi|$.
This factor will be used to defeat the factor in front of $\height{p}$
appearing in the upper bound for $\height{\varphi(p)}$. 

Amalgamating step 1 through 3 yields the  Generically Bounded Height
Theorem stated below as Theorem \ref{thm:genhb}.
Roughly speaking, it states that if $\oap{X}{r}\not=\emptyset$, then 
there is a non-empty Zariski open subset $U$ of $X$ on which any point
lying in $\sgu{r}$ has height bounded from above. 
The bound for the height of $p$ is given explicitly and so are bounds
for height, degree, and number of the irreducible components of $X\ssm U$.

{\bf Step IV:}
The fourth and final step deals with passing from the
unspecified $U$ to $\oap{X}{r}$ itself. The idea here is to
work with the irreducible 
components of $X\ssm U$ and apply induction on the dimension.
This step is carried out in Section \ref{sec:descent}.
The cost of this descent argument becomes apparent
when one compares the bounds in Theorem \ref{thm:genhb} and in our
main result, Theorem \ref{thm:effbhc}.

The author would like to thank Eric Katz for pointing me towards
 Sturmfels and Tevelev's work.
He is also grateful to Jenia Tevelev for answering tropically related
questions. 

\section{Notation}
\subsection{Generalities}
\label{sec:gennotation}

This section serves as a reference for general notation used throughout
the whole article. 

We use
 $\IN$ to denote the positive  integers $\{1,2,3,\ldots\}$ and $\IN_0$
for the non-negative integers $\IN\cup\{0\}$.

Let $R$ be a commutative unitary ring and 
$p=(p_1,\dots,p_n)$ a tuple of elements in $R$. If
$\alpha= (\alpha_1,\dots,\alpha_n)\in\IN_0^n$, we
set $p^{\alpha} = p_1^{\alpha_0}\cdots p_n^{\alpha_n}$
where our convention is $0^0=1$. 
For $k\in\IN_0$ we define $p^k = (p_1^k,\dots,p_n^k)$. 
If all  $p_i$ are in $R^\times$, the group of units of $R$, then
 we allow $\alpha$ to have negative entries and $k\in\IZ$.
It is convenient to set $|\alpha|_1 =  
\alpha_0+\cdots +\alpha_n$.
Say $I$ is an ideal of $R$.
If $R$ is graded by the integers,
 then $I_a$ will denote the homogeneous elements in $I$ with degree $a\in\IZ$.
If $R$ is bigraded by $\IZ^2$ then $I_{(a,b)}$ denotes the elements in
$I$
with bidegree $(a,b)\in\IZ^2$.

If $A$ is any matrix, then $\trans{A}$ is its transpose
and $\rk{A}$ is its rank.

We fix $\IQbar$ to be the algebraic closure of $\IQ$ in $\IC$, the
field of complex numbers. 

Throughout this paper, and if not stated elsewise,  $n,r\in\IN$ are fixed integers and 
$\mathbf{X} = (X_0,\dots,X_n),\mathbf{Y} = (Y_0,\dots,Y_r)$ 
are collections of independent variables. We often abbreviate
$\IQbar[X_0,\dots,X_n]$ by $\IQbar[\mathbf{X}]$ and $\IQbar[Y_0,\dots,Y_r]$ by
$\IQbar[\mathbf Y]$.  These rings have a natural grading
and $\IQbar[\mathbf X,\mathbf Y]$ has a natural bigrading.
In our notation we have $\mathbf{X}^\alpha \in
\IQbar[\mathbf{X}]_{|\alpha|_1}$
for $\alpha\in\IN_0^{n+1}$.

If $\alpha$ is as before, then the multinomial coefficient ${|\alpha|_1 \choose \alpha}$ is
$\frac{|\alpha|_1!}{\alpha_0!\cdots \alpha_n!}$.

  For a rational prime $p$ we let $|\cdot|_p$ denote the $p$-adic
absolute value on $\IQ$. Let $[x]$ denote the
greatest integer at most $x\in\IR$.

\begin{remark}
  Say $k=|\alpha|_1$, it is well-known that
\begin{alignat}1
\nonumber
  \log \left|{k \choose \alpha}\right|_p &=
-\left(\sum_{\atopx{e\ge 1}{p^e\le k}} \left[\frac{k}{p^e}\right]-
\left[\frac{\alpha_0}{p^e}\right]-\cdots
-\left[\frac{\alpha_n}{p^e}\right]\right)\log p 
\ge
-\left(\sum_{\atopx{e\ge 1}{p^e\le k}} {n+1}\right)\log p\\
\nonumber
&= -\left[\frac{\log k}{\log p}\right](n+1)\log p
\ge -(n+1)\log k.
\end{alignat}
Thus we obtain the  estimate
\begin{equation}
\label{eq:multinomial}
   \left|{k \choose \alpha}\right|_p \ge
\left\{
   \begin{array}{cl}
     k^{-(n+1)} &:\text{if }p\le k, \\
     1        &:\text{else wise.}
   \end{array}
\right.
\end{equation}
It will provide useful in many of our estimates.
\end{remark}

Let for the moment $K$ be an algebraically closed field contained in $\IC$.
We define an isomorphism of $K$-vector spaces
\begin{equation*}
\iota :K[\mathbf{X}]_a \rightarrow K^{{n+d}\choose{d}}
\end{equation*}
by setting $\iota({\bf X}^i) = {a \choose i}^{-1/2} e_i$
for $i\in\IN_0^{n+1}$ with $|i|_1=a$;
here $(e_i)_i$ is the standard basis of $K^{{n+d}\choose d}$
ordered lexicographically.
We get an analog
 isomorphism $\iota$ on $K[\mathbf Y]_{a}$. 
If $b\in\IN_0$, these  linear maps induce an isomorphism between
$K[\mathbf X,\mathbf Y]_{(a,b)}$ and the appropriate power
$K$. They also induce isomorphisms between finite direct sums
of $K[\mathbf X,\mathbf Y]_{(a,b)}$ and the appropriate power if
 $K$. By abuse of notation, all these isomorphisms
will be denoted by $\iota$.

Finally, we introduce some notation in connection with the algebraic
torus $\IG_m^n$.

A homomorphism $\IG_m^n\rightarrow\IG_m^r$ will always mean a
homomorphism of algebraic groups. In particular, it sends the unit
element to the unit element. Each homomorphism
$\IG_m^n\rightarrow\IG_m^r$ is uniquely determined 
by a matrix $\varphi\in \mat{r,n}{\IZ}$
with rows $\varphi_1,\dots,\varphi_r\in\IZ^n$. 
Indeed, the homomorphism is given as
$x\mapsto (x^{\varphi_1},\dots,x^{\varphi_r})$.
We will often use the same symbol $\varphi$ for the matrix and
corresponding homomorphism.
Any algebraic subgroup of $\IG_m^n$ with codimension $r$ 
is the kernel of some $\varphi\in \mat{r,n}{\IZ}$ of rank $r$,
cf.   Corollary 3.2.15 \cite{BG}.

We fix the open immersion $\IGm^n \hookrightarrow \IP^n$
given by $(p_1,\dots,p_n)\mapsto [1:p_1:\cdots:p_n]$.

\subsection{Heights of an Algebraic Point}
\label{sec:heights}

We now spend some time defining the height of algebraic numbers and
also algebraic points in projective space. 

 Let $K$ be a number field. We let $\pl{K}$ denote the set of 
 places of
$K$. Let $v\in\pl{K}$.
We will identify $v$ with an absolute value on $K$ which when
restricted to $\IQ$
is  $|\cdot|_p$ for some rational prime $p$ or the complex absolute
on $\IQ$. 
We call $v$ finite if it satisfies the ultrametric triangle inequality
and we call $v$ infinite if it does not.

If $v$ is a finite place
we let $K_v$ denote a completion of $K$ with respect to
$v$ and let $\sigma_v$ be the embedding $K\hookrightarrow K_v$.
We also fix $\IC_v$,  a completion of an algebraic closure of $K_v$,
together with an embedding $K\hookrightarrow \IC_v$, which by abuse of
notation we also call $\sigma_v$.
We write $|\cdot|_v$ for the $v$-adic
absolute value on $\IC_v$.

An infinite $v$ corresponds to an embedding of $K$ into $\IR$ or
$\IC$; in the latter case the embedding is only determined up-to
complex conjugation. We take $K_v=\IR$ or $K_v=\IC$ and $\sigma_v$,
accordingly. Moreover, we always set $\IC_v=\IC$ in the infinite
case. 
We sometimes abbreviate $|\cdot|=|\cdot|_v$ for the standard absolute value
on $\IC$;

By abuse of notation we let $\IQ_v$ denote $\IQ_{v|_\IQ}$.

  If $p = (p_0,\dots,p_n)\in \IC_v^{n+1}$, we set
\begin{equation*}
|p|_v = \left\{
\begin{array}{cl}
(\sum_{i=0}^n |p_i|^2)^{1/2} &: \text{ if $v$ is infinite and} \\
\max_{0\le i\le n} |p_i|_v &: \text{ if $v$ is finite.} \\
\end{array}
\right.
\end{equation*}

Now let us assume $p=(p_0,\dots,p_n)\in K^{n+1}\ssm\{0\}$. 
We define the height of $p$ as
\begin{equation*}
\height{p} = \frac{1}{[K:\IQ]} \sum_{v\in\pl{K}} [K_v:\IQ_v] \log
|p|_v.
\end{equation*}
 This height is
invariant under multiplication of $p$ by a non-zero scalar and under
a finite field extension of $K$. 
The proof of these statements is similar as the argument for heights with the
sup-norm at the infinite places given in Chapter 1.5 \cite{BG}. 
Hence we obtain a well-defined height function
\begin{equation*}
 \heightS:
\IP^{n}(\IQbar)\rightarrow [0,\infty). 
\end{equation*}

Using the open immersion 
from Section \ref{sec:heights}
we obtain a height function  $\IGm^n(\IQbar)\rightarrow [0,\infty)$;
for which we use the same letter $\heightS$.
In addition to this height function with $l^2$-norm at
infinite places we will use the height function 
\begin{equation*}
 \heightsS : \IG_m^n(\IQbar)\rightarrow [0,\infty) 
\end{equation*}
 with the sup-norm
at infinite places, cf. Chapter 1.5 \cite{BG}. This height has the
advantage that it plays along
nicely with the group law on $\IGm^n(\IQbar)$. For example, in the
case $n=1$ and  if $p,q\in
\IGm(\IQbar)$ 
then $\heights{p^k} = |k|\heights{p}$ for all $k\in\IZ$ 
and $\heights{pq}\le \heights{p}+\heights{q}$.
Both statements follow from the definition of $\heightsS$
and for the first one needs the product formula if $k<0$. 
If $p=(p_1,\dots,p_n)\in\IGm^n(\IQbar)$, then
\begin{equation}
\label{eq:heightcompare}
\max\{\heights{p_1},\ldots,\heights{p_n}\}\le 
\heights{p}\le \heights{p_1}+\cdots +\heights{p_n}
\end{equation}
is also a consequence of the definition.
We conclude that $\heights{p^u}\le n|u|_\infty \heights{p}$.
Finally, $\heights{p}\le\height{p}$ because the sup-norm is always at
most the $l^2$-norm and $\height{p}\le \frac 12\log(n+1) +
\heights{p}$
from a standard norm inequality.
Moreover, if $k\in\IN_0$ then $\heights{p^k} = k\heights{p}$. 
In the following, these inequalities and statements will be
refered to as basic height properties. 

Say $V$ is  a finite direct sum of 
various $\IQbar[{\bf X}]_a$ and $\IQbar[{\bf Y}]_b$ and
 $\IQbar[{\bf X},{\bf Y}]_{(a,b)}$ with $a,b\in\IZ$. Then
$\iota$ as described above defines an isomorphism of $V$ with some
power of $\IQbar$. If $P\in V$ 
we define $\height{P} = \height{\iota(P)}$. Moreover, if the coefficients of $\iota(P)$
are in $K$ we set $|P|_v = |\iota(P)|_v$ for any place $v$ of $K$. 
Then we obtain a height function, also denoted by $\heightS$, defined on non-zero
elements of $V$.

We also need a height function defined on $\mat{M,N}{\IQbar}$, 
the vector space of $M\times N$ matrices with algebraic coefficients.
Let $A\in \mat{M,N}{K}$ where $K$ is  a finite extension of $\IQ$.
If $1\le t\le \rk{A}$ is an integer and if $v$ is an infinite place
of $K$ we define the local height of $A$ as
\begin{equation*}
\lh{v,t}{A} = \frac 12 \log \sum_{A'\subset A} |\det A'|_v^2
\end{equation*}
where the sum runs over all $t\times t$ submatrices $A'$ of $A$. 
If $v$ is a finite place of $K$ we define the local height as
\begin{equation*}
\lh{v,t}{A} =  \log\max_{A'\subset A} |\det A'|_v
\end{equation*}
again  the $A'$  run over all $t\times t$ submatrices of $A$.
The height of $A$ is then
\begin{equation*}
\matheight{t}{A} = \frac{1}{[K:\IQ]} \sum_v [K_v:\IQ_v] \lh{v,t}{A}.
\end{equation*}

The following lemma provides a useful inequality concerning local heights of matrices.

\begin{lemma}
\label{lem:matrixlh}
Let $A$ and $K$ be as above and $B$  a matrix with $N$ rows and
coefficients in $K$. Then 
$\lh{v,t}{AB} \le \lh{v,t}{A}+\lh{v,t}{B}$ for all integers $1\le t\le
\rk{AB}$ and all places $v$ of $K$.
\end{lemma}
\begin{proof}
This is a direct consequence of the Cauchy-Binet formula together with
the Cauchy-Schwarz inequality for infinite $v$  and the ultrametric
triangle inequality 
for finite $v$.
\end{proof}

We proceed by defining the height of an $N$-dimensional vector subspace $V\subset
\IQbar^M$. Indeed, let $K$ be a finite extension of $\IQ$ and 
$A\in\mat{M,N}{K}$ a matrix whose column constitute a basis of $V$. We
define $\vsheight{V} = \matheight{N}{A}$. Then
$\vsheight{V}$ is independent of the choice of basis,
cf. Chapter 2.8 \cite{BG}. 

\begin{remark}
\label{rem:orthcomp}
This height behaves well with certain constructions typical for vector
spaces. If $V\subset\IQ^M$ and $W\subset \IQbar^{M'}$ are two vector
  subspaces.
Then the subspace $V\times W\subset\IQ^{M+M'}$ satisfies
\begin{equation}
\label{eq:heightprodineq}
 \vsheight{V\times W} \le \vsheight{V}+\vsheight{W} 
\end{equation}
by  Remark 2.8.9 \cite{BG}.

If $W\subset
\IQbar^M$, then by a theorem of Schmidt we have
\begin{equation}
\label{eq:schmidtineq}
 \vsheight{V\cap W}\le \vsheight{V}+\vsheight{W},
\end{equation}
cf. Theorem 2.8.13 \cite{BG}.

Finally, by Proposition 2.8.10 \cite{BG} the height of a vector space equals the height of its
orthogonal complement. In other words, 
\begin{equation}
\label{eq:orthcompl}
  \vsheight{\{(u_1,\dots,u_M);\,\, \sum_i u_i v_i = 0 
\text{ for all } (v_1,\dots,v_M)\in V\}} = \vsheight{V}.
\end{equation}
\end{remark}

Using the notation introduced up until now we are able to define the
height of a vector subspace 
$V$ of finite direct sums of 
various $\IQbar[{\bf X}]_a$ and $\IQbar[{\bf Y}]_b$ and
 $\IQbar[{\bf X},{\bf Y}]_{(a,b)}$.
Indeed, we set $\vsheight{V} = \vsheight{\iota(V)}$.


Let $X\subset\IP^n$ be an irreducible closed subvariety defined over
$\IQbar$ with homogeneous ideal $I \subset \IQbar[{\bf X}]$. 
The arithmetic and geometric
Hilbert functions of $X$ are defined as 
\begin{equation*}
\arithhilb{a;X}  = \vsheight{I_{a}}\quad\text{resp.}\quad
\hilbfunc{a;X}  = \dim \IQbar[{\bf X}]_a/{I_{a}},
\end{equation*}
respectively.

Let $Z\subset\IP^n\times\IP^m$ be an irreducible
closed subvariety defined over $\IQbar$ with bihomogeneous ideal $I \subset \IQbar[{\bf X},{\bf Y}]$.
The arithmetic and
geometric Hilbert functions of $Z$ are defined as
\begin{equation*}
\arithhilb{a,b;Z}  = \vsheight{I_{(a,b)}}
\quad\text{and}\quad
\hilbfunc{a,b;Z}  = \dim \IQbar[{\bf X},{\bf Y}]_{(a,b)}/{I_{(a,b)}},
\end{equation*}
respectively.

\subsection{Height of a  Projective Variety}
\label{sec:chowform}

We will briefly describe the definition of the height of an irreducible closed subvariety
$X\subset\IP^n$ defined over $\IQbar$
as it is given by
 Philippon \cite{HauteursAlt3}.

In this section we  use  ${\bf V}_i$ to denote an $N_i = {n+d_i
  \choose d_i}$-tuple of
independent
variables $U_{i0},\dots,U_{i,N_i-1}$.
Each packet ${\bf V}_i$  corresponds to coefficients of a homogeneous
polynomial in $n+1$ variables of degree $d_i$. 

Let $f\in K[{\bf V}_0,\dots,{\bf V}_r]\ssm\{0\}$
be homogeneous of degree $\degS_{{\bf V}_i}(f)$ in the variables ${\bf V}_i$ for
 $0\le i\le r$.

If $v$ is a finite place of $K$ we define $\mahler{v}{f}$ to be
the logarithm of maximum of the absolute values  of coefficients of
$f$
with respect to $v$. 

Say $v$ is an infinite place of $K$ and $\sigma=\sigma_v:K\rightarrow \IC$ a
corresponding embedding. We set
\begin{equation*}
  \mahler{v}{f} = \int_{S_{N_0}\times\cdots\times S_{N_r}}
\log |\sigma(f)(u_0,\dots,u_r)|
d\mu_0(u_0)\cdots d\mu_r(u_r)+
\sum_{i=0}^r \degS_{{\bf V}_i}(f) \sum_{j=1}^{N_i-1} \frac{1}{2j}
\end{equation*}
where $S_{N_i}$ is the unit circle in $\IC^{N_i}$ on which
$\mu_i$ is the  measure, invariant under the unitary group, of
total mass $1$.

\begin{remark}
 It is known that
\begin{equation}
\label{eq:heightnonneg}
\sum_{v \in \pl{K}}
[K_v:\IQ_v]\mahler{v}{f} \ge 0.  
\end{equation}
Let us indicate the proof.

The left-hand side of (\ref{eq:heightnonneg}) remains unchanged when replacing $K$ by a larger
number field. By the product formula, cf. Chapter 1.4 \cite{BG}, it also
 remains unchanged when replacing $f$ by a
 multiple with factor in $K^\times$. After
replacing $K$ by a larger number field we may find
$\lambda\in K^\times$ such that 
$\log |\lambda|_v + \mahler{v}{f}=0$ for all finite places $v$. 
We replace $f$ by $\lambda f$, which has coefficients in the ring of
integers of $K$. It suffices to show that 
$\sum_{v \text{ infinite}}
[K_v:\IQ_v]\mahler{v}{f} \ge 0$. So say $v$ is an infinite place of $K$.
By Lelong's Th\'eor\`eme 4 \cite{Lelong:Mahler} the value
$\mahler{v}{f}$ is at least the logarithmic Mahler measure of $\sigma_v(f)$. 
Because the Mahler measure is multiplicative,  $\sum_{v \text{ infinite}}
[K_v:\IQ_v]\mahler{v}{f}$
is at least the logarithmic Mahler measure of
$\prod_{v\text{ infinite}}\sigma_v(f)^{[K_v:\IQ_v]}$. This is a non-zero polynomial in
integer coefficients.  
But the logarithmic Mahler measure of a non-zero polynomial in integer
coefficients is non-negative.
Our claim (\ref{eq:heightnonneg}) follows.
\end{remark}

Let $r= \dim X$ and let $K$ be a number field over which $X$ is
defined. Say  $d=(d_0,\dots,d_r) \in \IN^{r+1}$.

We let $f_{X,d}$ denote the elimination or Chow
form of $X$ associated to $d$. For a more proper survey on Chow forms
and their the properties stated below we refer to Philippon's article
\cite{Philippon:DG}.
It is uniquely defined up-to multiplication by a scalar in $K^\times$.
The Chow form is 
homogeneous in each ${\bf V}_i$ of degree $\deg{X}d_0\cdots d_r / d_i$.

Now we are ready to define the height of $X$. It is
\begin{equation}
\label{eq:definehX}
  \height{X} = \frac{1}{d_0\cdots d_r}
\sum_{v \in \pl{K}}
\frac{[K_v:\IQ_v]}{[K:\IQ]}\mahler{v}{f_{X,d}}.
\end{equation}
This value is independent of $d_0,\dots,d_r$, the choice of number
field $K$, and 
the choice of $f_{X,d}$.
This height is sometimes called the
 Faltings height. 

From (\ref{eq:heightnonneg}) we deduce $\height{X}\ge 0$.

Recall that we have fixed an open immersion $\IG_m^n\hookrightarrow \IP^n$.
If $X\subset \IGm^n$ is an irreducible closed subvariety defined over $\IQbar$, then
$\deg{X}$ and $\height{X}$ are defined to be the degree and height of the
Zariski closure of $X$ in $\IP^n$, respectively. Similarly, the
height of a point in $\IGm^n(\IQbar)$ is the height of its image
in $\IP^n(\IQbar)$.

\section{Estimates for Hilbert Functions}
\label{sec:hilbestimates}

\subsection{The Geometric Hilbert Function}
\label{sec:hilbestimates}

Let $Z$ be an irreducible closed subvariety of $\IP^n\times\IP^r$.
The purpose of this section is to estimate $\hilbfunc{a,b;Z}$
explicitly in  terms of $\hhilbpoly{a,b;Z}$. 
The upper bound for the geometric Hilbert function
in terms of the Hilbert polynomial  given in Lemma \ref{lem:hilbfunccurve} 
follows  arguments of Bertrand and Koll\'ar 
 \cite{Bertrand:Hilbert}.

It is known that if $\min\{a,b\}$ is large enough, then
 $\hilbfunc{a,b;Z}$
is the value at $(a,b)$ of a polynomial depending only on $Z$.
 This uniquely determined polynomial is called the
Hilbert polynomial of $Z$. 
We define $\hhilbpoly{T_1,T_2;Z} \in \IQ[T_1,T_2]$ to be $(\dim Z)!$ times
the highest degree homogeneous part of the Hilbert polynomial. 
Section 3 \cite{Philippon} contains a treatment of Hilbert
polynomials.

An treatment of the intersection theory needed in this
article can be found in Fulton's book \cite{Fulton}. 

Let $\pi_1:\IP^n\times\IP^r \rightarrow \IP^{n}$ and 
$\pi_2:\IP^n\times\IP^r \rightarrow \IP^r$
denote the projection onto the two factors.
 Let $\bigO{1}$ denote the dual of the
tautological line bundle on any projective space. 
Then
\begin{equation}
\label{eq:defhilbpoly}
\hhilbpoly{T_1,T_2;Z} 
= \sum_{i=0}^{\dim Z}
{{\dim Z}\choose {i}} (\pi_1^*\bigO{1}^{i}\pi_2^*\bigO{1}^{\dim
  Z-i}[Z]) T_1^{i} T_2^{\dim Z -i}. 
\end{equation}
where $[Z]$ is the class
of cycles on $\IP^n\times\IP^r$ which are linearly equivalent to $Z$.

Let ${\bf U}$ be an $(n+1)(r+1)$-tuple of independent variables
$(U_{ij})$ where $0\le i\le n$ and $0\le j\le r$. 
For any field $K$ we define the Segre homomorphism
\begin{equation*}
s^*: K[{\bf U}] \rightarrow K[{\bf X},{\bf Y}]
\end{equation*}
 by setting $s^*(U_{ij}) = X_i Y_j$. By abuse of notation  we also let
$s^*$ denote  the restriction
$K[{\bf U}]_k \rightarrow K[{\bf X},{\bf
    Y}]_{(k,k)}$ for all  $k\in\IN_0$.

The Segre homomorphism  defines a morphism $s :
\IP^n\times \IP^{r}\rightarrow
\IP^{nr+n+r}$ of varieties.
We define
\begin{equation*}
  \deg{Z} = \deg{s(Z)}
\quad\text{and}\quad
\height{Z}= \height{s(Z)}\quad \text{if $Z$ is defined over $\IQbar$}
\end{equation*}
where the degree of a subvariety of projective space is the degree of
the cycle class obtained by intersecting the variety  the
appropriate number of times with  $\bigO{1}$.

By abuse of notation, $s^*$ will also denote the induced
homomorphism
between the Picard groups of $\IP^{nr+n+r}$ and
$\IP^n\times\IP^r$.

Next we define the Veronese maps. Let $a\in\IN$ and let
${\bf\widetilde X}$ be the ${n+a}\choose {a}$-tuple 
$(\widetilde X_\gamma)$ where $\gamma$ runs through all vectors in $\IN_0^{n+1}$
with $|\gamma|_1 = a$. We define the twisted Veronese homomorphism 
\begin{equation*}
v^*_{a} : K[{\bf\widetilde X}] \rightarrow K[{\bf X}]
\end{equation*}
by setting $v^*_{a}(\widetilde X_\gamma) = {a\choose\gamma}^{1/2} {\bf
  X}^{\gamma}$; of course, $v^*_a$ is only defined if $K$ contains
all roots ${a \choose \gamma}^{1/2}$. The choice of the square root will be irrelevant
in our applications.  
By abuse of notation  we also let $v^*_a$ denote the
restriction
$K[{\bf\widetilde X}]_{k} \rightarrow K[{\bf X}]_{ak}$. 

The Vernose homomorphism
$v^*_{a}$  induces a closed immersion 
$v_{a}:\IP^n\rightarrow
\IP^{{n+a\choose a}-1}$ of varieties.
If $b\in\IN$, then the product 
$v_{ab}:\IP^n\times\IP^r \rightarrow
\IP^{{n+a\choose a}-1}\times
\IP^{{r+b\choose b}-1}$ is also a closed immersion.

By abuse of notation, $v_{a}^*$ and $v_{ab}^*$ will also denote
the induced homomorphisms between Picard groups of 
$\IP^{{n+a\choose a}-1}\times
\IP^{{r+b\choose b}-1}$ and of $\IP^n\times\IP^r$.

\begin{lemma}
\label{lem:hilbpolyboundsegre}
  We have $\hhilbpoly{1,1;Z} = \deg{Z}$ and 
    $\hhilbpoly{a,b;Z} \le \max\{a,b\}^{\dim Z}\deg{Z}$
for $a,b\in\IN$.
\end{lemma}
\begin{proof}
By definition we have
$\deg{s(Z)} = (\bigO{1}^{\dim Z}[s(Z)])$.
The projection formula, and the fact that $s$ is a closed embedding
imply
$\deg{s(Z)} = (s^{*}\bigO{1}^{\dim Z}[Z])$.
Since $s^{*}\bigO{1} = \pi_1^*\bigO{1}\otimes\pi_2^*\bigO{1}$ we
deduce
\begin{equation*}
  \deg{s(Z)} = \sum_{i+j=\dim Z}{\dim Z\choose i}
(\pi_1^*\bigO{1}^{i}\pi_2^*\bigO{1}^{j}[Z]) = \hhilbpoly{1,1;Z}.
\end{equation*}
The first statement of the lemma follows.

The second statement follows from the first one and 
 $(\pi_1^*\bigO{1}^{i}\pi_2^*\bigO{1}^{j}[Z])\ge 0$.
\end{proof}

\begin{lemma}
\label{lem:hilbpolyboundvero}
  We have $\hhilbpoly{1,1;v_{ab}(Z)} = \hhilbpoly{a,b;Z}$
for $a,b\in\IN$.
\end{lemma}
\begin{proof}
For this lemma  we write $\pi'_{1}$ and $\pi'_2$ for the first and
second  projection
on $\IP^{{n + a\choose a}-1}\times\IP^{{r + b \choose b}-1}$, respectively.
By definition and since $[v_{ab}(Z)]={v_{ab}}_*([Z])$ we have
\begin{alignat*}1
  \hhilbpoly{1,1;v_{ab}(Z)} &= \sum_{i+j = \dim Z}
{\dim Z \choose i} ({\pi'_1}^* \bigO{1}^{i} {\pi'_2}^* \bigO{1}^j
{v_{ab}}_{*}[Z]) \\
&= \sum_{i+j = \dim Z}
{\dim Z \choose i} (v_{ab}^*{\pi'_1}^* \bigO{1}^{i} v_{ab}^*{\pi'_2}^*
\bigO{1}^j [Z])
\end{alignat*}
where we used the projection formula.
We note $\pi'_1\circ v_{ab} = v_a \circ \pi_1$ and
 $v_{ab}^*{\pi'_1}^* \bigO{1} = (\pi'_1\circ v_{ab})^* \bigO{1}
=(v_a\circ\pi_1)^*\bigO{1}
= \pi_1^* v_a^*\bigO{1}
= \pi_1^*\bigO{1}^{\otimes a}$.
Similarly, $v_{ab}^*{\pi'_2}^* \bigO{1} = \pi_2^*\bigO{1}^{\otimes b}$. 
The lemma follows.
\end{proof}

\begin{lemma}
\label{lem:hilbfunccurve}
Assume $Z$ is a curve, then
$\hilbfunc{a,b;Z} \le 1+ \hhilbpoly{a,b;Z}$
for $a,b\in\IN$.
\end{lemma}
\begin{proof}
 Lemma 1.1 \cite{Bertrand:Hilbert}
implies $\hilbfunc{a,b;Z}-1\le\deg{s(v_{ab}(Z))}$. 
So $\hilbfunc{a,b;Z} \le 1 + \hhilbpoly{1,1;v_{ab}(Z)}$
by Lemma \ref{lem:hilbpolyboundsegre}. 
We conclude the current lemma by referring to Lemma \ref{lem:hilbpolyboundvero}.
\end{proof}

We now generalize this bound to higher dimension using a Bertini-type
argument. 

\begin{lemma}
\label{lem:hilbfuncub}
Say $\dim Z\ge 1$. We have
\begin{equation*}
\hilbfunc{ak,bk;Z} \le \hhilbpoly{a,b;Z} {\dim Z+k \choose \dim Z}
\end{equation*}
for $a,b,k\in\IN$.
\end{lemma}
\begin{proof}
We prove the lemma by induction on $d=\dim Z$. 
If $d = 1$, then Lemma \ref{lem:hilbfunccurve} leads to 
$\hilbfunc{ak,bk;Z}\le 1 + \hhilbpoly{ak,bk;Z} = 
1+\hhilbpoly{a,b;Z}k$ because the Hilbert polynomial is homogeneous of
degree $d$.
 But $\hhilbpoly{a,b;Z}\ge 1$ because it is a
positive integer by (\ref{eq:defhilbpoly}). Hence
$\hilbfunc{ak,bk;Z}\le \hhilbpoly{a,b;Z}(1+k) = \hhilbpoly{a,b;Z}{1+k
  \choose 1}$ as desired.  So let us assume $d \ge 2$.

Let $I\subset R= \IQbar[{\bf X},{\bf Y}]$ be the
bihomogeneous  prime ideal of $Z$.
By Bertini's Theorem, c.f. Corollaire 6.11 (2) and (3), page 89 \cite{Jouanolou}
there is $F \in R_{(a,b)} \ssm
I_{(a,b)}$ such that $I+F\cdot R$ is again a prime ideal of a variety
$Z'\subset \IP^n\times\IP^r$ of dimension $d -1$.
If $k\ge 1$, then 
\begin{equation*}
  0\rightarrow R_{(a(k-1),b(k-1))}/I_{(a(k-1),b(k-1))}\rightarrow
   R_{(ak,bk)}/I_{(ak,bk)}\rightarrow R_{(ak,bk)} / (I+F\cdot
   R)_{(ak,bk)}\rightarrow 0
\end{equation*}
is an exact sequence, the second arrow being multiplication by $F$.
So $\hilbfunc{ak,bk;Z} = \hilbfunc{a(k-1),b(k-1);Z} + 
\hilbfunc{ak,bk;Z'}$. 
By induction on $k$ we have
\begin{equation*}
 \hilbfunc{ak,bk;Z} = \sum_{j=0}^k \hilbfunc{aj,bj;Z'}. 
\end{equation*}

We have $\dim Z'= d -1$ and by induction
$\hilbfunc{aj,bj;Z'}\le \hhilbpoly{a,b;Z'}{d-1+j\choose d-1}$ if
$j\not=0$.
This upper bound also holds for $j=0$ since $\hilbfunc{0,0;Z'}=1$. So
\begin{equation*}
  \hilbfunc{ak,bk;Z}\le  \hhilbpoly{a,b;Z'} \sum_{j=0}^k {d-1+j\choose d-1}
= \hhilbpoly{a,b;Z'}{d+k\choose d}
\end{equation*}
by properties to binomial coefficients.

Philippon's Lemme 3.1 \cite{Philippon} implies
$\hhilbpoly{a,b;Z'} = \hhilbpoly{a,b;Z}$ and this completes the proof.
\end{proof}

Now we come to lower bounds. 

\begin{lemma}
\label{lem:hilbfunclb}
Let $\Delta=(\pi_2^*\bigO{1}^{\dim Z}[Z])$.
We have
\begin{equation*}
\hilbfunc{a,b;Z} \ge 
\Delta 
 {{\dim Z+b-\Delta}\choose{\dim Z}}
\end{equation*}
for all $a,b\in \IN$ with $a\ge \Delta$ and $b\ge\Delta$.
\end{lemma}
\begin{proof}
Since $\Delta$ is a non-negative integer we may assume $\Delta\ge 1$.
For brevity, we write $d=\dim Z$.
By Bertini's Theorem  we  find linear forms
 $l_{0},\dots,l_{d}\in\IQbar[{\bf Y}]$ which define a rational
map $[l_{0}:\cdots:l_{d}]:\IP^n\times\IP^r\rightarrow
\IP^d$ whose restriction $\Psi$ to $Z$ is a
morphism
of degree $\Delta$.
We may suppose that the form $l_{0}$ does not vanish identically on $Z$.
 The morphism $\Psi$
induces a extension of function fields
$\IQbar(Z)/\Psi^*\IQbar(\IP^d)$ of degree
$\Delta$.

We shall assume that the projective coordinates $X_0$ and $Y_0$ do
 not vanish identically on
$Z$. Otherwise, the argument given below goes through when working with some
other
pair $X_i$ and $Y_j$.

We  apply the Primitive Element Theorem to
 find a rational linear combination $\gamma$
of $X_1/X_0,\dots,X_n/X_0,Y_1/Y_0,\dots,Y_r/Y_0$
which, when considered as an element of $\IQbar(Z)$, generates
the field extension
$\IQbar(Z)/\Psi^*\IQbar(\IP^d)$.

We set
 $b' = b-\Delta \ge 0$
and let $P_0,\dots,P_{\Delta-1}\in\IQbar[Y_0,\dots,Y_j]_{b'}$.
Let us consider
\begin{equation}
\label{eq:bihomopoly}
\left(  \sum_{k=0}^{\Delta-1} \gamma^k
  P_{k}\left(1,\frac{l_{1}}{l_{0}},\dots,\frac{l_{d}}{l_{0}}\right)\right)
X_0^a Y_0^\Delta l_{0}^{b'}
\end{equation}
as an element of $\IQbar[{\bf X},{\bf Y}]_{(a,\Delta+b')}$. But the
sum in brackets
is  a rational function on $Z$. Since $\gamma$ has degree $\Delta$
over $\Psi^*\IQbar(\IP^i\times\IP^j)$ we see that
(\ref{eq:bihomopoly}) vanishes on $Z$ if and only if all $P_k$ are zero.
The current lemma follows from this
and
$\dim \IQbar[Y_0,\dots,Y_d]_{b'} = 
{d+b-\Delta\choose d}$.
\end{proof}

\subsection{The Arithmetic Hilbert Function}

In this subsection we bound from above the arithmetic Hilbert function of
an irreducible closed subvariety $X\subset\IP^n$ defined over
$\IQbar$.

We begin by citing a  result of David and Philippon \cite{MinorSousVarieties}.

Say $V$ is a subvariety of projective space defined over a
 field $K$. If $\sigma:K\rightarrow L$ is an embedding of $K$ into
 another field $L$, then $V_\sigma$ denotes the induced subvariety of projective
 space defined over $L$.

\begin{lemma}
\label{lem:DP}
  Let $V$ be a non-trival vector subspace of $\IQbar^N$ and
 let $\widetilde V \subset \IP^{N-1}$ be the set of lines in $V$.
 If $\epsilon > 0$ there is a linear form $l$
on $\IQbar^N$
which does not vanish completely on $V$ such that
\begin{equation*}
  \sum_{v\in \pl{K}} \frac{[K_v:\IQ_v]}{[K:\IQ]}
\log \sup_{p\in \widetilde V_{\sigma_v}(\IC_v)}
\frac{|\sigma(l)(p)|_v}{|p|_v}
\le -  \frac{\vsheight{V}}{\dim V} + \frac{1}{\dim V} \sum_{i=1}^{\dim
  V-1 }\sum_{j=1}^i \frac{1}{2j} + \epsilon
\end{equation*}
where $K$ is a number field containing the coefficients of $l$.
We remark that the quotient $|\sigma(l)(p)|_v / |p|_v$ is well-defined
for $p\in \IP^{N-1}(\IC_v)$.
\end{lemma}
\begin{proof}
  This is Corollaire 4.9 \cite{MinorSousVarieties} in the case
  $R=\IQbar$.
\end{proof}

We now bound the arithmetic Hilbert function from above explicitly in
terms of the height and degree of a variety. We essentially follow the
argumentation given by
 Philippon in the proof of Th\'eor\`eme 7  \cite{Philippon:Approx}.
 Our inequality is
absolute in the sense that the bound is independent of a field 
of definition.

\begin{proposition}
\label{prop:arithhilbbound}
  If $k \in \IN$, then
  \begin{equation*}
    \arithhilb{k;X}\le 
\hilbfunc{k;X}\left(k \frac{\height{X}}{\deg{X}}
+ \frac{1}{2} \log\hilbfunc{k;X}\right).
  \end{equation*}
\end{proposition}
\begin{proof}
  By definition we have $\arithhilb{k;X} =
  \vsheight{I_k}$
where $I\subset \IQbar[X_0,\dots,X_n]$ is the ideal of $X$.

Say $N = {n+k\choose k}$, this is the  maximal number of 
 monomials in a
homogeneous polynomial of degree $k$ with $n+1$ variables. 
We will apply Lemma \ref{lem:DP} to 
\begin{equation*}
  V = \{v\in\IQbar^N;\,\,  \trans{v}\cdot  \iota(P)  = 0 \quad\text{for
    all}\quad P\in I_k \} \subset \IQbar^N.
\end{equation*}
Equality (\ref{eq:orthcompl}) implies that the height of $V$ is
$\arithhilb{k;X}$. We have $\dim V = \hilbfunc{k;X} \ge 1$.

Say $\epsilon >0$.
By Lemma \ref{lem:DP} there is a linear form $l$, which we may
identify with an element of $\IQbar^N$, that
 does not vanish identically on $V$
 with
\begin{equation}
\label{eq:DPbound}
 \sum_{v\in \pl{K}} \frac{[K_v:\IQ_v]}{[K:\IQ]}
\log \sup_{p\in \widetilde V_{\sigma_v}(\IC_v)}
\frac{|\sigma_v(l)(p)|_v}{|p|_v}
\le -  \frac{\arithhilb{k;X}}{\hilbfunc{k;X}} +
\frac{1}{\hilbfunc{k;X}} 
\sum_{i=1}^{\hilbfunc{k;X}-1}\sum_{j=1}^i \frac{1}{2j} + \epsilon
\end{equation}
for a  number field $K\subset\IQbar$
 containing the coefficients of $l$ and all of the finitely many
 algebraic numbers appearing in this proof.

There is a homogeneous $P\in \IQbar[{\bf X}]$ of
degree  $k$ with $l=\iota(P)$. 
We write $P=\sum_{|\lambda|_1=k} P_\lambda{\bf X}^\lambda$.

Let  $f = f_{X,(1,\dots,1,k)} \in K[{\bf V}_0,\dots,{\bf V}_r]$ be a Chow form of $X$.
In the notation of Section \ref{sec:chowform}, ${\bf V}_r$ is an
$N$-tuple of variables corresponding each to a monomial of degree $k$
in $n+1$ variables.
 Let $\rho(f) \in K[{\bf V}_0,\dots,{\bf V}_{r-1}]$
be the form obtained by specializing the variables ${\bf V_r}$ to
equal the corresponding coefficients of $P$. 

Let $v\in\pl{K}$  and say $\sigma = \sigma_v:K\rightarrow \IC_v$
is the embedding choosen in Subsection \ref{sec:heights}.

We assume first that $v$ is infinite. 
By Lemma 4.1 \cite{Philippon:DG} 
there is a measure $\Omega$ of total mass
$\deg{X}$ on $X_\sigma (\IC)$ with
\begin{equation}
\label{eq:mahlerdiff}
  \mahler{v}{\rho(f)} - \mahler{v}{f} = 
\int_{X_\sigma(\IC)} \log \frac{|\sigma(P)(p)|}{|p|^k}
\Omega(p). 
\end{equation}

Let $p=(p_0,\dots,p_n)\in \IC^{n+1}$ such that $[p_0:\cdots:p_n]\in
X_\sigma(\IC)$.
For $\lambda\in \IN_0^{n+1}$ with $|\lambda|_1 = k$ 
we set $q_\lambda = \sigma({k\choose \lambda}^{1/2}) p^\lambda$
and  $q=(q_\lambda)_\lambda \in \IC^N$.
Then $q \in \sigma(V)$. 
Moreover, we have
\begin{equation*}
\sigma(P)(p)
= \sum_{|\lambda|_1=k} \sigma(P_\lambda)p^\lambda=
\sum_{|\lambda|_1=k} \sigma(P_\lambda)
\sigma\left(
{k \choose \lambda}^{-1/2}\right)q_\lambda=
 \sigma(\iota(P))\cdot \trans{q } =
\sigma(l)(q).
\end{equation*}
We evaluate
\begin{equation*}
|p|^{2k}=
\sum_{|\lambda|_1=k} {k\choose \lambda}
|p^\lambda|^2= \sum_{|\lambda|_1=k} |q_\lambda|^2
=|q|^2
\end{equation*}
and thus conclude
\begin{equation*}
\frac{|\sigma(P)(p)|}{|p|^k}=
  \frac{|\sigma(l)(q)|}{|q|}
\le 
 \sup_{q'\in \sigma(\widetilde V)(\IC) }
\frac{|\sigma(l)(q')|}{|q'|}
\end{equation*}
Using (\ref{eq:mahlerdiff}) we obtain
\begin{equation}
\label{eq:mahlerinfinite}
  \mahler{v}{\rho(f)} - \mahler{v}{f} \le
\deg{X}\log \sup_{q'\in \sigma(\widetilde V)( \IC)}
\frac{|\sigma(l)(q')|}{|q'|}.
\end{equation}

Now we assume that $v$ is finite.
By Lemme 2 \cite{HauteursAlt2} we have 
\begin{equation}
 \label{eq:mahlerdiff2}
\mahler{v}{\rho(f)} - \mahler{v}{f} = \int_{X_{\sigma}(\IC_v)} 
\log \frac{|\sigma(P)(p)|_v}{|p|_v^k} \Omega_v(p),
\end{equation}
with a measure $\Omega_v$ on $X_{\sigma}(\IC_v)$ of mass  
$\deg{X}$.

Say $p=(p_0,\cdots,p_n)\in\IC_v^{n+1}$ with $[p_0:\cdots:p_n]\in
X_\sigma(\IC_v)$.
As in the case where $v$ is infinite we have
$\sigma(P)(p) = \sigma(l)(q)$
where $q\in\IC_v^{N}$ is defined in a similar
manner.
This time $|q|_v = \max_\lambda \{|q_\lambda|_v\} \le \max_\lambda\{|p^\lambda|_v\}
= \max\{|p_0|_v,\dots,|p_n|_v\}^k = |p|_v^k$, so
\begin{equation*}
  \frac{|\sigma(P)(p)|_v}{|p|_v^k}
\le \frac{|\sigma(l)(q)|_v}{|q|_v}.
\end{equation*}
By (\ref{eq:mahlerdiff2}) we deduce
\begin{equation}
  \label{eq:mahlerfinite}
\mahler{v}{\rho(f)} - \mahler{v}{f} \le \deg{X} 
\log \sup_{y\in \widetilde{V}(\IC_v)}
\frac{|\sigma(l)(y)|_v}{|y|_v}.
\end{equation}

We multiply (\ref{eq:mahlerinfinite}) and (\ref{eq:mahlerfinite}) 
with $[K_v:\IQ_v]/[K:\IQ]$ and the take the sum over all $v\in\pl{K}$.
Using the definition of $\height{X}$ given in (\ref{eq:definehX}) we obtain
\begin{equation*}
-  k \height{X} + \sum_{v\in\pl{K}} \frac{[K_v:\IQ_v]}{[K:\IQ]}
\mahler{v}{\rho(f)} \le
\deg{X} \sum_{v\in\pl{K}} \frac{[K_v:\IQ_v]}{[K:\IQ]}
\log \sup_{q\in \sigma_v(\widetilde V)(\IC_v)}
\frac{|\sigma(l)(q)|_v}{|q|_v}.
\end{equation*}
Recall  that $ \sum_{v\in\pl{K}} \frac{[K_v:\IQ_v]}{[K:\IQ]}
\mahler{v}{\rho(f)}\ge 0$ by (\ref{eq:heightnonneg}). 
Using (\ref{eq:DPbound}) we bound the 
 right-hand side  to get
 \begin{equation}
\label{eq:almostdone}
\deg{X}\frac{\arithhilb{k;X}}{\hilbfunc{k;X}}
\le k\height{X} + \left(\frac{1}{\hilbfunc{k;X}}
\sum_{i=1}^{\hilbfunc{k;X}-1}\sum_{j=1}^i \frac{1}{2j}+\epsilon\right)\deg{X}.
 \end{equation}

If $T\ge 1$ is an integer, then elementary inequalities lead to 
\begin{equation*}
  \sum_{i=1}^{T-1} \sum_{j=1}^i \frac{1}{j} 
\le \sum_{i=1}^{T-1} (1+\log i)
\le T - 1 + \int_{1}^{T} \log(i)di
 = T\log T.
\end{equation*}
From (\ref{eq:almostdone}) we conclude
\begin{equation*}
  \arithhilb{k;X} \le k\hilbfunc{k;X} \frac{\height{X}}{\deg{X}}
+ \frac{1}{2} \hilbfunc{k;X}\log\hilbfunc{k;X}+\epsilon\hilbfunc{k;X}.
\end{equation*}
The proposition follows since $\epsilon > 0$ was arbitrary.
\end{proof}

\section{More on Segre and Veronese}
\label{sec:sav}

This section is on height inequalities in connection with Segre
and Veronese morphisms. Both were defined in Section \ref{sec:hilbestimates}.

Let $K$ be a field equipped with an automorphism $\tau : K\rightarrow
K$ of order at most $2$. 
Let $V$ and be a finite dimensional vector space  over $K$.
A function $\langle \cdot,\cdot\rangle: V\times V\rightarrow K$
is called a $\tau$-inner product, or short inner product,
if it satisfies the following three properties. 
\begin{enumerate}
  \item[(i)]  It is
$K$-linear in its first variable.
\item[(ii)] 
 We have $\langle v,w\rangle = \tau(\langle
w,v\rangle)$ for all $w,v\in V$.
\item[(iii)] If $v\in V$ and
 $\langle v,w\rangle=0$ for all $w\in V$, then $v=0$.
\end{enumerate}

If $W$ is a vector subspace of $V$, then $\orth{W}$ is the orthogonal
complement of $W$, i.e.
$ \orth{W} = \{v\in V;\,\, \langle v,w\rangle=0 \text{ for all }w\in
W\}$.
Property (iii) in the definition implies
$\dim W + \dim \orth{W} = \dim V$.

Now assume $W$ is a second finite dimensional
 vector space over $K$
with a $\tau$-inner product $\langle\cdot,\cdot\rangle'$.
We call a linear map $f:V\rightarrow W$ an isometry if 
\begin{equation}
\label{def:isometry}
\langle f(v),f(v') \rangle' = 
\langle v,v' \rangle
\quad\text{for all}\quad v \in \orth{(\ker f)} \quad\text{and}\quad v'\in V.
\end{equation}
If $f$ is injective, then our notion of isometry is the usual one. 

\begin{remark}
\label{rem:isokernel}
We claim that $\ker f \cap \orth{(\ker f)}=0$. Indeed, if $v$ lies in
this intersection, then $0 = \langle 0,f(v')\rangle = \langle f(v),f(v') \rangle' = \langle
v,v'\rangle $ for all $v'\in V$. So $v=0$ because 
of condition (iii) in the definition of $\langle\cdot,\cdot\rangle$.   
\end{remark}

On $V\otimes W$ we have a natural $\tau$-inner product determined by
$\langle v\otimes v',w\otimes w'\rangle'' = 
\langle v,w \rangle
\langle  v', w'\rangle'$.

To ease notation we sometimes use $\langle \cdot,\cdot\rangle$ to
denote inner-products on multiple vector spaces. 

The next lemma is a simple application of linear algebra.

\begin{lemma}
\label{lem:tensoriso}
Let $K$ and $\tau$ be as above and
assume $V,W,V',$ and $W'$ are finite dimensional
vector spaces over $K$, each one equipped with a $\tau$-inner product. 
If $f:V\rightarrow W$ and
$f':V'\rightarrow W'$ are isometries, then so is $f\otimes f':
V\otimes V'\rightarrow W\otimes W'$. 
\end{lemma}
\begin{proof}
We write $\langle\cdot,\cdot\rangle$ for any inner product in this proof.
  The lemma formally follows from
$\orth{(\ker f)}\otimes \orth{(\ker f')} = \orth{(\ker f\otimes f')}$
which we now show.
To prove the inclusion ``$\subset$'' we let $v\in\orth{(\ker f)}$,
 $v'\in\orth{(\ker f')},$ and $u=\sum_{i} u_i\otimes u'_i\in
\ker{f\otimes f'}$. 
So
$\langle v\otimes v',u\rangle = \sum_i \langle v\otimes v',u_i\otimes
u'_i\rangle = \sum_i \langle v,u_i\rangle\langle v',u'_i\rangle$. But
$f$ and $f'$ are isometries, therefore
$\langle v\otimes v',u\rangle = \sum_i \langle f(v)\otimes
f'(v'),f(u_i)\otimes f'(u'_i)\rangle = \langle f(v)\otimes
f'(v'),(f\otimes f')(u)\rangle = 0$. It follows that $v\otimes v'\in
\orth{(\ker f\otimes f')}$ since $u$ was arbitrary. The desired
inclusion holds. 

Now $\dim \orth{(\ker f)}\otimes \orth{(\ker f')} = (\dim V - \dim
\ker f)(\dim V' - \dim \ker f')$.
Moreover, 
$\dim \orth{(\ker f\otimes f')} = \dim V\otimes V' - \dim \ker f\otimes
f = \dim f(V)\otimes f'(V') =\dim f(V) \dim f'(V')$.
We conclude that $\orth{(\ker f)}\otimes \orth{(\ker f')}$
and $\orth{(\ker f\otimes f')}$ have equal dimension, so they coincide.
\end{proof}

\begin{remark}
\label{rem:defineip}
Let $K$ be a field of characteristic $0$ 
together with 
 an involution $\tau$. We shall assume that if $k\in\IN$ and  $\sqrt{k}\in K$
 then $\tau(\sqrt{k})= \sqrt{k}$. 
We fix $a\in\IN_0$.
An important example of a  $\tau$-inner product   on $K[{\bf
  X}]_a$ is defined in the following manner. 
If $P=\sum_\gamma P_\gamma {\bf X}^\gamma$ 
and $Q=\sum_\gamma Q_\gamma {\bf X}^{\gamma}$
where $\gamma$ runs over elements in $\IN_0^{n+1}$ with
$|\gamma|_1=a$, then 
\begin{equation*}
  \langle P,Q\rangle = \sum_\gamma {a\choose \gamma}^{-1} 
P_\gamma \tau(Q_\gamma) \in K.
\end{equation*}

If $b\in\IN_0$, then
we may identify $K[{\bf X},{\bf Y}]_{(a,b)}$
with $K[{\bf X}]_a\otimes_K K[{\bf Y}]_{b}$. Thus
 we obtain  a 
 $\tau$-inner product on $K[{\bf X},{\bf Y}]_{(a,b)}$.
\end{remark}

In the next two lemmas, 
$K$ is an algebraically closed field of characteristic $0$, $\tau$ is an involution on
$K$, and $\langle\cdot,\cdot\rangle$ is the $\tau$-inner product as given
by Remark \ref{rem:defineip}.

We recall that  Segre and Veronese homomorphisms were defined in Subsection 
\ref{sec:hilbestimates}.

\begin{lemma}
\label{lem:segreisometry}
The Segre homomorphism $s^*:K[{\bf U}]_k \rightarrow K[{\bf X},{\bf
    Y}]_{(k,k)}$ 
is a surjective isometry for all $k\in\IN$.
\end{lemma}
\begin{proof} 
We fix $k\in\IN$.  Surjectivity holds because elements in the target are
bihomogeneous of bidegree $(k,k)$. 

For brevity we set $N=(n+1)(r+1)$.
Let $P,Q\in K[{\bf U}]_k$ and $P\in \orth{(\ker {s^*})}$,
we must show $\langle s^*(P),s^*(Q)\rangle =\langle P,Q\rangle$.
We write $P= \sum_{\gamma} P_\gamma {\bf U}^\gamma$
and $Q=\sum_{\gamma} Q_\gamma {\bf U}^\gamma$, where here and
below the
sum is over all $\gamma\in \IN_0^N$ with $|\gamma|_1=k$.

We call $\gamma,\gamma_0\in \IN_0^N$ with $|\gamma|_1=|\gamma_0|_1 =
k$ equivalent, and write $\gamma\sim\gamma_0$, if and only if 
$s^*({\bf U}^\gamma) = s^*({\bf U}^{\gamma_0})$.
Let $R\subset\IN_0^N$ be a set of representatives of the equivalence classes.

Say $\gamma\sim\gamma_0$. 
Then $P\in \orth{(\ker s^*)}$ implies
$\langle P,{\bf U}^\gamma\rangle = \langle P,{\bf U}^{\gamma_0}\rangle$, so
\begin{equation}
\label{eq:Pcoeffequality}
{k\choose\gamma}^{-1}  {P_\gamma} = {k\choose \gamma_0}^{-1} 
P_{\gamma_0}.
\end{equation}

We have $\langle s^*(P),s^*(Q)\rangle = 
\sum_{\gamma_0\in R} \langle \sum_{\gamma\sim\gamma_0} P_\gamma s^*({\bf
  U}^\gamma),
\sum_{\gamma\sim\gamma_0} Q_\gamma s^*({\bf U}^\gamma)\rangle$
because  $s^*({\bf U}^\gamma)$ and $s^*({\bf U}^{\gamma_0})$ are
orthogonal if $\gamma\not\sim\gamma_0$.
Equality (\ref{eq:Pcoeffequality}) gives
\begin{alignat}1
\label{eq:sPsQ}
\langle s^*(P),s^*(Q)\rangle &=
\sum_{\gamma_0\in R} \left(\sum_{\gamma\sim\gamma_0} P_\gamma\right)
\left(\sum_{\gamma\sim\gamma_0} {\tau(Q_\gamma)}\right)
\langle s^*({\bf U}^{\gamma_0}),s^*({\bf U}^{\gamma_0})\rangle \\
\nonumber
&= \sum_{\gamma_0\in R}
{k\choose \gamma_0}^{-1}P_{\gamma_0}
\left(\sum_{\gamma\sim\gamma_0}{k\choose \gamma}\right)
\left(\sum_{\gamma\sim\gamma_0} {\tau(Q_\gamma)}\right)
\langle s^*({\bf U}^{\gamma_0}),s^*({\bf U}^{\gamma_0})\rangle.
\end{alignat}

Say $s^*({\bf U}^{\gamma_0}) = {\bf X}^{\delta}{\bf Y}^{\delta'}$
with $\delta\in\IN_0^{n+1}$, $\delta\in\IN_0^{r+1}$, and
$|\delta|_1=|\delta'|_1=k$.
By definition we have 
$\langle s^*({\bf U}^{\gamma_0}),s^*({\bf U}^{\gamma_0})\rangle
=\langle {\bf X}^\delta,{\bf X}^\delta\rangle
\langle {\bf Y}^{\delta'},{\bf Y}^{\delta'}\rangle
= {k \choose \delta}^{-1} {k\choose \delta'}^{-1}$.
On exanding both sides of
$s^*((\sum_{ij} U_{ij})^k) = (\sum_{ij} X_iY_j)^k$
and comparing coefficients we find
$\left(\sum_{\gamma\sim\gamma_0} {k\choose \gamma}\right)
\langle s^*({\bf U}^{\gamma_0}),s^*({\bf U}^{\gamma_0})\rangle=1$.
Therefore, (\ref{eq:sPsQ}) gives
$\langle s^*(P),s^*(Q)\rangle =\sum_{\gamma_0\in R}\sum_{\gamma\sim\gamma_0}
{{k\choose \gamma_0}}^{-1}{P_{\gamma_0}\tau(Q_\gamma)}$. But the right-hand
side is 
$ \langle P,Q\rangle$ because of
(\ref{eq:Pcoeffequality}). 
\end{proof}

We can state an analog result for the Veronese homomorphism.
The proof goes along similar lines as well.

\begin{lemma} 
\label{lem:veroisometry}
 The Veronese map 
$v^*_a:K[{\bf\widetilde X}]_k \rightarrow K[{\bf X}]_{ak}$
 is a surjective isometry for all $a,k\in\IN$.
\end{lemma}
\begin{proof}
We fix $a,k\in\IN$. 
Surjectivity follows immediately.
Let $P,Q\in K[{\bf\widetilde X}]_k$ and $P\in\orth{(\ker
  v^*_a)}$, we must show 
$\langle v^*_a(P),v^*_a(Q)\rangle = \langle P,Q\rangle$.

For brevity let $N = {n+a\choose n}$.
We recall that  ${\bf\widetilde X}$ is an $N$-tuple of independent
variables
$(X_\gamma)$ where $\gamma$ runs over elements of $\IN_0^{n+1}$ with
$|\gamma|_1 = a$. We use $\gamma$ to index elements of $\IN_0^N$.
If $\delta\in\IN_0^N$ with $|\delta|_1 = k$
then
\begin{equation*}
  v_a^*({\bf \widetilde X}^\delta) = {\bf X}^{\alpha(\delta)} 
\prod_{\gamma} {a \choose \gamma}^{\delta_\gamma/2}
\end{equation*}
for some $\alpha(\delta)\in\IN_0^{n+1}$, here and below $\gamma$ runs
over all elements of $\IN_0^{n+1}$ with $|\gamma|_1 = a$.

We call $\delta,\delta_0\in \IN_0^N$ with $|\delta|_1=|\delta_0|_1 =
k$ equivalent, and write $\delta\sim\delta_0$, if and only if
$\alpha(\delta)=\alpha(\delta_0)$. I.e. if and only if
$v_a^*({\bf \widetilde X}^\delta)$ and  $v_a^*({\bf \widetilde
  X}^{\delta_0})$ 
are equal up-to a factor in $K^\times$.
Let $R\subset\IN_0^N$ be a set of representatives of the equivalence classes.

Let $P=\sum_\delta P_\delta {\bf \widetilde X}^\delta$ and 
$Q=\sum_\delta Q_\delta {\bf\widetilde X}^{\delta}$ where here and
below
$\delta$ runs over all elements of $\IN_0^N$ with $|\delta|_1 = k$.

Say $\delta\sim\delta_0$. Our assumption $P\in \orth{(\ker v_a^*)}$
implies
\begin{equation*}
 \left\langle P, {\bf \widetilde X}^{\delta}\prod_\gamma {a\choose
   \gamma}^{-\delta_\gamma/2} \right\rangle =
\left\langle P,
{\bf \widetilde X}^{\delta_0}\prod_\gamma {a\choose \gamma}^{-\delta_{0\gamma}/2}\right\rangle.
\end{equation*}
So
\begin{equation*}
P_{\delta} 
{k \choose \delta}^{-1} \tau\left(\prod_\gamma {a \choose \gamma}^{-\delta_\gamma/2} \right)
= P_{\delta_0} {k \choose \delta_0}^{-1} \tau\left(\prod_\gamma {a \choose
    \gamma}^{-{\delta_0}_\gamma/2}\right).
\end{equation*}
and because $\tau$ acts trivially on roots of positive integers we obtain
\begin{equation}
\label{eq:Pdelta}
P_{\delta}{k \choose \delta}^{-1}\prod_\gamma {a \choose \gamma}^{-\delta_\gamma/2} 
= P_{\delta_0}{k \choose \delta_0}^{-1}\prod_\gamma {a \choose
    \gamma}^{-{\delta_0}_\gamma/2}.
\end{equation}

We note that $\langle v^*_a({\bf\widetilde
  X}^{\delta}),v^*_a({\bf\widetilde X}^{\delta_0})\rangle =0$
if $\delta\not\sim\delta_0$ and  evaluate
\begin{alignat*}1
  \langle v^*_a(P),v^*_a(Q)\rangle &= 
\sum_{\delta_0\in R}\left\langle\sum_{\delta\sim\delta_0} P_\delta
v^*_a({\bf\widetilde X}^\delta),
\sum_{\delta\sim\delta_0} Q_\delta
v^*_a({\bf\widetilde X}^\delta)\right\rangle \\
&= 
\sum_{\delta_0\in R}
\frac{P_{\delta_0}}{{k \choose \delta_0} \prod_\gamma {a\choose
    \gamma}^{{\delta_0}_\gamma/2}}
\left\langle
{\bf X}^{\alpha(\delta_0)}
\sum_{\delta\sim\delta_0} 
{k\choose \delta}
\prod_\gamma {a\choose    \gamma}^{{\delta}_\gamma},
{\bf X}^{\alpha(\delta_0)}
\sum_{\delta\sim\delta_0} Q_\delta
\prod_\gamma {a\choose \gamma}^{\delta_\gamma/2}
\right\rangle
\end{alignat*}
after applying (\ref{eq:Pdelta}). The definition of the inner product
shows that 
$\langle v^*_a(P),v^*_a(Q)\rangle$ equals
\begin{alignat*}1
\sum_{\delta_0\in R}
\frac{P_{\delta_0}}{{k \choose \delta_0}
\prod_\gamma {a\choose \gamma}^{{\delta_0}_\gamma/2}}
{ak \choose \alpha(\delta_0)}^{-1}
\left(\sum_{\delta\sim \delta_0} {k\choose \delta} \prod_\gamma
     {a\choose \gamma}^{\delta_\gamma}\right)
\left(\sum_{\delta\sim \delta_0} \tau(Q_\delta) \prod_\gamma
     {a\choose \gamma}^{\delta_\gamma/2}\right).
\end{alignat*}

We evaluate both sides of $v_a^*\left(\left(\sum_\gamma {a\choose \gamma}^{1/2} \widetilde
X_\gamma\right)^k\right)=
\left(\sum_\gamma {a\choose\gamma} {\bf X}^\gamma\right)^k
=\left(\sum_{i=0}^n X_i\right)^{ak}$ and
compare coefficients to deduce
$\sum_{\delta\sim\delta_0} {k\choose \delta} \prod_{\gamma}{a\choose
  \gamma}^{\delta_\gamma} = {ak\choose \alpha(\delta_0)}$.
Hence
\begin{equation*}
\langle v^*_a(P),v^*_a(Q)\rangle =  
 \sum_{\delta_0\in R}
 \frac{P_{\delta_0}}{{k \choose \delta_0}
\prod_\gamma {a\choose \gamma}^{{\delta_0}_\gamma/2}}
\left(\sum_{\delta\sim \delta_0} \tau(Q_\delta) \prod_\gamma
     {a\choose \gamma}^{\delta_\gamma/2}\right)
= \sum_{\delta} \frac{P_\delta \tau(Q_\delta)}{{k\choose \delta}}
\end{equation*}
where the second inequality follows from
 (\ref{eq:Pdelta}). But the right-hand side is just $\langle
P,Q\rangle$ by definition.
\end{proof}

Let $a,b\in\IN$  and let
${\bf\widetilde Y}$ be the ${r+b\choose b}$-tuple 
$(\widetilde Y_\gamma)$ where $\gamma$ runs through all elements in $\IN_0^{r+1}$
with $|\gamma|_1 = b$.
A consequence  Lemmas \ref{lem:tensoriso} and \ref{lem:veroisometry} is
that the linear map
\begin{equation*}
v^*_{ab}=v^*_a\otimes v^*_b:K[{\bf\widetilde X},{\bf \widetilde
      Y}]_{(k,k)}
\rightarrow K[{\bf X},{\bf  Y}]_{ak,bk}
\end{equation*}
is  a surjective isometry; 
we identified
$K[{\bf\widetilde X}]_k\otimes_K K[{\bf \widetilde
      Y}]_k = K[{\bf\widetilde X},{\bf\widetilde Y}]_{(k,k)}$
and
$K[{\bf X}]_{ak}\otimes_K K[{\bf 
      Y}]_{bk} = [{\bf X},{\bf  Y}]_{(ak,bk)}$.

The next lemma is  an application  to 
 an irreducible closed subvariety $Z\subset \IP^n\times\IP^r$ defined
 over $\IQbar$. We recall
that $s$ is the Segre morphism defined in Subsection \ref{sec:hilbestimates}.

\begin{lemma} 
\label{lem:segrebound}
We have
\begin{equation*}
\arithhilb{k,k;Z} \le \arithhilb{k;s(Z)} \quad\text{and}\quad 
\hilbfunc{k,k;Z} = \hilbfunc{k;s(Z)}.
\end{equation*}
for all $k\in\IN$.
\end{lemma}
\begin{proof}
 Let $I\subset \IQbar[{\bf X},{\bf Y}]$ be the ideal of $Z$
and $J\subset \IQbar[{\bf U}]$  the ideal of $s(Z)$. Then
\begin{equation}
\label{eq:segrepullback}
{s^*}^{-1}(I_{(k,k)}) = J_{k}.
\end{equation}

Remark \ref{rem:defineip}  gives us a $\tau$-inner product
$\langle\cdot,\cdot\rangle$
on $\IQbar[{\bf X},{\bf Y}]_{(k,k)}$
and $\IQbar[{\bf U}]_k$ with $\tau$ the identity on $\IQbar$. 

We now fix a basis $\{Q_1,\dots, Q_t \}$ of $\orth{J_k}$ where
$t = \hilbfunc{k;S(Z)} \ge 1$. We define
\begin{alignat*}1
 A &= [ \iota(Q_1),\dots, \iota(Q_t)]\in \mat{{nr+n+r + k \choose
     k},t}{\IQbar}\quad\text{and}\\
C &= [\iota(s^*(Q_1)),\dots, \iota(s^*(Q_t)) ] 
\in \mat{{n+k \choose k}{r+k\choose k},t}{\IQbar}.
\end{alignat*}

Let us remark that
$\ker s^*\subset J_k$ by (\ref{eq:segrepullback}). This implies
$\orth{(\ker s^*)} \supset \orth{J_k}$ and hence
 $Q_i\in\orth{(\ker s^*)}$.
We claim 
\begin{equation}
\label{eq:claimsorthorth}
  s^*(\orth{J_k}) = \orth{I_{(k,k)}}.
\end{equation}
Indeed, if $P\in I_{(k,k)}$ then there is $Q\in J_k$ with $s^*(Q) =
P$.
Therefore, $\langle s^*(Q_i),P\rangle = \langle Q_i,Q\rangle = 0$ for
all $i$
because $s^*$ is an isometry by Lemma \ref{lem:segreisometry}.
Hence $s^*(Q_i) \in \orth{I_{(k,k)}}$ because $P$ was arbitrary. This
shows that the left side of (\ref{eq:claimsorthorth}) is contained in
the right side.
Equality (\ref{eq:segrepullback}) implies $\dim J_k = \dim \ker s^* + \dim
I_{(k,k)}$. Moreover, $\dim \ker s^* = \dim \IQbar[{\bf U}]_k - \dim
\IQbar[{\bf X},{\bf Y}]_{(k,k)}$ since $s^*$ is
surjective. It follows that $\orth{J_k}$ and $\orth{I_{(k,k)}}$ have
equal dimension. But $\ker s^* \cap \orth{J_k}\subset \ker s^* \cap
\orth{(\ker s^*)} = 0$ by the comment shortly after
(\ref{def:isometry}). So $s^*|_{\orth{J_k}}$ is injective
and hence $\dim s^*(\orth{J_k}) = \dim \orth{I_{(k,k)}}$.

This settles our claim (\ref{eq:claimsorthorth})
and the equality involving the geometric Hilbert function in the
assertion. 
We also remark that $C$ and $A$ have equal rank $t$.

Let $K\subset\IQbar$ be a finite normal extension of $\IQ$ containing all 
of the finitely many algebraic numbers which are involved in the proof below.
Let $v\in \pl{K}$.

Say $v$ is infinite and let $\sigma = \sigma_v :K\rightarrow \IC$ be
an associated embedding. Since $K$ is normal over $\IQ$ there is an
automorphism $\eta$ of $K$ with $\sigma(\eta(x)) = \overline{\sigma(x)}$
for all $x\in K$.
By the Cauchy-Binet formula we have
\begin{alignat*}1
\exp{  2 \lh{v,t}{A}} &= \det(\trans{\sigma(A)} \overline{\sigma(A)})
= \det [\trans{\sigma(\iota(Q_i))}
  \overline{\sigma(\iota(Q_j))}]_{1\le i,j\le t} 
= \sigma(\det [\trans{\iota(Q_i)}
  \eta(\iota(Q_j))]_{ij})
\end{alignat*}
All coefficients involved in $\iota$ are totally real algebraic
numbers; hence 
invariant under $\eta$.
We obtain $\exp{2\lh{v,t}{A}} = \sigma(\det[\langle Q_i,\eta(Q_j)\rangle]_{ij})$.
Together with the Cauchy-Binet formula we also get
\begin{alignat*}1
\exp{  2 \lh{v,t}{C}} &= \det(\trans{\sigma(C)} \overline{\sigma(C)})
= \det [\trans{\sigma(\iota(s^*(Q_i)))}
  \overline{\sigma(\iota(s^*(Q_j)))}]_{ij} \\
&
= \sigma(\det [\trans{\iota(s^*(Q_i))} \eta(\iota(s^*(Q_j))) ]_{ij})
= \sigma(\det [\langle s^*(Q_i), s^*(\eta(Q_j))\rangle ]_{ij}).
\end{alignat*}
But $s^*$ is an isometry
so $\exp 2\lh{v,t}{C} =  \sigma(\det [\langle Q_i ,\eta(Q_j)\rangle ]_{ij})$,
hence
\begin{equation}
\label{eq:localequalityinfinite}
 \lh{v,t}{C} = \lh{v,t}{A}. 
\end{equation}

Now say $v$ is finite. If $B$ is a matrix with $m$ rows and $1\le i_1
< \cdots < i_t\le m$ are integers, we let $B_{i_1\dots i_t}$ denote
the submatrix of $B$ consisting of the rows $i_1,\dots,i_t$. 
We fix $i_1,\dots,i_t$  with
$|\det C_{i_1\dots i_t}|_v = \exp \lh{v,t}{C}$.
For $1\le i\le t$ we may find polynomials $\widetilde Q_i\in K[{\bf
    U}]_k$ whose terms are also terms of $Q_i$ with the following
property. The rows $i_1,\dots,i_t$ of
$\widetilde C = [\iota(s^*(\widetilde
  Q_1)),\cdots,\iota(s^*(\widetilde Q_t))]$
equal the corresponding rows of $C$ and all other rows are $0$.
The Cauchy-Binet formula implies $\det(\trans{C}\widetilde C)= 
\det(C_{i_1\dots i_t})^2$.

On the other hand, we have
$\det(\trans{C}\widetilde C) = \det [\langle
  s^*Q_i,s^*\widetilde Q_j\rangle]_{ij}
=\det [\langle
  Q_i,\widetilde Q_j\rangle]_{ij}$ because $s^*$ is an isometry. 
So $\det(\trans{C}\widetilde C) = \det(\trans{A}\widetilde A)$ with
$\widetilde A=[\iota(\widetilde Q_1),\ldots, \iota(\widetilde Q_t)]$.
Now we apply the Cauchy-Binet formula to evaluate
$\det(\trans{A}\widetilde A) = \sum_{i'_1<\cdots< i'_t}
\det(A_{i'_1\dots i'_t})\det(\widetilde A_{i'_1\dots i'_t})$.
Since a row of $\widetilde A$ either equals the corresponding row of
$A$ or is $0$ we see that  $\det(\widetilde A_{i'_1\dots i'_t})$
is either $\det( A_{i'_1\dots i'_t})$ or $0$.
The ultrametric triangle inequality implies
$|\det C_{i_1\dots i_t}|_v^2 =
|\det\trans{C}{\widetilde C}|_v=|\det(\trans{A}\widetilde A)|_v\le \max_{i'_1< \cdots< i'_t}
|\det A_{i'_1\dots i'_t}|_v^2 = \exp(2\lh{v,t}{A})$.
We conclude 
\begin{equation}
\label{eq:localequalityfinite}
  \lh{v,t}{C}\le \lh{v,t}{A}.
\end{equation}

We multiply (\ref{eq:localequalityinfinite}) and
(\ref{eq:localequalityfinite}) with $[K_v:\IQ_v]/[K:\IQ]$ and take the
sum over all places of $K$ to obtain
\begin{equation}
\label{eq:matheightCbound}
 \matheight{t}{C}\le \matheight{t}{A}.
\end{equation}

By definition we have $\arithhilb{k,k;Z} =
\vsheight{I_{(k,k)}}$.
We know how the height behaves under taken the orthogonal complement,
see (\ref{eq:orthcompl}) in Remark \ref{rem:orthcomp}.
We conclude $\arithhilb{k,k;Z} = \vsheight{\orth{I_{(k,k)}}}$.
The columns of $C$ are a basis of $\iota(\orth{I_{(k,k)}})$ by
(\ref{eq:claimsorthorth}),
so $\arithhilb{k,k;Z} = \matheight{t}{C}$. 
On the other hand, $\arithhilb{k,s(Z)} = \vsheight{J_k}
= \vsheight{\orth{J_k}} = \matheight{t}{A}$ by similar
arguments. 
The proof follows from (\ref{eq:matheightCbound}).
\end{proof}


\begin{lemma} 
\label{lem:verobound}
We have
\begin{equation*}
\arithhilb{ak,bk;Z} \le \arithhilb{k,k;v_{ab}(Z)}
\quad\text{and}\quad
\hilbfunc{ak,bk;Z} = \hilbfunc{k,k;v_{ab}(Z)}
\end{equation*}
for $a,b,k\in\IN$.
\end{lemma}
\begin{proof}
 Let 
$I$ be the ideal of $Z\subset\IP^n\times\IP^r$ and let
$J$ be the
ideal of the closed subvariety $v_{ab}(Z)\subset\IP^{{n+a
    \choose a}-1}\times \IP^{{r+b \choose b}-1}$. Following the
convention that  $v^*_{ab}$  denotes the restriction of the
Veronese homomorphism to polynomials of fixed bidegree, we have
\begin{equation*}
{v^*_{ab}}^{-1}(I_{(ak,bk)}) =  J_{(k,k)}.
\end{equation*}

We mainly follow the lines of the proof of Lemma \ref{lem:segrebound}.
Remark \ref{rem:defineip}  gives us a $\tau$-inner product
$\langle\cdot,\cdot\rangle$
on source and target of $v^*_{ab}$ with $\tau$ the identity on $\IQbar$.
Let us fix a basis $\{Q_1,\dots,Q_t\}$ of $\orth{J_{(k,k)}}$
with $t= \hilbfunc{k,k;v_{ab}(Z)}$. We also introduce matrices
\begin{alignat*}1
  A &= [\iota(Q_1),\dots \iota(Q_t)] \in\mat{{ {n+a \choose a}- 1+k
      \choose k}{ {r+b \choose b}-1+ k
      \choose k},t}{\IQbar}\quad\text{and} \\
  C &= [\iota(v_{ab}^*(Q_1)),\dots, \iota(v_{ab}^*(Q_t))] 
\in \mat{ {n+ ak \choose ak}{r + bk \choose bk},t}{\IQbar}.
\end{alignat*}
We have $J_{(k,k)} \supset \ker v^*_{ab}$, so $\orth{J_{(k,k)}}\subset
\orth{(\ker v^*_{ab})}$ and hence
$Q_i \in \orth{(\ker v^*_{ab})}$.

As in the proof of Lemma \ref{lem:segrebound}  we use the fact that
$v^*_{ab}$ is an isometry, cf. Lemma \ref{lem:veroisometry}, to deduce
\begin{equation*}
  v_{ab}^*(\orth{J_{(k,k)}}) = \orth{I_{(ak,bk)}}
\end{equation*}
as well as the second claim of the current lemma and also  that  $C$
has rank $t$.

Again we let $K$ be a finite normal extension of $\IQ$ which contains
all algebraic numbers which follow. Let $v\in \pl{K}$.

Say $v$ is infinite. Just as in Lemma
\ref{lem:segrebound} we have
\begin{equation*}
  \lh{v,t}{C} = \lh{v,t}{A};
\end{equation*}
in order to show this we need to use the fact that $v^*_{ab}$ is an
isometry but also that its coefficients are totally real numbers. 
Now say $v$ is finite. Using a similar argument as in
Lemma \ref{lem:segrebound} we find
\begin{equation*}
  \lh{v,t}{C}\le \lh{v,t}{A}.
\end{equation*}
The remainder of the proof is just as in Lemma \ref{lem:segrebound}.
\end{proof}

\section{Correspondences and Height Inequalities}
\label{sec:corr}
The goal of this section is to prove  a height inequality on
correspondences.

 Let $Z\subset\IP^n\times\IP^r$ be an irreducible
 closed subvariety of positive dimension defined over $\IQbar$. Recall
 that $s:\IP^n\times\IP^r\rightarrow \IP^{nr+n+r}$ is the Segre
 morphism. We  define the height
 $\height{Z}$ and degree $\deg{Z}$ of $Z$ as height and degree of
 $s(Z)$. 
Recall that $\pi_1$ and $\pi_2$ are the two projections
$\IP^n\times\IP^r\rightarrow \IP^n$ and 
$\IP^n\times\IP^r\rightarrow \IP^r$.
We may attach to $Z$ a
 $(1+\dim Z)$-tuple of bidegrees
 \begin{equation*}
   \Delta_i(Z) = (\pi_1^*\bigO{1}^i \pi_2^*\bigO{1}^{\dim Z-i} [Z]) \ge 0
\quad\text{for}\quad i\in \{0,\ldots,\dim Z\}.
 \end{equation*}
In this notation, the Hilbert polynomial of $Z$ is 
\begin{equation}
\label{eq:hilbpolydelta}
  \hhilbpoly{T_1,T_2;Z} = \sum_{i=0}^{\dim Z}
{\dim Z \choose i} \Delta_i(Z) T_1^i T_2^{\dim Z-i}.
\end{equation}

We introduce two restrictions on $Z$.
\begin{enumerate}
\item [(i)] We suppose $\Delta_0(Z) > 0$. This is equivalent to stating
  that 
$\pi_2|_Z : Z\rightarrow \IP^r$ is generically finite.
\item[(ii)] We suppose there exists $i\in \{1,\dots,\dim Z\}$ with $\Delta_i(Z)>0$.
\end{enumerate}

Property (i) implies $\dim Z\le r$. 
Property (ii) excludes examples such as
 $Z=\{p\}\times Z'$ for some irreducibe closed subvariety 
$Z'\subset\IP^r$. 

If $Z$ satisfies conditions (i) and (ii) above, we  define 
\begin{equation}
\label{eq:definekappa}
\kappa(Z) = \min \left\{ \frac{\Delta_0(Z)}{\Delta_1(Z)},
\left(\frac{\Delta_0(Z)}{\Delta_2(Z)}\right)^{1/2},\dots,
\left(\frac{\Delta_0(Z)}{\Delta_d(Z)}\right)^{1/d}\right\}.
\end{equation}
Note that  certain quotients can be infinity. But the minimum
is always a positive real number by (i) and (ii).

\begin{proposition}
\label{prop:heightlb}
 Let $Z\subset\IP^n\times\IP^r$ be an irreducible
 closed subvariety defined over $\IQbar$ of dimension $d \ge 1$. We assume that $Z$ satifies
 conditions (i) and (ii) above.
Let $\kappa = \kappa(Z)$
and $\Delta_0=\Delta_0(Z)$.
 For brevity, we
introduce the constant
\begin{equation*}
k_0 =
\max\{17\cdot 3^d d! n \Delta_0^{d-1},\deg{Z}\}\times 
\left\{
  \begin{array}{cl}
    1 &: \text{if }
    \kappa \ge 17dn \text{ and}\\
    100 \frac{dn}{\kappa} &: \text{elsewise.}
  \end{array}
\right.
\end{equation*}
There exist  $a,b\in \IN$ 
with $\max\{a,b\}\le k_0$ and 
\begin{equation*}
a=b \quad\text{if}\quad \kappa \ge 17dn
\end{equation*}
 and $F\in\IQbar[{\bf X},{\bf Y}]_{(a,b)}$
with
 \begin{equation*}
\height{F}\le 
\left(400n\max\{n,r\}^2 \frac{\max\{1,\kappa\}^{d+1}}{\Delta_0}
(\height{Z}+\deg{Z}) + 5 d\deg{Z}\right)k_0
 \end{equation*}
 such that  the following properties hold. The polynomial $F$ does not
 vanish identically on $Z$
and if $(p,q)\in Z(\IQbar)$ then we are in one of the following cases.
\begin{enumerate}
\item [(i)] We have $F(p,q)=0$ or some projective coordinate of
  $\IP^n$ vanishes at $p$.
\item [(ii)] We have the inequality
\begin{equation*}
\kappa \height{p}\le 2^5 dn \height{q} + 
2^{14} n^2r\max\{n,r\}^2
 \frac{\max\{1,\kappa\}^{d+1}}{\Delta_0}
(\height{Z}+\deg{Z}) + 2^{8} n r^2\deg{Z}.
 \end{equation*}
\end{enumerate}
\end{proposition}

For our application a slightly modified version of the previous
proposition will be of central importance. 

\begin{proposition}
\label{prop:heightlb2}
 Let $Z\subset\IP^n\times\IP^r$ be an irreducible
 closed subvariety defined over $\IQbar$ of dimension $d \ge 1$. We assume that $Z$ satifies
 conditions (i) and (ii) above.
Let $\kappa = \kappa(Z)$
and $\Delta_0=\Delta_0(Z)$, we shall assume $\kappa \ge 17dn$.
We set $k_0 = \max\{17\cdot 3^d d! n \Delta_0^{d-1},\deg{Z}\}$.
There is a finite collection of irreducible closed subvarieties
$V_1,\dots,V_N\subset \pi_1(X)  \subset \IP^n$ defined over $\IQbar$ 
with 
  \begin{alignat}1
\nonumber
\dim{V_i} & = d-1, \\
\label{eq:degreebound}
    \sum_{i=1}^N \deg{V_i} &\le k_0 \deg
    Z, \quad\text{and}\\
\nonumber
 \quad \height{V_i} &\le
2^9
\max\{n^2r^2,nr^3,n^3r\}
\max\left\{1,\frac{\kappa^{d+1}}{\Delta_0}\right\} (\height{Z}+\deg{Z}) \deg{Z} k_0,
  \end{alignat}
and such that the following property holds. 
Let $(p,q)\in Z(\IQbar)$, then we are in one of the following cases.
\begin{enumerate}
\item [(i)] Some  projective coordinate of $\IP^n$ vanishes at
  $p$
or $(p,q)$ is not isolated in $\pi_1|_Z^{-1}(p)$.
\item [(ii)] There is $1\le i\le N$ with $p\in V_i(\IQbar)$.
\item[(iii)] We have 
  \begin{equation*}
\kappa \height{p}
\le 2^5 dn \height{q}
+       2^{15}  \max\{n^4r,n^2r^3\}
\max\left\{1,\frac{\kappa^{d+1}}{\Delta_0}\right\}(\height{Z}+\deg{Z}).
 \end{equation*}
\end{enumerate}
\end{proposition}

\subsection{Height Lower Bounds}

Let $a,b\in \IN_0$
Recall that $\mathbf{X}=(X_0,\dots,X_n)$ and
$\mathbf{Y}=(Y_0,\dots,Y_r)$ and that we  defined a height
of  $(n+1)$-tuples $F=(F_0,\dots,F_n)\in
(\IQbar[\mathbf{X},\mathbf{Y}]_{(a,b)})^{n+1}$ in Section 
\ref{sec:heights}. 

Until the end of Section \ref{sec:corr} we will use the following
notation. We let $Z\subset\IP^n\times\IP^r$ be an irreducible closed
subvariety defined over $\IQbar$ of dimension $d\ge 1$. 
We suppose that $Z$ satisfies conditions (i) and (ii) introduced
at the beginning of Section \ref{sec:corr}. We suppose additionally
that the product
$X_0\cdots X_n$ does not vanish identically on $Z$. 
Finally, for brevity we set $\Delta_i=\Delta_i(Z)$ for $0\le i\le d$.

The main tool for the height inequality below is the product formula.

\begin{lemma}
\label{lem:heightlb}
Suppose $I \subset \IQbar[{\bf X},{\bf Y}]$ is the ideal of $Z$. 
Let $F=(F_0,\dots,F_n)$ be an $(n+1)$-tuple as above such
that each $F_i$ is bihomogeneous of bidegree $(a,b)$
with $a,b\in\IN$. We
assume  that there is a positive integer $c$ such that
\begin{equation*}
X_i^c F_0 - X_0^c F_i \in I\quad\text{for}\quad 1\le i\le n.
\end{equation*}
Let $(p,q)\in Z(\IQbar)$ where
$p=[p_0:\cdots:p_n]$ and assume that $p_0\not=0$ and $F_0(p,q)\not=0$.
 Then 
\begin{equation*}
(c-a) \height{p} \le b \height{q} + \height{F} + \frac c2 \log(n+1).
\end{equation*}
\end{lemma}
\begin{proof}
  Let $K\subset\IQbar$ be a number field containing algebraic projective
  coordinates $p_0,\dots,p_n$ and $q_0,\dots,q_r$ of $p$ and $q$,
  respectively. We also assume that $K$ contains all of the finitely many algebraic numbers implicit in the
  subsequent proof. 

We write $x = (p_0,\dots,p_n)\in K^{n+1},y=(q_0,\dots,q_r)\in
K^{r+1},$
and $z = (F_0(x,y),\dots,F_n(x,y))\in K^{n+1} \ssm\{0\}$.

For $v\in\pl{K}$ we set
\begin{equation*}
  \lambda(v) 
= \log \frac{ |x|_v^a |y|_v^b }{|z|_v}.
\end{equation*}

The definition of $|F_i|_v$ was given in Section 
\ref{sec:heights}. The Cauchy-Schwarz inequality gives
 $|F_i(x,y)|_v \le |F_i|_v |x|_v^a |y|_v^b$  if $v$
is an infinite place of $K$.
The same inequality holds for finite $v$ by the ultrametric triangle
inequality.
For infinite $v$ we have
$|z|_v = (|F_0(x,y)|_v^2+\cdots + |F_n(x,y)|_v^2)^{1/2}\le
(|F_0|_v^2+\cdots+|F_n|^2)^{1/2} |x|_v^a |y|_v^b
= |F|_v |x|_v^a |y|_v^b$. 
And for finite $v$ the corresponding statement is
 $|z|_v = \max\{|F_0(x,y)|_v,\dots,|F_n(x,y))|_v \}\le |F|_v |x|_v^a
|y|_v^b$.  

These bounds imply
\begin{equation}
  \label{eq:lambdabound}
\lambda(v)\ge  - \log|F|_v
\end{equation}
regardless if $v$ is finite or not.

The hypothesis implies $x_i^c F_0(x,y) = x_0^c F_i(x,y)$ for $0\le i\le
n$,
therefore 
\begin{equation*}
z= F_0(x,y) x_0^{-c} (x_0^c,\dots,x_n^c).  
\end{equation*}
Hence
\begin{equation*}
  \lambda(v) 
= \log \left|\frac{x_0^c}{F_0(x,y)}\right|_v
 + \log \frac{ |x|_v^a |y|_v^b }{|(x_0^c,\dots,x_n^c)|_v}.
\end{equation*}
If $v$ is finite, then $|(x_0^c,\dots,x_n^c)|_v = |x|_v^c$ and if $v$
is infinite we may estimate
$|x|_v^{2c} = (\sum_{i=0}^n |x_i|_v^2)^c\le
(n+1)^{c}|(x_0^c,\dots,x_n^c)|^2_v$.

Now we sum over all places of $K$ weighted with the appropriate local
degress to obtain
\begin{alignat*}1
  \sum_{v\in\pl{K}} \frac{[K_v:\IQ_v]}{[K:\IQ]}
\lambda(v) 
\le  &\sum_{v\in\pl{K}} \frac{[K_v:\IQ_v]}{[K:\IQ]}
  \log \left|\frac{x_0^c}{F_0(x,y)}\right|_v \\
& + \sum_{v\in \pl{K}}\frac{[K_v:\IQ_v]}{[K:\IQ]}\log  |x|_v^{a-c} |y|_v^b 
+ \sum_{v\text{ infinite}} \frac{[K_v:\IQ_v]}{[K:\IQ]}\frac c2 \log(n+1).
\end{alignat*}

The first term on the right side of the inequality is zero by the product formula.
The second term is $(a-c)\height{p}+b\height{q}$ by the definition
of our height. The final term is $c/2 \log (n+1)$
since $\sum_{v\text{ infinite}} [K_v:\IQ_v] = [K:\IQ]$.
We use (\ref{eq:lambdabound}) to bound the left side of the inequality
from below. This completes the proof.
\end{proof}

\subsection{Dimension Inequalities}
\label{sec:dimlb}

For brevity we set $R=\IQbar[{\bf X},{\bf Y}]$.
Let $I\subset R$ be the ideal of 
 $Z\subset\IP^n\times\IP^r$.

By our convention, stated in Section \ref{sec:gennotation},  $\IQbar$ is a subfield of $\IC$.
We now take $\tau$ to be complex conjugation restricted to $\IQbar$. 
It acts trivially on square roots of positive integers. So it satifies
the restrictions imposed in Remark \ref{rem:defineip} 

In the following sections, $\langle\cdot,\cdot\rangle$ is the
$\tau$-inner product as in said
remark  on  vector spaces of polynomials with
coefficients in $\IQbar$.

\begin{remark}
  \label{rem:posdef}
Suppose $F \in I_{(a,b)}$.
Since $\tau$ is the restriction of complex conjugation, the inner
product
 $\langle F,F\rangle$ vanishes if and only if $F=0$.
Therefore, we have $I_{(a,b)}\cap \orth{I_{(a,b)}}=0$. 
\end{remark}

Until the end of this section we treat $D,E,k\in \IN$ as
parameters. They will be chosen later
on.  We set
\begin{alignat}1
\label{eq:defWDEk}
W_{DEk} =\{ &(F_0,\dots,F_n)\in (\orth{ I_{(k,Ek)}})^{n+1}; 
X_i^{Dk} F_0 - X_0^{Dk}F_i\in I \text{ for } 1\le i \le n\}.
\end{alignat}

Any  $F = (F_0,\dots,F_n)\in W_{DEk}$ satisfies the hypothesis
of Lemma \ref{lem:heightlb} with $(a,b) = (k,Ek)$ and $c = Dk$. 

Our general strategy is a classical step in many proofs in diophantine
approximation. We shall find a non-zero element of small height in
$W_{DEk}$. 
We carry out this strategy by applying an absolute version
of Siegel's Lemma.  
Said element will then be used  to apply Lemma
\ref{lem:heightlb}. 

Of course, we can only find non-zero elements in $W_{DEk}$ if this
vector space is non-trivial. 
Our first task will be to bound  the dimension of $W_{DEk}$ from below. 

\begin{lemma}
\label{lem:firstWDEklb}
  We have $\dim W_{DEk} \ge (n+1)\hilbfunc{k,Ek;Z}-n\hilbfunc{(D+1)k,Ek;Z}$.
\end{lemma}
\begin{proof}
  The vector space  $W_{DEk}$ is the kernel  of
\begin{alignat*}1
(\orth{I_{(k,Ek)}})^{n+1}&\rightarrow
\left(R_{((D+1)k,Ek)}/I_{((D+1)k,Ek)}\right)^n
\\
(F_0,\dots,F_n)&\mapsto (X_i^{Dk}F_0-X_0^{Dk}F_i)_{1\le i\le n}.
\end{alignat*}
Hence we have the dimension inequality
\begin{alignat*}1
\dim W_{DEk} &\ge (n+1) \dim \orth{I_{(k,Ek)}}
-n \dim\orth{I_{((D+1)k,Ek)}} \\
&= (n+1)\hilbfunc{k,Ek;Z}-n\hilbfunc{(D+1)k,Ek;Z}.\qedhere
\end{alignat*}
\end{proof}

We will use the bounds developed in Section \ref{sec:hilbestimates} to
estimate the dimension of $W_{DEk}$ in terms  of  the bidegree
$\Delta_0$.

\begin{lemma}
\label{lem:dimlb}
We assume $k\ge \Delta_0$ and
\begin{equation*}
\frac{D+1}{E} \le \frac{1}{4dn} \kappa(Z).
\end{equation*}
Then
\begin{equation}
\label{eq:hpolyineq}
\hhilbpoly{D+1,E;Z} \le \left(1+\frac{1}{2n}\right)\Delta_0 E^d
\end{equation}
and 
\begin{equation*}
\frac{1}{\Delta_0 E^d}\dim W_{DEk} \ge  \frac{1}{2}\frac{k^d}{d!} -
4n e^d \Delta_0^{d-1}k^{d-1}.
\end{equation*}
\end{lemma}
\begin{proof}We begin by showing (\ref{eq:hpolyineq}). 
We use (\ref{eq:hilbpolydelta}) to estimate values of the Hilbert polynomial.
It suffices to
  show the inequality in
\begin{equation}
\label{eq:hpolyineqclaim}
\frac{1}{E^d} \hhilbpoly{D+1,E;Z} - \Delta_0= \sum_{i=1}^d {d\choose i} \Delta_i \left(\frac{D+1}{E}\right)^i 
\stackrel{?}{\le} \frac{\Delta_0}{2n}.
\end{equation}

We set $\delta_i = (\Delta_0/\Delta_i)^{1/i}$ if $\Delta_i\not=0$; by hypothesis we
have $(D+1)/E \le \delta_i /(4dn)$. We deduce
\begin{equation*}
\sum_{i=1}^d {d\choose i} \Delta_i \left(\frac{D+1}{E}\right)^i 
\le
\sum_{\atopx{i=1}{\Delta_i\not=0}}^d {d\choose i} \Delta_i\delta_i^i \frac{1}{(4dn)^i}
\le
\Delta_0 \sum_{i=1}^d {d\choose i}  \frac{1}{(4dn)^i}.
\end{equation*}
It suffices to show that the right-hand side is at most $\Delta_0/(2n)$. We have
\begin{equation*}
\sum_{i=1}^d {d\choose i}
\frac{1}{(4dn)^i}
=  \left(1+\frac{1}{4dn}\right)^d -1 \le \exp\left(\frac{1}{4n}\right)-1
\end{equation*}
since $(1+\frac{1}{4dn})^d$ is increasing in $d$ with limit 
$\exp(1/(4n))$. Elementary estimates show 
$e^{1/(4n)}-1 \le 1/(2n)$
and our claim
(\ref{eq:hpolyineqclaim}) follows. 

We come to the second part of the lemma. 
By Lemma \ref{lem:hilbfunclb}  we have
\begin{equation*}
\hilbfunc{k,Ek;Z} \ge \Delta_0 { d+Ek-\Delta_0 \choose d }
\ge \frac{\Delta_0}{d!} (Ek-\Delta_0)^d. 
\end{equation*}
Expanding the expression on the right  and isolating the
term $(Ek)^d$ gives
\begin{alignat}1
\label{eq:hilbkEk}
\hilbfunc{k,Ek;Z} &\ge \Delta_0\frac{(Ek)^d}{d!}
+ \sum_{i=0}^{d-1}{d \choose i} (Ek)^i (-\Delta_0)^{d-i} \\
\nonumber
&\ge \Delta_0\frac{(Ek)^d}{d!}
 - (Ek)^{d-1} \sum_{i=0}^{d-1}{d \choose i}  \Delta_0^{d-i}
\ge \Delta_0\frac{(Ek)^d}{d!}
 - (Ek)^{d-1}(2\Delta_0)^d.
\end{alignat}

On the other hand, Lemma \ref{lem:hilbfuncub} implies the upper bound 
\begin{alignat*}1
\hilbfunc{(D+1)k,Ek;Z} &\le {d+k\choose d} \hhilbpoly{D+1,E;Z}
\le \left(\frac{k^d}{d!} + e^d k^{d-1}\right)\hhilbpoly{D+1,E;Z};
\end{alignat*}
the second inequality follows from basic calculus where $e=2.71828\dots$. 
If we apply  inequality (\ref{eq:hilbkEk}) to
the conclusion of Lemma \ref{lem:firstWDEklb} we get
\begin{alignat*}2
\dim W_{DEk} &\ge 
\frac 12 \hilbfunc{k,Ek;Z} + n \left(\left(1+\frac
      {1}{2n}\right)\hilbfunc{k,Ek;Z}
- \hilbfunc{(D+1)k,Ek;Z}\right) \\
&\ge \frac {\Delta_0}{2}  \frac{(Ek)^d}{d!} + n
\frac{k^d}{d!}\left(\left(1+\frac{1}{2n}\right) \Delta_0 E^d -
\hhilbpoly{D+1,E;Z}\right) +\\
 &\quad\quad -(n+1) (Ek)^{d-1}(2\Delta_0)^d -
ne^d k^{d-1}\hhilbpoly{D+1,E;Z}.
\end{alignat*}
Because of the first statement of this lemma, the second term on the very
right of the inequality is non-negative. 
This statement also controls the remaining
$\hhilbpoly{D+1,E;Z}$, hence
\begin{alignat*}1
\dim W_{DEk} &\ge \frac {\Delta_0}{2}  \frac{(Ek)^d}{d!} 
 -(n+1) (Ek)^{d-1}(2\Delta_0)^d -
ne^d k^{d-1}\hhilbpoly{D+1,E;Z} \\
&\ge \frac {\Delta_0}{2}  \frac{(Ek)^d}{d!} 
 -(n+1) (Ek)^{d-1}(2\Delta_0)^d -
\left(n+\frac 12\right)e^d k^{d-1}\Delta_0 E^d.
\end{alignat*}
The proof follows from
$(n+1) (Ek)^{d-1}(2\Delta_0)^d + (n+1/2)e^d k^{d-1}\Delta_0 E^d
\le 4n e^d \Delta_0^d E^d k^{d-1}$.
\end{proof}

In particular, if $D$ and $E$ satisfy the lemma's hypothesis,
then $W_{DEk} \not= 0$ for large $k$.

\subsection{Bounding the Height of $W_{DEk}$}
In order to apply  Siegel's Lemma to $W_{DEk}$ we must also bound its height
from above. It turns out to be easier to work with a  vector
space $W'_{DEk}$ closely related to $W_{DEk}$ which we proceed to define.

For each $i\in \{0,\ldots,n\}$  we set 
\begin{equation*}
V_i = \{ X_i^{Dk} F;\,\, F\in \orth{I_{(k,Ek)}} \} \subset 
R_{((D+1)k,Ek)}
\end{equation*}
and note that $\dim V_i = \dim \orth{I_{(k,Ek)}} = \hilbfunc{k,Ek;Z}$.
We also define
\begin{alignat*}1
W'_{DEk} =\{ (G_1,G'_1,\dots,G_n, G'_n)\in \prod_{i=1}^n V_0\times
V_i;\,\, 
& G_i - G'_i \in I\text{ for }1\le i\le n,  \\
&\frac{G'_1}{X_1^{Dk}} = \cdots = \frac{G'_n}{X_n^{Dk}} \}.
\end{alignat*}

We recall that the product $X_0\cdots X_n$ does not vanish identically on $Z$.

There is a homomorphism $\Psi: W'_{DEk}\rightarrow W_{DEk}$ defined by
\begin{equation}
\label{eq:defPsi}
(G_1,G'_1,\dots,G_n,G'_n)\mapsto \left(\frac{G'_1}{X_1^{Dk}},
\frac{G_1}{X_0^{Dk}},\frac{G_2}{X_0^{Dk}},\dots,\frac{G_n}{X_0^{Dk}}\right).
\end{equation}
It is readily checked to be injective. Moreover it has a 
right-inverse given by
\begin{equation*}
(F_0,\dots,F_n) \mapsto (F_1 X_0^{Dk},F_0 X_1^{Dk}, F_2 X_0^{Dk}, F_0
  X_2^{Dk}, \dots, F_n X_0^{Dk}, F_0 X_n^{Dk}). 
\end{equation*}
So $\Psi$ is an isomorphism and  $\dim
W'_{DEk} = \dim W_{DEk}$. 

We proceed by bounding the height of $W'_{DEk}$. 
To do this, we write our vector space as
an intersection
\begin{equation*}
W'_{DEk} = W_1\cap W_2 \cap W_3
\end{equation*}
and bound the height of each $W_i$ separately. The desired height
bound will then follow from
(\ref{eq:schmidtineq}).

Explicitly, we set
\begin{alignat*}1
W_1 &= \prod_{i=1}^n V_0\times V_i, \\
W_2 &= \{(G_1,G'_1,\dots,G_n,G'_n)\in R_{((D+1)k,Ek)}^{2n};
\,\, G_i-G'_i \in I\text{ for }1\le i\le n\},\quad\text{and}\\
W_3 &=\{(G_1,G'_1,\dots,G_n,G'_n)\in R_{((D+1)k,Ek)}^{2n};
\,\, X_1^{Dk} G'_i = X_i^{Dk} G'_1\text{ for }1\le i\le n\}.
\end{alignat*}

A preliminary step in bounding the heights of the $W_j$ is to relate 
the height of $\orth{I_{(a,b)}}$ with the value of the arithmetic
  Hilbert function at $(a,b)$. 
Later is just the height of $I_{(a,b)}$.
 The connection is a simple as one
  could hope for.

  \begin{lemma}
\label{lem:heightcompare}
    We have $\vsheight{\orth{I_{(a,b)}}} = \arithhilb{a,b;Z}$ for all $a,b\in\IN$.
  \end{lemma}
  \begin{proof}
    We fix a basis 
 $\{Q_1,\dots, Q_t \}$ of $I_{(a,b)}$.
A polynomial $P$ lies in $\orth{I_{(a,b)}}$ if and only if
$\langle P, Q_i \rangle = \trans{\iota(P)}\cdot \tau(\iota(Q_i)) =  0$
for all $1\le i\le t$. So $\iota(\orth{I_{(a,b)}})$ 
is the kernel of the matrix $A$ with $s$ columns given by 
$\tau(\iota(Q_i))$. 
By (\ref{eq:orthcompl}) the height
$\matheight{t}{A}$ is  $\vsheight{\iota(\orth{I_{(a,b)}})}=
\vsheight{\orth{I_{(a,b)}}}$.
On the other hand, $\matheight{t}{A} = \matheight{t}{\tau(A)}$.
The columns of  $\tau(A)$ come from a basis of $I_{(a,b)}$
since $\tau(\tau(\iota(Q_i)))=\iota(Q_i)$. 
Hence $\matheight{t}{\tau(A)} = \vsheight{I_{(a,b)}} = 
\arithhilb{a,b;Z}$ and the proof is complete. 
  \end{proof}

\begin{lemma} 
\label{lem:hW1}
 We have 
\begin{alignat*}1
\vsheight{W_1} \le &2n 
\arithhilb{k,Ek;Z}+
20n\max\{n,r\}^2 \max\{D,E\} \hilbfunc{k,Ek;Z}k.
\end{alignat*}
\end{lemma}
\begin{proof}
Since $W_1 = \prod_{i=1}^n V_0\times V_i$ we  use (\ref{eq:heightprodineq}) to deduce 
\begin{equation}
\label{eq:decomposehW1}
\vsheight{W_1} \le n\vsheight{V_0}+\sum_{i=1}^n\vsheight{V_i}
\le 2n \max\{\vsheight{V_0},\ldots,\vsheight{V_n}\}.
\end{equation}
We continue by bounding the height of each $V_i$. 

Let $t=\dim(\orth{I_{(k,Ek)}}) = \hilbfunc{k,Ek;Z}\ge 1$ 
and let $P_1,\dots,P_t$ be a basis
of $\orth{I_{(k,Ek)}}$. We define $A$ to be the matrix whose
$t$ columns are $\iota(P_1),\dots,\iota(P_t)$. 

A basis of $V_i$ is given by $X_i^{Dk}P_1,\dots,X_i^{Dk}P_t$. Hence we
may realize a basis for $\iota(V_i)$ as the columns of the product
$BA$  where 
\begin{equation*}
B\in\mat{{{n+(D+1)k}\choose n}{r+Ek\choose r},{{n+k}\choose n}{r+Ek\choose r}}{\IQbar}
\end{equation*}
is a transformation matrix.
The image of a monomial under the homomorphism $P\mapsto X_i^{Dk}P$ is
also a monomial. So each row of $B$ has at most one
non-zero entry. This entry is of the form
\begin{equation}
\label{eq:Bentry}
{ k \choose \alpha}^{1/2} {(D+1)k \choose \alpha'}^{-1/2}
\end{equation}
here $\alpha,\alpha' \in \IN_0^{n+1}$ correspond to monomials; they 
satisfy $|\alpha|_1 = k$ and $|\alpha'|_1 = (D+1)k$. 

Let $K\subset\IQbar$ be a number field containing the finitely many
algebraic numbers
 which
appear in this proof and let $v\in\pl{K}$.

Say $v$ is infinite. The absolute value of
(\ref{eq:Bentry}) with respect to $v$ is at most $(n+1)^{k/2}$. If $B'$
is a $t\times t$ submatrix of $B$, then $|\det B'|_v \le
(n+1)^{tk/2}$. The number of possibilities for $B'$ is 
\begin{alignat*}1
{ {{n+(D+1)k}\choose n}{r+Ek\choose r} \choose t} { {{n+k}\choose
    n}{r+Ek\choose r} \choose t}
&\le {{n+(D+1)k}\choose n}^t {{n+k}\choose
    n}^t {r+Ek\choose r}^{2t} \\
&\le (n+(D+1)k)^{nt} (n+k)^{nt} (r+Ek)^{2rt}.
\end{alignat*}
The triangle inequality implies
\begin{equation*}
\lh{v,t}{B} \le \frac{tk}{2} \log (n+1) + \frac{nt}{2}
\log((n+(D+1)k)(n+k))
+rt \log(r+Ek). 
\end{equation*}

Say $v$ is  finite place and let $p$ be the rational prime with
$|p|_v<1$.  Statement (\ref{eq:multinomial}) gives
\begin{equation*}
\left|{ k \choose \alpha}^{1/2} {(D+1)k \choose
  \alpha'}^{-1/2}\right|_v\le
\left\{
\begin{array}{cl}
  ((D+1)k)^{(n+1)/2} &:\text{if }p\le (D+1)k, \\
  1 &:\text{else wise.}
\end{array}
\right.
\end{equation*}
From the description of the
entries of $B$ given around (\ref{eq:Bentry}) we deduce that 
\begin{equation*}
\lh{v,t}{B} \le 
\left\{
\begin{array}{cl}
  t\frac{n+1}{2}  \log((D+1)k) &:\text{if }p\le (D+1)k,\\
  1 &:\text{else wise.}
\end{array}
\right.
\end{equation*}

Now for arbitrary $v$ we have $\lh{v,t}{BA}\le
\lh{v,t}{B}+\lh{v,t}{A}$ by Lemma \ref{lem:matrixlh}. We multiply the
 local heights with the corresponding local degrees and sum over all
places of $K$ to obtain 
\begin{alignat*}1
\matheight{t}{BA} \le &\frac{tk}{2}(n+1) 
+ \frac{nt}{2} \log((n+(D+1)k)(n+k)) + rt \log(r+Ek) \\
&+t\frac{n+1}{2}\pi((D+1)k)\log((D+1)k)
 +\matheight{t}{A},
\end{alignat*}
where
$\pi((D+1)k)$ denotes the number of rational primes at most $(D+1)k$. 
It is
known that $\pi((D+1)k)\log ((D+1)k) \le 2(D+1)k$, see
for example the work of Rosser and Sch\"onfeld \cite{RS}. 
  So 
\begin{alignat*}1
\matheight{t}{BA} &\le \frac{tk}{2}(n+1) + \frac{nt}{2}
\log((n+(D+1)k)(n+k)) + rt \log(r+Ek) + (n+1)(D+1)tk +\matheight{t}{A}
\\
&\le ntk + \frac{nt}{2} \log(6n^2Dk^2) + rt \log(2rEk) + 4nDtk +
\matheight{t}{A} \\
&\le nt\log(\sqrt{6}nDk) + rt\log(2rEk)+ 5nDtk + \matheight{t}{A} \\
&\le nt\log(\sqrt{6}nD) + ntk + rt\log(2rE)+ rtk+ 5nDtk + \matheight{t}{A}.
\end{alignat*}

We recall (\ref{eq:decomposehW1}), together with $\matheight{t}{BA} =
\vsheight{V_i}$
  this implies
 \begin{equation*}
   \vsheight{W_1}\le
2n^2t\log(\sqrt{6}nD) + 2n^2tk + 2nrt\log(2rE)+  2nrtk+ 10n^2Dtk + 2n\matheight{t}{A}.
 \end{equation*}
Since $\matheight{t}{A} =
\vsheight{\orth{I_{(k,Ek)}}} = \arithhilb{k,Ek;Z}$ 
by Lemma \ref{lem:heightcompare} we have
\begin{equation*}
\vsheight{W_1} \le 2n 
\arithhilb{k,Ek;Z}
+(2n\log(\sqrt{6}nD) +2n+ 2r\log(2rE)+ 2r+10nD)ntk.
\end{equation*}
The lemma follows since $t=\hilbfunc{k,Ek;Z}$
and 
\begin{alignat*}1
2n&\log(\sqrt{6}nD) +2n+ 2r\log(2rE)+ 2r+10nD
\le 2n^2D + 2n +2r^2 E + 2r + 10nD \\
&\le (2n^2 +2n+2r^2 + 2r + 10n)\max\{D,E\}
\le 18\max\{n,r\}^2 \max\{D,E\}.\qedhere
  \end{alignat*}
\end{proof}

\begin{lemma}
\label{lem:hW2}
We have 
\begin{equation*}
\vsheight{W_2} \le  n 
\arithhilb{(D+1)k,Ek;Z}+n\hilbfunc{(D+1)k,Ek;Z}.
\end{equation*}
\end{lemma}
\begin{proof}
For brevity set $ t= \hilbfunc{(D+1)k,Ek;Z}$ and let $P_1,\dots,
P_t$ be a basis of $\orth{I_{((D+1)k,Ek)}}$. 

We can write $W_2$ as 
\begin{alignat*}1
\{&(G_1,G'_1,\dots,G_n,G'_n) \in R_{((D+1)k,Ek)}^{2n};\,\,\langle
G_i,P_j\rangle =
 \langle G'_i,P_j\rangle 
\text{ for }1\le i\le n,\, 1\le j\le t\}.
\end{alignat*}

Let $A\in\mat{N,t}{\IQbar}$ be the matrix whose columns are
${\iota(\tau(P_1))},\dots,\iota(\tau(P_t))$; here $N={ n+(D+1)k \choose n}{r+Ek
  \choose r}$. Then $\iota(W_2)$
is the kernel of the rank $nt$ matrix
\begin{equation*}
B = \left( 
\begin{array}{ccccc}
\trans{A}   & -\trans{A} &  & & 0\\
& & \ddots &  &\\
0 & & & \trans{A} & -\trans{A}
\end{array}\right)
\in\mat{nt,2nN} {\IQbar}.
\end{equation*}

Let $K\subset\IQbar$ be a number field containing the finitely many
algebraic numbers
 which
appear in this proof and let $v\in\pl{K}$.

If $v$ is  infinite and $\sigma = \sigma_v:K\rightarrow\IC$
then the Cauchy-Binet formula implies
$\exp{2 \lh{v,nt}{B}} =  \det ({\overline{\sigma(B)}}
\trans{\sigma(B)})$.
 We obtain a block-diagonal matrix
\begin{equation*}
{\overline{\sigma(B)}} \trans{\sigma(B)}=
\left(\begin{array}{ccc}
2 \trans{\overline{\sigma(A)}} \sigma (A) & & 0  \\
& \ddots &  \\
0 & & 2 \trans{\overline{\sigma(A)}} \sigma (A)
\end{array}\right).
\end{equation*}
In total there are $n$ blocks and so 
$\det ({\overline{\sigma(B)}}\trans{\sigma(B)}) = 2^{nt} \det (\trans{\overline{\sigma(A)}} \sigma
(A))^n$. If we apply the Cauchy-Binet formula again we arrive at 
\begin{equation}
\label{eq:lhW2inf}
\lh{v,nt}{B} =  n\lh{v,t}{A} +\frac{nt}{2} \log 2.
\end{equation}

Say $v$ is  finite.  We fix a $t\times t$ submatrix $A'$ of
$\trans{A}$ with $|\det A'|_v$  maximal. In particular, $\det
A'\not=0$. It follows  that
$\lh{v,t}{A'^{-1}\trans{A}}=\lh{v,t}{\trans{A}} - \log|\det A'|_v = 0$.
We observe that the entries of $A'^{-1} \trans{A}$ are $v$-integers.
 Consider the rank $nt$ matrix
\begin{equation}
\label{eq:CBprod}
C =  \underbrace{\left(\begin{array}{ccc}
A'^{-1} & & 0\\
& \ddots & \\
0 & & A'^{-1}
\end{array}\right)}_{\text{$n$ blocks}}B=
\left(\begin{array}{ccccc}
A'^{-1}\trans{A} & -A'^{-1}\trans{A} & & & 0\\
& & \ddots & & \\
0 & & & A'^{-1}\trans{A} & -A'^{-1}\trans{A}
\end{array}\right).
\end{equation}
We have $\lh{v,nt}{C}\le 0$.
The block matrix in the middle  of (\ref{eq:CBprod}) has determinant $(\det A')^{-n}$.
 This
shows the  equality in 
\begin{equation*}  
\lh{v,nt}{B}-n\log |\det A'|_v = \lh{v,nt}{C}\le 0.
\end{equation*}
We obtain 
\begin{equation}
\label{eq:lhW2fin}
\lh{v,nt}{B} \le n\lh{v,t}{A}.
\end{equation}

We multiply (\ref{eq:lhW2inf}) and (\ref{eq:lhW2fin})
with $[K_v:\IQ_v]/[K:\IQ]$
 and sum over all places 
to obtain 
\begin{equation*}
\matheight{nt}{B} \le   n \matheight{t}{A}+\frac{nt}{2}\log 2. 
\end{equation*}

Now $\matheight{t}{A}$ is the height of
$\iota(\orth{I_{((D+1)k,Ek)}})$; indeed, applying $\tau$ does not
change the height.  
So $\matheight{t}{A}=\arithhilb{(D+1)k,Ek;Z}$ 
by Lemma \ref{lem:heightcompare}.  
 The lemma follows
because passing to the orthogonal complement (\ref{eq:orthcompl}) does
not change the height; i.e.  $\vsheight{W_2} = \matheight{nt}{B}$.
\end{proof}

We will not bound the height of $W_3$ directly. Rather we construct 
a larger space of controlled height which, 
 together with $W_1$ and $W_2$,
still cuts out
$W'_{DEk}$.

\begin{lemma}
\label{lem:hW3} 
There exists a vector subspace $W'_3$ of $R_{((D+1)k,Ek)}$ such that
\begin{equation*}
W'_{DEk} = W_1\cap W_2\cap W'_3
\end{equation*}
and
\begin{alignat*}1
\vsheight{W'_3} \le 
20n \max\{n,r\}^2 \max\{D,E\}\hilbfunc{k,Ek;Z}k.
\end{alignat*}
\end{lemma}
\begin{proof}
For brevity we set $M= \dim R_{((D+1)k,Ek)} = { (D+1)k+n \choose n}{Ek+r  \choose r}$. 
By the definition of 
$W_3$ we see that $\iota(W_3) \subset \IQbar^{2nM}$ is cut out
  by $n M$
linear equations; for each monomial in $R_{((D+1)k,Ek)}$ we
 need $n$ equations. 
And each  linear equation comes from one equation
\begin{equation*}
X_i^{Dk}  G'_1 - X_1^{Dk} G'_i =  0.
\end{equation*}
With respect to the usual basis, 
 a typical linear equations has coefficients
\begin{equation}
\label{eq:W3matentries}
{(D+1)k \choose \alpha}^{1/2} \quad\text{and}\quad 
-{(D+1)k \choose \alpha'}^{1/2}\quad\text{with}\quad
\alpha,\alpha'\in\IN_0^{n+1}\text{ and }
|\alpha|=|\alpha'| = (D+1)k.
\end{equation}

Among these $nM$ linear equations we can find $t\le \dim W_1$ linearly
independent ones such that if $\iota(W'_3)$ is their common kernel then $W'_{DEk}=
W_1\cap W_2 \cap W'_3$. 
Let $A$ be a $t\times 2nM$ matrix whose  rows are precisely these chosen linear
equations.  Hence the kernel of $A$ is $\iota(W'_3)$
and each row has at most two non-zero entries of the form (\ref{eq:W3matentries}).

Let $K\subset\IQbar$ be a number field containing the finitely many
algebraic numbers
 which
appear in this proof and let $v\in\pl{K}$.


Say  $v$ is  infinite. The absolute value of the multinomials in 
 (\ref{eq:W3matentries}) is at most
$(n+1)^{(D+1)k/2}$. 
Let $A'$ be a $t\times t$ submatrix of
$A$.  By the discussion above we obtain
\begin{equation*}
|\det A'|_v \le (2(n+1)^{(D+1)k/2})^t
\end{equation*}
from the Leibniz formula for the determinant. 
Now the number of possible $t\times t$ submatrices of $A$ is 
\begin{equation*}
{2nM \choose t}\le (2nM)^t  = (2n)^t { (D+1)k+n \choose n}^t{Ek+r  \choose r}^t
\le  (2n)^t ((D+1)k+n)^{nt}(Ek+r)^{rt}. 
\end{equation*}
By definition of the local height of $A$ we get
\begin{equation}
\label{eq:lhW3fin}
\lh{v,t}{A} \le \frac 12 tk(D+1)\log(n+1) + t\log 2 + \frac{t}{2} \big(\log(2n)
+ n\log((D+1)k+n) + r\log(Ek+r)\big).
\end{equation}

If $v$ is finite  then $\lh{v,t}{A}\le 0$ since the
coefficients of $A$ are algebraic integers. 

We use this observation,  multiply (\ref{eq:lhW3fin}) with
$[K_v:\IQ_v]/[K:\IQ]$, and sum over all places to obtain 
\begin{equation*}
\matheight{t}{A} \le 
\frac 12 tk(D+1)\log(n+1) + t\log 2 + \frac{t}{2} \big(\log(2n)
+ n\log((D+1)k+n) + r\log(Ek+r)\big).
\end{equation*}

Now $t\le \dim W_1 = 2n \dim \orth{I_{(k,Ek)}}  = 2n
\hilbfunc{k,Ek;Z}$. So
\begin{equation*}
\matheight{t}{A} \le 
\big(nk(D+1)\log(n+1)  + n (\log(8n)
+ n\log((D+1)k+n) + r\log(Ek+r))\big) \hilbfunc{k,Ek;Z}.
\end{equation*}

Again, height invariance under passing to the orthogonal complement gives
 $\vsheight{W'_3} = \matheight{t}{A}$.
The lemma follows from the following elementary inequalities
\begin{alignat*}1
nk(D+1)\log(n+1)  &+ n (\log(8n)
+ n\log((D+1)k+n) + r\log(Ek+r)) \\
&\le 2n^2 D k + 8n^2 + 2n^2Dk + n^3 + nrEk + nr^2\\
& \le 20n\max\{n,r\}^2 \max\{D,E\}k. \qedhere
\end{alignat*}
\end{proof}

We need precise  estimates for the arithmetic Hilbert function
$\arithhilb{ak,bk;Z}$ at large values  $k$.
Our tools are Proposition \ref{prop:arithhilbbound}  and Zhang's inequality \cite{ZhangArVar} for the essential
minimum for a subvariety of $\IP^n$.
We recall that $s$ and $v_{ab}$ denote the Segre and Veronese morphism, respectively.
Also,  degree and height of a subvariety of $\IP^n\times\IP^r$ is the
degree and height
of its embedding into $\IP^{nr+n+r}$ under the Segre morphism, respectively.

\begin{lemma}
  \label{lem:hXhvX}
We have
\begin{equation*}
  \height{v_{ab}(Z)}  \le (1+d) \max\{a,b\}^{d+1} \height{Z}
\end{equation*}
for $a,b\in\IN$.
\end{lemma}
\begin{proof}
  If $X\subset\IP^n$ is an irreducible closed subvariety defined over
  $\IQbar$, then its
   essential minimum is
  \begin{equation*}
    \essmin{X} = \inf \big\{\theta\in\IR;\,\, \{z\in X(\IQbar);\,\,
    \height{z}\le \theta\}\text{ is Zariski dense in }X \big\}.
  \end{equation*}

Let $\epsilon>0$, by definition
\begin{equation*}
  \left\{z\in s(Z)(\IQbar);\,\,\height{z}\le
  \essmin{s(Z)}+\epsilon \right\}
\quad\text{is Zariski dense in}\quad s(Z).
\end{equation*}
Any $z$ in this set is of the form $s(p,q)$
with $(p,q)\in Z(\IQbar)$. By Section 1.5.14 \cite{BG} we have
$\height{z} = \height{p}+\height{q}$.
An elementary local estimate shows
$\height{v_a(p)}\le a \height{p}$
and $\height{v_b(a)}\le b \height{q}$.
We have $v_{ab}(z) =(v_a(p),v_b(q))$,
so
$\height{s(v_{ab})(z)} = \height{v_a(p)}
+ \height{v_b(q)} \le a\height{p}+b\height{q}\le \max\{a,b\}\height{z}$.
The set of $s(v_{ab}(z))$ thus obtained is Zariski dense in
$s(v_{ab}(Z))$. We get
  $\essmin{s(v_{ab}(Z))} \le \max\{a,b\} (\essmin{s(Z)}+\epsilon)$
for all $\epsilon > 0$. Letting  $\epsilon$ go to zero gives
\begin{equation}
  \label{eq:essminineq}
  \essmin{s(v_{ab}(Z))} \le \max\{a,b\} \essmin{s(Z)}.
\end{equation}

Zhang's inequality states
\begin{equation}
 \label{eq:Zhang}
\frac{\height{X}}{(1+\dim X)\deg{X}} \le \essmin{X}\le
\frac{\height{X}}{\deg{X}}.
\end{equation}

We use his inequality to bound $\essmin{s(v_{ab}(Z))}$ from below
and $\essmin{s(Z)}$ from above. Inequality 
(\ref{eq:essminineq})
implies
\begin{equation}
\label{eq:hdegbound}
\frac{\height{v_{ab}(Z)}}{(1+d) \deg{
    v_{ab}(Z)}} \le \essmin{s(v_{ab}(Z))}
 \le \max\{a,b\} \essmin{s(Z)} 
 \le  \max\{a,b\} \frac{\height{Z}}{\deg{Z}};
\end{equation}
we note $\dim s(v_{ab}(Z)) = d$.
Lemmas \ref{lem:hilbpolyboundsegre} 
and \ref{lem:hilbpolyboundvero}
 give
 $\deg{v_{ab}(Z)} = \hhilbpoly{1,1;v_{ab}(Z)}
=\hhilbpoly{a,b;Z} \le \max\{a,b\}^d \deg{Z}$.
The current lemma follows from (\ref{eq:hdegbound}).
\end{proof}

\begin{lemma} 
\label{lem:arithakbk}
If $k\ge \deg{Z}$, then
\begin{equation*}
\arithhilb{ak,bk;Z} 
\le 2r \max\{a,b\}^{d+1}(\height{Z}+\deg{Z}){d+k\choose d}k
\end{equation*}
for $a,b\in\IN$.
\end{lemma}
\begin{proof}
By Lemma \ref{lem:verobound} we deduce
$\arithhilb{ak,bk;Z} \le \arithhilb{k,k;v_{ab}(Z)}$.
For brevity, we set $X=s(v_{ab}(Z))$ and
 use Lemma \ref{lem:segrebound} to estimate
$\arithhilb{k,k;v_{ab}(Z)} \le \arithhilb{k;X}$. 
Hence
\begin{alignat*}1
\arithhilb{ak,bk;Z} \le \arithhilb{k;X}.
\end{alignat*}

We apply Proposition \ref{prop:arithhilbbound} 
to bound the arithmetic Hilbert function of $X$
and obtain 
\begin{equation}
\label{eq:arithhilbakbkZ}
\arithhilb{ak,bk;Z} \le
\hilbfunc{k;X}\left(k\frac{\height{X}}{\deg{X}}
+ \frac 12 \log \hilbfunc{k;X}\right).
\end{equation}


We continue by bounding the height  of $X$.
Lemma \ref{lem:hXhvX} implies 
 \begin{equation*}
  \height{X} \le (d+1) \max\{a,b\}^{d+1}
\height{Z}.
 \end{equation*}
We insert this into
(\ref{eq:arithhilbakbkZ}) 
to get
\begin{equation}
\label{eq:lastarithbound}
  \arithhilb{ak,bk;X}
\le (d+1)\max\{a,b\}^{d+1}k \frac{\hilbfunc{k;X}}{\deg{X}}\height{Z}
+ \frac 12 \hilbfunc{k;X}\log\hilbfunc{k;X}. 
\end{equation}

 Chardin's bound for the geometric Hilbert function  \cite{Chardin} states
\begin{equation*}
\hilbfunc{k;X} \le \deg{X}{d+k\choose d}.
\end{equation*}
Hence
\begin{equation}
\label{eq:lastarithbound}
  \arithhilb{ak,bk;X}
\le (d+1)\max\{a,b\}^{d+1} \height{Z} {d+k\choose d}k 
+ \frac 12 \deg{X}{d+k\choose d} \log\hilbfunc{k;X}. 
\end{equation}

Lemmas \ref{lem:hilbpolyboundsegre} and \ref{lem:hilbpolyboundvero} give
$\deg{X} = \deg{v_{ab}(Z)} = \hhilbpoly{a,b;Z}
\le \max\{a,b\}^{d}\deg{Z}$.
Elementary estimates imply ${d+k\choose d}\le (d+k)^d/d!
 \le (d+1)^d k^d /d! \le (ek)^d$. 
Hence $\log \hilbfunc{k;X}\le \log(\deg{X}(ek)^d) \le d \log(e\max\{a,b\}\deg{Z} k)$. We
recall $k\ge \deg{Z}$ to see that  
\begin{equation*}
\log\hilbfunc{k;X}\le
 d \log(e\max\{a,b\}k^2)
\le 2d \log(e\max\{a,b\}k)
\le 2d \max\{a,b\}k.
\end{equation*}
We apply this and the bound for $\deg{X}$ from above to
(\ref{eq:lastarithbound}) and obtain
\begin{alignat*}1
  \arithhilb{ak,bk;X} &\le (d+1) \max\{a,b\}^{d+1} \height{Z}
            {d+k\choose d} k
 + d  \max\{a,b\}^{d+1} \deg{Z} {d+k\choose d}k.
\end{alignat*}
The lemma follows since $d+1\le 2d\le 2r$.
\end{proof}

We now bound the height of $W'_{DEk}$ from above explicitly in terms
of $\height{Z},\deg{Z},$ and the parameters $D,E$, and $k$.

\begin{lemma}
\label{lem:boundhW1W2W3}
If $k\ge \deg{Z}$, then
\begin{alignat*}1
  \vsheight{W_1} 
&\le 24 n\max\{n,r\}^2  \max\{D,E\}E^d (\height{Z}+\deg{Z}) {d+k\choose d}k,\\
\vsheight{W_2} &\le
3nr \max\{D+1,E\}^{d+1} (\height{Z}+\deg{Z}) {d+k\choose d}k,\text{ and}\\
\vsheight{W'_3} &\le
20n\max\{n,r\}^2 \max\{D,E\}E^d \deg{Z} {d+k\choose d}k.
\end{alignat*}
\end{lemma}
\begin{proof}
The inequalities
  \begin{alignat}1
    \label{eq:arithkEk2}
    \arithhilb{k,Ek;Z} &\le 2r E^{d+1} (\height{Z}+\deg{Z}){d+k\choose
      d}k\\ 
\nonumber
    \hilbfunc{k,Ek;Z} & \le \hhilbpoly{1,E;Z}{d+k\choose d}
  \end{alignat}
follow from Lemmas \ref{lem:arithakbk} and \ref{lem:hilbfuncub},
respectively.
Lemma \ref{lem:hilbpolyboundsegre} gives the bound $\hhilbpoly{1,E;Z}\le E^d \deg{Z}$, so
\begin{equation}    
\label{eq:geomkEk2}
    \hilbfunc{k,Ek;Z} \le  E^d \deg{Z}{d+k\choose d} .
\end{equation}

 The same lemmas and Lemma \ref{lem:hilbpolyboundsegre} imply
  \begin{alignat}1
    \label{eq:arithDkEk2}
    \arithhilb{(D+1)k,Ek;Z} &\le 2r \max\{D+1,E\}^{d+1}
    (\height{Z}+\deg{Z}) {d+k\choose d}k, \\
    \label{eq:geomDkEk2}
    \hilbfunc{(D+1)k,Ek;Z} &\le 
\hhilbpoly{D+1,E;Z}{d+k\choose d}
\le   \max\{D+1,E\}^d \deg{Z} {d+k\choose d}.
  \end{alignat}

First, we bound $\vsheight{W_1}$. By the estimate in Lemma \ref{lem:hW1} together with
(\ref{eq:arithkEk2}) and (\ref{eq:geomkEk2}) we get
\begin{alignat*}1
  \vsheight{W_1}&\le
4nr E^{d+1}(\height{Z}+\deg{Z}) {d+k\choose d}k + 
20n \max\{n,r\}^2\max\{D,E\} E^d \deg{Z} {d+k\choose d}k \\
&\le (4nr +
20n\max\{n,r\}^2) \max\{D,E\}E^d (\height{Z}+\deg{Z}) {d+k\choose d}k  \\
&\le 24 n\max\{n,r\}^2  \max\{D,E\}E^d (\height{Z}+\deg{Z}) {d+k\choose d}k.
\end{alignat*}
 This is the  bound for $\vsheight{W_1}$ in the lemma.

Next, we bound $\vsheight{W_2}$. For this we need Lemma \ref{lem:hW2} which,
together with the bounds (\ref{eq:arithDkEk2}) and (\ref{eq:geomDkEk2}), gives
\begin{alignat*}1
\vsheight{W_2} &\le 2nr \max\{D+1,E\}^{d+1}(\height{Z}+\deg{Z})
         {d+k\choose d}k +
n\max\{D+1,E\}^d \deg{Z} {d+k\choose d} \\
&\le 3nr \max\{D+1,E\}^{d+1} (\height{Z}+\deg{Z}) {d+k\choose d}k
\end{alignat*}
and so the second bound in the assertion.

Finally, we give a bound for $\vsheight{W_3'}$. Applying Lemma \ref{lem:hW3} and
(\ref{eq:geomkEk2}) leads to the final bound
\begin{equation*}
\vsheight{W_3'} \le 20n \max\{n,r\}^2 \max\{D,E\} E^d \deg{Z} {d+k\choose d}k.
\qedhere
\end{equation*}
\end{proof}

By Lemmy \ref{lem:hW3}, the vector space $W'_{DEk}$ is the intersection
$W_1\cap W_2\cap W'_3$, this enables us to bound its height.

\begin{lemma}
\label{lem:hWDEkub}
  If $k\ge \deg{Z}$, then
  \begin{equation*}
    \vsheight{W'_{DEk}} \le 
50 n \max\{n,r\}^2 \max\{D+1,E\}^{d+1} (\height{Z}+\deg{Z})
{d+k\choose d}k.
  \end{equation*}
\end{lemma}
\begin{proof}
  A theorem of Schmidt, cf. (\ref{eq:schmidtineq}) in Remark \ref{eq:schmidtineq}, implies
  $\vsheight{W'_{DEk}} \le
  \vsheight{W_1}+\vsheight{W_2}+\vsheight{W'_3}$. 
Adding the bounds provided by the previous lemma
leads to the desired inequality.
\end{proof}

\subsection{Proof of Propositions \ref{prop:heightlb} and \ref{prop:heightlb2}}
We continue working with the notation introduced in the preceeding subsections.
We recall
that
$\kappa(Z)$ was defined in (\ref{eq:definekappa}).

We will need some preparatory estimates. 

\begin{lemma}
  We have 
\begin{equation}
\label{eq:kappadegbound}
  \kappa(Z)\le \deg{Z}.
\end{equation}
\end{lemma}
\begin{proof}
Indeed, Lemma \ref{lem:hilbpolyboundsegre} and (\ref{eq:hilbpolydelta})
imply $\deg{Z} = \sum_{i=0}^d{d \choose i}\Delta_i(Z) \ge \Delta_0(Z)$.
We may assume $\kappa(Z) \ge 1$.
By definition there is $1\le i\le d$ with
$\Delta_i(Z)\not=0$ and $\kappa(Z) =
(\Delta_0(Z)/\Delta_i(Z))^{1/i} \le \Delta_0(Z) / \Delta_i(Z)
\le\Delta_0(Z)\le \deg{Z}$ and our lemma holds.
\end{proof}

The next lemma bounds the value of $\Psi$, defined near (\ref{eq:defPsi}), from
above. 

\begin{lemma}
\label{lem:Psibound}
Let $G = (G_1,G'_1,\dots,G_n,G'_n)\in \prod_{i=1}^n V_0\times V_i$ be non-zero, then 
$\height{\Psi(G)} \le \height{G} +  (\frac{D+1}{2} \log(n+1) + n +
1)k$. 
\end{lemma}
\begin{proof}
By the convention introduced in Section \ref{sec:heights},
the height of $G$ is  the height  of
  $\iota(G)$.

Tracing through the definition of $\iota$ we see that
  each coordinate of $\iota(\Psi(G))$ is some coordinate of $\iota(G)$
  times a factor of the form
\begin{equation}
\label{eq:PsiGcoeff}
{(D+1)k \choose \alpha}^{1/2} {k \choose \alpha'}^{-1/2}
\quad\text{with}\quad \alpha,\alpha'\in \IN_0^{n+1} \text{ and }
|\alpha|_1 = (D+1)k, |\alpha'|_1 = k. 
\end{equation}

Let $K\subset\IQbar$
 be a number field containing all algebraic numbers which
appear below and let $v\in \pl{K}$.

Say $v$ is  infinite. Then the absolute value 
 of the expression  (\ref{eq:PsiGcoeff}) with respect to $v$ 
is bounded by
$(n+1)^{(D+1)k/2}$. It follows that 
\begin{equation}
\label{eq:PsiGinf}
 | \Psi(G)|_v \le (n+1)^{(D+1)k/2} |G|_v
\end{equation}

Now say if $v$ is finite and let $p$ be the rational prime with
$|p|_v<1$. 
The $v$-adic absolute value 
of  (\ref{eq:PsiGcoeff}) is at most
$k^{(n+1)/2}$ if $p\le k$ and at most $1$ if $p>k$, cf. (\ref{eq:multinomial}).
 It follows that 
\begin{equation}
\label{eq:PsiGfin}
\log | \Psi(G)|_v \le 
\log |G|_v+
\left\{\begin{array}{cl}
\frac{n+1}{2} \log k  &: \text{if } p\le k, \\
0 &: \text{else wise.}
\end{array}
\right.
\end{equation}

Multiplying the expressions (\ref{eq:PsiGinf}) and (\ref{eq:PsiGfin})
with the corresponding local degrees and taking the sum over all
places gives
\begin{alignat*}1
\height{\Psi(G)} &\le \height{G}+ \frac{D+1}{2}k \log(n+1) +
\frac{n+1}{2} \pi(k) \log k 
\end{alignat*}
here, as in the proof of Lemma \ref{lem:hW1}, $\pi(k)$ denotes
 the number of rational primes at most
$k$. 
We have $\pi(k)\log k \le 2k$ and this
 completes the proof.
\end{proof}

The following absolute version  of Siegel's Lemma  is  due to
Zhang.
We could have also refered to the version of Roy and Thunder \cite{RTAbsSiegel}. 

\begin{lemma}
\label{lem:Siegel}
Suppose $\dim W'_{DEk} \ge 2$. There exists
a non-zero  $G\in W'_{DEk}$  such that 
\begin{equation*}
  \height{G} \le  \frac{\vsheight{W'_{DEk}}}{\dim W'_{DEk}} + \log
  \dim W'_{DEk}. 
\end{equation*}
\end{lemma}
\begin{proof}
 This is a consequence of David and Philippon's Lemme 4.7 \cite{MinorSousVarieties} which
 is based on a result of Zhang.
We added the artificial hypothesis $\dim W'_{DEk}\ge 2$  to avoid the
$\epsilon$ in the reference.
\end{proof}

This variant of Siegel's Lemma is needed in the next lemma. 

\begin{lemma}
\label{lem:applySiegel}
  We assume
  \begin{equation*}
    \frac{D+1}{E} \le \frac{1}{4dn} \kappa(Z) 
  \end{equation*}
and
\begin{equation*}
 k \ge \max\{[17 n d! e^d \Delta_0^{d-1}],\deg{Z} \}.
\end{equation*}
There exists a non-zero $F\in W_{DEk}$
such that 
\begin{equation*}
  \frac{1}{Ek} \height{F}\le  400n \max\{n,r\}^2
\frac{\max\left\{1,\kappa(Z)\right\}^{d+1}}{\Delta_0} 
  (\height{Z}+\deg{Z}) + 5d \deg{Z}.
\end{equation*}
\end{lemma}
\begin{proof}
Recall that $\Psi$ given by (\ref{eq:defPsi}) defines an isomorphism
between  $W'_{DEk}$ and $W_{DEk}$. 

Let us assume that 
$k$ is as in the hypothesis.
Then $k \ge 16nd! e^d \Delta_0^{d-1}$ and so
\begin{equation*}
\frac 12 \frac{k^d}{d!} - 4n  e^d \Delta_0^{d-1}k^{d-1} \ge 
\frac 14  \frac{k^d}{d!}.
\end{equation*}
Lemma  \ref{lem:dimlb} implies
\begin{equation*}
  \dim W'_{DEk} = \dim W_{DEk} \ge \frac 14 \Delta_0 E^d \frac{k^d}{d!};
\end{equation*}
this implies $\dim W'_{DEk} \ge 2$ by our choice of $k$.

An elementary calculation shows ${d+k\choose d}\le 2 k^d / d!$ since
 $k\ge 16 d!$. So
\begin{equation*}
  \dim W'_{DEk} \ge \frac 18 \Delta_0 E^d {d+k \choose d}.
\end{equation*}
Because of  Lemmas \ref{lem:Siegel} and  \ref{lem:hWDEkub}
there exists a non-zero
$G\in W'_{DEk}$  with
\begin{alignat*}1
  \height{G} &\le \frac{\vsheight{W'_{DEk}}}{\dim W'_{DEk}}  + \log
  \dim W'_{DEk}\\
 &\le 
400 n \max\{n,r\}^2 \frac{\max\{D+1,E\}^{d+1}}{\Delta_0 E^d}
(\height{Z}+\deg{Z}) k + \log \dim W'_{DEk}.
\end{alignat*}

We need to bound $\dim W'_{DEk} = \dim W_{DEk}$ from above in order to get a handle on the
logarithm. Indeed, by definition (\ref{eq:defWDEk}) we have 
$ \dim W'_{DEk} \le (n+1)\hilbfunc{k,Ek;Z}$.
Just as around (\ref{eq:geomkEk2}) in the proof of Lemma \ref{lem:boundhW1W2W3}
we have $\hilbfunc{k,Ek;Z}\le  E^d \deg{Z}{d+k\choose d}\le 2E^d
\deg{Z} k^d/d!$, so $\log \dim W'_{DEk}\le 
\log (4n E^d \deg{Z} k^d)\le d \log(4nE\deg{Z}k)$. We obtain
\begin{alignat*}1
  \frac {d\log (4nE\deg{Z}k)}{Ek} 
&= \frac{d\log(4n)}{Ek} + \frac{d\log E}{Ek} + \frac{d\log \deg{Z}}{Ek}
+\frac{d \log k}{Ek}
 \le 1 + 1 + d + d \le 4d
\end{alignat*}
from $k\ge 16 dn$ and $k\ge \deg{Z}$.
Thus
\begin{equation*}
 \log \dim W_{DEk} \le 4dEk.
\end{equation*}

We set $F=\Psi(G)\not=0$. Its height is bounded above by Lemma
\ref{lem:Psibound},
we obtain the estimate
\begin{alignat*}1
 \height{F} &\le  \height{G}+
\left(\frac{D+1}{2} \log(n+1) + n+1\right)k  \le \height{G} +3nDk\\
&\le  400 n\max\{n,r\}^2 \frac{\max\{D+1,E\}^{d+1}}{\Delta_0 E^d} 
  (\height{Z}+\deg{Z}) k + 4d E k + 3nDk.
\end{alignat*}
We use the bound $(D+1)/E \le \kappa(Z)/(4dn) \le\kappa(Z)$
to estimate
\begin{alignat*}1
\frac{1}{Ek} \height{F} 
\le  400n\max\{n,r\}^2 \frac{\max\{1,\kappa(Z)\}^{d+1}}{\Delta_0}
  (\height{Z}+\deg{Z}) + 4d + \kappa(Z).
\end{alignat*}

Inequality (\ref{eq:kappadegbound}) implies
 $4d+\kappa(Z)\le 5d\deg{Z}$ and the lemma follows.
\end{proof}

\begin{proof}[Proof of Proposition \ref{prop:heightlb}]
In order to prove the proposition we may assume that no projective
coordinate $X_i$ vanishes identically on $Z$. Indeed, otherwise we are
in case (i).

Let us assume for the moment that $\kappa < 17dn$.
For $Q \ge 1$ there are integers $x,y\in\IZ$ with $1\le y \le Q$ such
that
$|y \kappa/(8dn)-x|\le Q^{-1}$; this follows easily from Dirichlet's
Box Principle as employed on the first page of 
 Cassels' book \cite{Cassels}. 

We take $Q = \max\{1,
16dn / \kappa\}$. Then $|x/y| \le \kappa/(8dn) + 1/(yQ) \le \kappa/(8dn) +
\kappa/(16dn) \le \kappa/(4dn)$ and 
$|x/y| \ge \kappa/(8dn)-1/(yQ) \ge \kappa/(8dn) - \kappa/(16dn)
= \kappa/(16dn)$.
 Finally,
$x\ge 1$; indeed, otherwise we would have
$y\kappa/(8dn)-x \ge \kappa/(8dn) > Q^{-1}$ and this is a contradiction.

We set $D = 4x-1$ and $E=4y$. 
Hence  $D,E$ are positive integers with
\begin{equation}
\label{eq:defineDE}
D\ge 3,\,\,
E\le 4\max\{1,16 dn/\kappa\},\,\,\text{and}\,\,
\frac{\kappa}{16dn} \le \frac{D+1}{E}\le \frac{\kappa}{4dn}.
\end{equation}

If $\kappa \ge 17dn$, we set  $D =
[\kappa/(4dn)]-1$ and $E=1$. So they also
satisfy all inequalities in (\ref{eq:defineDE}). 

We  pick $k =\max\{[17 n d! e^d \Delta_0^{d-1}],\deg{Z} \}  $
in accordance with Lemma \ref{lem:applySiegel}. Let
 $F=(F_0,\dots,F_n)\in W_{DEk}$ be as provided by this lemma. 
We note that 
\begin{equation*}
  Ek \le k_0.
\end{equation*}

Some $F_i$ is non-zero.
As an element of $\orth{I_{(k,Ek)}}$, it does not vanish identically
on $Z$ by what was stated in
Remark \ref{rem:posdef}.
We have assumed that  $X_0$ does not vanish identically on $Z$.
It follows from (\ref{eq:defWDEk}),  the definition  of $W_{DEk}$,
that $F_0$ does not vanish identically on $Z$.
We take $F$ in the assertion of the current proposition to be $F_0$.
Note that it is bihomogeneous of bidegree
$(a,b)=(k,Ek)$.

The desired bounds for $a$ and $b$ follow from our choice of $k$. 
Moreover, if $\kappa \ge 17dn$, then $E=1$
and so $a=b$.
The bound for $\height{F_0}$ follows from $\height{F_0}\le \height{F}=
\height{F_0,\dots,F_n}$, from our choice of $k$, and from
Lemma \ref{lem:applySiegel}.

Say $(p,q)\in Z(\IQbar)$ and let us assume that we are not in case (i)
of the proposition.
We are in the position to apply Lemma \ref{lem:heightlb} with
$(a,b)$ as above and $c=Dk$. We obtain
\begin{equation*}
  \frac{D-1}{E} \height{p} \le\height{q} + \frac{1}{Ek}\height{F}
+ \frac{D}{2E}\log(n+1)\le\height{q} + \frac{1}{Ek}\height{F}
 + \frac{\kappa}{8d}
\end{equation*}
where we used $D/E \le \kappa/(4dn)$ from (\ref{eq:defineDE}) and
$\log(n+1)\le n$.

The bound for $\height{F}$ given by Lemma \ref{lem:applySiegel} implies
\begin{equation*}
  \frac{D-1}{E} \height{p}\le \height{q} + 
400 n\max\{n,r\}^2
\frac{\max\left\{1,\kappa\right\}^{d+1}}{\Delta_0} 
  (\height{Z}+\deg{Z}) + 5d \deg{Z}+ \frac{\kappa}{8d}.
\end{equation*}

Since $D\ge 3$ we have
$(D-1)/E \ge (D+1)/(2E)$ and so 
$(D-1)/E \ge \kappa/(32dn)$ by  (\ref{eq:defineDE}). Hence
\begin{equation*}
\kappa \height{p}\le 2^5 dn \height{q} + 
12800 dn^2\max\{n,r\}^2 \frac{\max\{1,\kappa\}^{d+1}}{\Delta_0}
(\height{Z}+\deg{Z}) + 160 d^2n \deg{Z} + 4n\kappa.
\end{equation*}
The inequality in part (ii) of the assertion  follows from $d\le r$
and $\kappa\le \deg{Z}$, cf. (\ref{eq:kappadegbound}).
\end{proof}


\begin{lemma}
\label{lem:pushforward}
  Let $a\in\IN$ and say $P\in \IQbar[{\bf X},{\bf Y}]_{(a,a)}$.
 There is $Q\in \IQbar[{\bf U}]_a$ with
$s^*(Q) = P$ and 
  \begin{equation*}
    \height{Q}\le \height{P} + (n+r+2)a.
  \end{equation*}
\end{lemma}
\begin{proof}
Just for this proof
we
let $\langle\cdot,\cdot\rangle$ be the inner product
attached to the identity on $\IQbar$
as in Remark
\ref{rem:defineip}. 

We know that $s^*:\IQbar[{\bf U}]_a\rightarrow \IQbar[{\bf
    X},{\bf Y}]_{(a,a)}$ is a surjective isometry by
Lemma \ref{lem:segreisometry}. It follows that $\ker s^* \cap \orth{(\ker s^*)}
= 0$ by Remark \ref{rem:isokernel}. So
$\ker s^* + \orth{(\ker s^*)} = \IQbar[{\bf U}]_a$. Therefore, 
 $s^*|_{\orth{(\ker s^*)}}: \orth{(\ker s^*)}\rightarrow 
\IQbar[{\bf X},{\bf Y}]_{(a,a)}$ is surjective. 
There is $Q\in \orth{(\ker s^*)}$ with $s^*(Q)=P$.
We write $Q = \sum_{\gamma} Q_\gamma {\bf U}^\gamma$ where $\gamma$
runs over elements in $\IN_0^{(n+1)(r+1)}$ with $|\gamma|_1 = a$.
It remains to bound the height of $Q$ from above.

 Let $K\subset\IQbar$
 be a finite normal extension of $\IQ$
 containing all of the finitely many  algebraic numbers that appear in
 the proof below.  Let $v\in \pl{K}$.


If $v$ is finite, there is $\gamma_0\in \IN_0^{(n+1)(r+1)}$ 
with $|Q|_v = \left| {a \choose \gamma_0}^{-1/2} Q_{\gamma_0}\right |_v$.
 We have ${a\choose \gamma_0}^{-1} Q_{\gamma_0} = 
\langle Q, {\bf U}^{\gamma_0}\rangle = \langle P,s^*({\bf
  U}^{\gamma_0})\rangle$ since $s^*$ is an isometry. This inner product equals
${a\choose \alpha}^{-1}{a\choose \beta}^{-1} P_{\alpha\beta}$
with $\alpha\in\IN_0^{n+1}$ and $\beta\in\IN_0^{r+1}$ and
where $P_{\alpha\beta}$ is a coefficient of $P$. Hence
\begin{equation*}
  |Q|_v = \left|{a\choose \gamma_0}^{-1/2} Q_{\gamma_0}\right|_v = 
 \left|{a\choose \gamma_0}^{1/2} {a\choose \alpha}^{-1}
{a\choose \beta}^{-1} P_{\alpha\beta}\right|_v
\le \left|{a\choose \alpha}
{a\choose \beta}\right|_v^{-1/2} |P|_v.
\end{equation*}
Together with (\ref{eq:multinomial}) we have
\begin{equation}
\label{eq:QPfinite}
  |Q|_v \le |P|_v\cdot\left\{
  \begin{array}{cl}
    a^{(n+r+2)/2} &:\text{ if }p\le a,\\
    1 &:\text{ else wise}
  \end{array}
\right.
\end{equation}
where $p$ is the rational prime with $|p|_v<1$.


Recall that $\tau$ is complex conjugation restricted to $\IQbar$. 
If $v$ is infinite and $\sigma = \sigma_v$, then there is an
automorphism $\eta$ of $K$ such that $\sigma\circ \eta =
\tau\circ \sigma$. We have
$|Q|^2_v = \trans{\iota(\sigma(Q))} \cdot{\tau(\iota(\sigma(Q)))} =
\sigma(\langle Q,\eta(Q)\rangle)$.
But $Q\in \orth{\ker {s^*}}$ and $s^*$ is an isometry, so
$|Q|^2_v = \sigma(\langle s^*(Q), s^*(\eta(Q))\rangle )
= \sigma(\langle s^* Q,\eta(s^*Q)\rangle)
= \sigma(\langle P, \eta(P)\rangle )
 = |P|_v^2$.
Therefore, 
\begin{equation}
  \label{eq:QPinfinite}
|Q|_v = |P|_v.
\end{equation}

We multiply the logarithm of (\ref{eq:QPfinite}) and (\ref{eq:QPinfinite}) with
$[K_v:\IQ_v]/[K:\IQ]$ and sum over all places of $K$ to obtain
\begin{equation*}
  \height{Q}\le \height{P} + \frac{n+r+2}{2} \pi(a) \log a
\end{equation*}
where $\pi(a)$ is the number of rational primes at most $a$. The lemma
follows from 
$\pi(a) \log a \le 2a$, an inequality we have already seen twice.
\end{proof}

We now prove Proposition \ref{prop:heightlb2}.

For brevity set $\hpd{Z} = \height{Z}+\deg{Z}$.

Let $a,b,$ and $F$ be as in Proposition \ref{prop:heightlb};
we note that $b=a$ since $\kappa \ge 17dn$ by hypothesis.

Say $(p,q)$ is as in the proposition. If some projective coordinate of
$p$ vanishes or if $p$ is not isolated in $\pi_1|_Z^{-1}(p)$, then we
are in case (i). So let us assume the contrary. 

If $F(p,q)\not=0$ then we are in case (ii) of Proposition
\ref{prop:heightlb}. The height bound for $\height{p}$ 
implies
\begin{alignat*}1
\kappa \height{p}&\le 2^5 dn  {\height{q}} + 
2^{14}n^2 r \max\{n,r\}^2 \frac{\kappa^{d+1}}{\Delta_0(Z)}
\hpd{Z} + {2^8 n r^2}\deg{Z}
\\
&\le 2^5 dn \height{q}
+ 2^{14}  n^2 r \max\{n,r\}^2
\left(\frac{\kappa^{d+1}}{\Delta_0(Z)}
+ 1\right)\hpd{Z}.
\end{alignat*}
The inequality in part (iii) now follows easily.

It remains to treat the case $F(p,q)=0$. Then $F$ will determine the
obstruction variety $V$ as follows. Let $\widetilde
V_1,\dots,\widetilde V_N$ be the
irreducible components 
 of the intersection of $Z$ with the zero set of $F$.
Of course, these are independent of $(p,q)$.
Then $\dim \widetilde V_i = d - 1$ because $F$ does not vanish
identically on $Z$.
We may omit those $V_i$ for which $\pi_1|_{\widetilde V_i}$ has
 degree zero since $(p,q)$ is isolated in ${\pi_1}^{-1}|_{Z}(p)$. 
Let us set $V_i=\pi_1(\widetilde V_i)$. 
Our point $p$ is contained in some $V_i$.

By the Fiber Dimension Theorem, $V_i$ has dimension $d-1$.

We apply Lemma \ref{lem:pushforward}
 to obtain $Q\in\IQbar[{\bf U}]_a$ with $s^*(Q) = F$ and
 \begin{equation}
\label{eq:hQbound}
 \height{Q}\le \height{F} + (n+r+2)a.  
 \end{equation}

The images $s(\widetilde V_i)$ are  irreducible components of the
intersection of $s(Z)$ with the zero-set of $Q$.
By B\'ezout's Theorem, cf. Example 8.4.6 \cite{Fulton}, we estimate 
$\sum_{i=1}^N \deg{\widetilde V_i} = \sum_{i=1}^N \deg{s(\widetilde V_i)}\le a
\deg{s(Z)} = a\deg{Z}$.

By Lemma \ref{lem:hilbpolyboundsegre} we have $\Delta_{d-1}(\widetilde V_i) \le
\deg{\widetilde V_i}$ for $1\le i\le N$.
The projection formula  implies
\begin{equation}
\label{eq:degVibound}
\deg{V_i} = \deg{\pi(\widetilde{V_i})} \le
\Delta_{d-1}(\widetilde{V_i}) \le \deg{\widetilde{V_i}}
\end{equation}
so
\begin{equation*}
\sum_{i=1}^N \deg{V_i} \le
\sum_{i=1}^N\deg{\widetilde V_i}\le
a \deg{Z}  
\end{equation*}
The bound for $a$ from Proposition \ref{prop:heightlb} leads to the
bound (\ref{eq:degreebound}).
 
It remains to bound each $\height{V_i}$ from above. For brevity se
 set $V=V_i$ and $\widetilde V = \widetilde{V_i}$ for some valid $i$.

Let $K\subset\IQbar$ be a number field containing all of the finitely
many  algebraic
numbers that appear in the proof below and let $v\in\pl{K}$.

If $v$ is infinite, then the Cauchy-Schwarz inequality implies $|\sigma_v(Q)(x)|\le |\iota(Q)|_v |x|_v^{\deg Q}$
for all $x\in \IC^{(n+1)(r+1)}$. 
Therefore,
\begin{equation*}
  \sum_{v\in\pl{K}} \frac{[K_v:\IQ_v]}{[K:\IQ]}
\mahler{v}{Q} \le \height{Q} + \deg{Q} \sum_{j=1}^{nr+n+r} \frac{1}{2j}.
\end{equation*}
with $\mahler{v}{Q}$ as in Section \ref{sec:chowform}.
We recall $\deg{Q}=a$.
If $W$ is the hypersurface defined by  $Q$ in $\IP^n$, then $\deg{W}=a$ and
$\height{W}\le \height{Q} + a \sum_{i=1}^{nr+n+r} \sum_{j=1}^i \frac{1}{2j}$
by the comment on the top of page 347 \cite{HauteursAlt3}.
We deduce
\begin{alignat}1
\label{eq:heightWub}
  \height{W}&\le \height{Q} + \frac{a}{2}
  \sum_{i=1}^{nr+n+r}(1+\log i)
\le \height{Q}+\frac{a}{2} (n+1)(r+1)\log((n+1)(r+1))
\\ 
\nonumber
&\le \height{Q}+2anr \log(4nr)
\le \height{F} + (n+r+2)a+2anr\log(4nr) \\
\nonumber
&\le \height{F} + 6 anr \log(4nr)
\end{alignat}
using the bound for $\height{Q}$ from (\ref{eq:hQbound}).
Since $s(\widetilde V)$ is an irreducible component of $s(Z)\cap W$
we have
\begin{equation*}
\height{\widetilde V}=  \height{s(\widetilde V)} 
\le a\height{Z}  + \height{W}\deg{Z} + c a\deg{Z}
\end{equation*}
by the Arithmetic B\'ezout Theorem, Th\'eor\'eme 3
\cite{HauteursAlt3};
here 
\begin{alignat*}1
c &\le (n+1)(r+1)\log 2 + \sum_{i=0}^{nr+n+r}\sum_{j=0}^{nr+n+r}
\frac{1}{2(i+j+1)}  \\
&\le (n+1)(r+1)\log 2 + \frac 12 (n+1)(r+1)
\sum_{j=1}^{(n+1)(r+1)} \frac 1j.  
\end{alignat*}

This implies
\begin{equation}
\label{eq:heighttVbound}
  \height{\widetilde V}\le a\height{Z} + \height{W}\deg{Z}
+ 8anr\log(4nr)\deg{Z}.
\end{equation}

By definition, for any $\epsilon > 0$ there is a Zariski dense set of points
$(p',q')\in \widetilde V(\IQbar)$ with 
$\height{p'}+ \height{q'} = \height{s(p',q')}\le
\essmin{s(\widetilde V)}+\epsilon$. The resulting set of $p'$ lies Zariski
dense in $V$, hence $\essmin{V}\le \essmin{s(\widetilde V)}$
after letting $\epsilon$ go to $0$.
Zhang's inequality (\ref{eq:Zhang}) implies
$\height{V}\le (\dim V+1) \frac{\deg V}{\deg{\widetilde V}}
\height{\widetilde V} \le d\frac{\deg V}{\deg{\widetilde V}}
\height{\widetilde V}$.
Using (\ref{eq:degVibound}) we get
$\height{V}\le d \height{\widetilde V}$. Moreover, (\ref{eq:heighttVbound})
gives
\begin{alignat*}1
\height{V}&\le ad \height{Z} + d \height{W}\deg{Z}
+ 8ad nr\log(4nr)\deg{Z}  \\
&\le 14adnr \log(4nr) \hpd{Z} + d \height{F}\deg{Z}.
\end{alignat*}
where we used (\ref{eq:heightWub}) to bound $\height{W}$ in terms of
$\height{F}$. 
The bound for $\height{F}$ and  the inequalities  $\kappa
\ge 1$ and $a\le k_0$
show
\begin{alignat}1
\label{eq:heightVub}
  \height{V} 
&\le 14 dnr \log(4nr) \hpd{Z}k_0 +
d\left(400n\max\{n,r\}^2 \frac{\kappa^{d+1}}{\Delta_0} \hpd{Z}
+ 5d\deg{Z}\right)\deg{Z}k_0 \\
\nonumber 
&\le d\left(14nr\log(4nr) + 400n\max\{n,r\}^2
\frac{\kappa^{d+1}}{\Delta_0} \deg{Z} +5d\deg{Z}\right)\hpd{Z}k_0.
\end{alignat}

We estimate $\log(4nr)=\log(2n)+\log(2r)\le n+r\le 2\max\{n,r\}$ and 
recall $d\le r$.
Inequality (\ref{eq:heightVub}) yields
\begin{alignat*}1
  \height{V} &\le   \max\{n^2r^2,nr^3,n^3r\}\left(28  + 400
\frac{\kappa^{d+1}}{\Delta_0}\deg{Z} + 5\deg{Z}\right)\hpd{Z}k_0 \\
&\le 2^9 \max\{n^2r^2,nr^3,n^3r\}
\max\left\{1,\frac{\kappa^{d+1}}{\Delta_0}\right\} \hpd{Z} \deg{Z} k_0.\qed
\end{alignat*}

\section{Degree and Height Upper Bounds for  Compactifications}
\label{sec:compact}

\newcommand{\cp}[2]{\overline{{#1}}^{#2}}

Throughout this subsection let
 $X\subsetneq\IGm^n$ be an irreducible closed subvariety defined over
$\IQbar$
 of dimension
$1\le r\le n-1$. 
Moreover, let $\varphi:\IGm^n\rightarrow \IGm^r$ be a homomorphism.

Using the  open immersion $\IGm^n\hookrightarrow \IP^n$ introduced
in Section \ref{sec:gennotation} we consider
the algebraic torus as an open subvariety of projective space. We let
$\overline X$ denote the Zariski closure of $X$ in $\IP^n$. 
Similarly we may consider $\IGm^n\times\IGm^r$ as an open subvariety
of $\IP^n\times\IP^r$. We let $\cp{X}{\varphi}$ denote the Zariski
closure in $\IP^n\times\IP^r$ of the graph of $\varphi|_X :
X\rightarrow \IGm^r$. Then $\cp{X}{\varphi}$
 is an irreducible closed subvariety of $\IP^n\times\IP^r$ of
dimension $r$. There are  open immersions
$X\rightarrow\overline{X}$ and 
$X\rightarrow \cp{X}{\varphi}$. The same compactification was used in
the earlier  paper \cite{BHC}.

As in  previous sections $\pi_{1}$ and $\pi_2$
denote the projections $\IP^n\times\IP^r\rightarrow \IP^n$
and $\IP^n\times\IP^r\rightarrow \IP^n$, respectively.
We end up with a commutative diagramm 
\begin{equation}
\begin{CD}
\label{eq:compactify2}
 \overline{X}@ <<  <      X    @>{\varphi|_X}>>   \IGm^r \\
 @|  @V{} VV          @VV{}V \\
 \overline{X}@<<{\pi_1|_{\cp{X}{\varphi}}} < \cp{X}{\varphi}     
@>>{\pi_2|_{\cp{X}{\varphi}}}>  \IP^r
\end{CD}
\end{equation}
of morphisms.

The purpose of this section is to bound degree and height of
$\cp{X}{\varphi}$ in terms of $X$ and $\varphi$. Recall that the
degree of $\cp{X}{\varphi}$ is by definition the degree of
$s(\cp{X}{\varphi})\subset \IP^{nr+n+r}$ where $s$ is the
Segre morphism. Recall that the height of $\cp{X}{\varphi}$ is the
height of $s(\cp{X}{\varphi})$ as a subvariety of $\IP^{nr+n+r}$.
 Moreover, $\Delta_i(\cp{X}{\varphi})$ are the
bidegrees introduced in Section \ref{sec:corr}. 

We use $|\cdot|_\infty$ to denote the sup-norm of any matrix with real
coefficients. 

Homomorphisms $\IGm^n\rightarrow \IGm^r$  can be identified with 
$n\times r$ matrices in integer coefficients. Therefore,
$|\varphi|_\infty$ is well-defined. It is non-zero if and only if
$\varphi$ is non-constant. 

\begin{lemma}
\label{lem:degZbound}
 If $\varphi$ is non-constant, then
 \begin{equation}
\label{eq:degcompact}
  \deg{\cp{X}{\varphi}} \le (4n|\varphi|_\infty)^r \deg{X}   
 \end{equation}
and
\begin{equation}
\label{eq:Deltai}
  \Delta_i(\cp{X}{\varphi}) \le (4n)^r |\varphi|_\infty^{r-i} \deg{X}
\quad\text{for all}\quad 0\le i \le r
\end{equation}
and
\begin{equation}
\label{eq:Delta0}
  \Delta_0(\cp{X}{\varphi}) = \deg{\varphi|_X} \le (4n|\varphi|_\infty)^r \deg{X}.
\end{equation}
\end{lemma}
\begin{proof}
  We essentially follow the argument given in Lemma 3.3 \cite{BHC}.
The $l^2$-norm on $\mat{n,r}{\IR}$ was used in  the reference, but
 here we work with the sup-norm.
There it is shown that 
 $\cp{X}{\varphi}$ is an irreducible component of $(\overline
  {X}\times\IP^r)\cap Y$ where $Y$ is the set of common zeros of
  bihomogeneous polynomials whose bidegrees are at most $(D,1)$ with
  $D\le {n} |\varphi|_\infty$. 

By Philippon's version of B\'ezout's Theorem, Proposition 3.3
\cite{Philippon}, we bound
\begin{alignat}1
\nonumber
  \hhilbpoly{D,1;\cp{X}{\varphi}} &\le \hhilbpoly{D,1;\overline X\times\IP^r} \\
& 
\label{eq:HD1XPr}
=  \sum_{i=0}^{2r}
{2r \choose i}
 ({\pi_1^*\bigo{1}}^i{\pi_2^*\bigo{1}}^{2r-i} [\overline X\times\IP^r])D^i.
\end{alignat}

Commutativity of intersection products and the projection
formula imply
\begin{equation*}
  ({\pi_1^*\bigo{1}}^i {\pi_2^*\bigo{1}}^{2r-i}
[\overline X\times\IP^r]) = 
\left({\bigo{1}}^{2r-i} {\pi_2}_* ({\pi_1^*\bigo{1}}^i
   [\overline{X}\times\IP^r])\right).
\end{equation*}
If $i>r = \dim\overline X$, then the intersection number on the
   right vanishes. 
On the other hand, if $2r-i>r$ then it vanishes too because 
\begin{equation*}
({\pi_1^*\bigo{1}}^i {\pi_2^*\bigo{1}}^{2r-i}
[\overline X\times\IP^r])
=\left({\bigo{1}}^i {\pi_1}_*({{\pi_2}^*\bigo{1}}^{2r-i}
[\overline X\times\IP^r])\right).  
\end{equation*}

So only the term $i=r$ survives in
   (\ref{eq:HD1XPr}) and we obtain
\begin{equation*}
  \hhilbpoly{D,1;\cp{X}{\varphi}} \le {2r \choose r}
\left({\pi_1^*\bigo{1}}^r{\pi_2^*\bigo{1}}^r [\overline
     X\times\IP^r]\right)  D^r.
\end{equation*}
The intersection number on the right is
$\deg{X}$ and we conclude 
\begin{equation*}
  \hhilbpoly{D,1;\cp{X}{\varphi}} \le {2r\choose r}\deg{X}
  D^r. 
\end{equation*}

Inserting the definition of
$\Delta_i(\cp{X}{\varphi})$ from Section \ref{sec:corr}
leads us to 
\begin{equation*}
  \sum_{i=0}^r {r\choose i} \Delta_i(\cp{X}{\varphi}) D^i 
=\hhilbpoly{D,1;\cp{X}{\varphi}}
\le {2r\choose r}
  \deg{X}D^r
 \le 4^{r} \deg{X} D^r.
\end{equation*}
We note that the $\Delta_i(\cp{X}{\varphi})$ cannot be negative. 
By Lemma \ref{lem:hilbpolyboundsegre} we have
 $\deg{\cp{X}{\varphi}} = \sum_{i=0}^r {r \choose i}\Delta_i(\cp{X}{\varphi})$.
So (\ref{eq:degcompact}) and (\ref{eq:Deltai}) follow  from  $1\le D\le n|\varphi|_\infty$.

We have $\Delta_0(\cp{X}{\varphi}) = (\pi_2^*\bigO{1}^{r}[\cp{X}{\varphi}])$
by definition. The projection formula implies
\begin{equation*}
 \Delta_0(\cp{X}{\varphi}) = (\bigO{1}^r {\pi_2}_*[\cp{X}{\varphi}])
= \deg{\pi_2|_{\cp{X}{\varphi}}}(\bigO{1}^r
[\pi_2(\cp{X}{\varphi})]). 
\end{equation*}
We have $\deg{\pi_2|_{\cp{X}{\varphi}}} = \deg{\varphi|_X}$
since the all vertical arrows in (\ref{eq:compactify2}) are birational
morphisms.
Equality (\ref{eq:Delta0}) certainly holds  if $\deg{\varphi|_X}=0$. Otherwise,
$\pi_2|_{\cp{X}{\varphi}}:\cp{X}{\varphi}\rightarrow \IP^r$
has generically finite fibers and we have $\dim \pi_2(\cp{X}{\varphi})
= \dim \cp{X}{\varphi} =  r$ by the Fiber Dimension Theorem.
Hence $\pi_2(\cp{X}{\varphi})=\IP^r$
and so $\Delta_0(\cp{X}{\varphi}) 
= \deg{\pi_2|_{\cp{X}{\varphi}}}$, as desired.
\end{proof}

We need information on $\kappa(\cp{X}{\varphi})$, as defined in
(\ref{eq:definekappa}).
 Recall that we only defined this 
quantity  if $\Delta_0(\cp{X}{\varphi}) >0$
and $\Delta_i(\cp{X}{\varphi})>0$ for some $1\le i\le r$.

\begin{lemma}
\label{lem:kappabound}
We assume
 $\deg{\varphi|_X}\ge 1$. Then
  $\Delta_0(\cp{X}{\varphi})>0, \Delta_r(\cp{X}{\varphi}) > 0,$ and we have
  \begin{equation*}
\frac{|\varphi|_\infty}{(4n)^r \deg{X}}
\frac{\deg{\varphi|_X}}{|\varphi|_\infty^r} 
\le    \kappa(\cp{X}{\varphi})\le \deg{\varphi|_X}.
  \end{equation*}
\end{lemma}
\begin{proof}
Positivity of $\Delta_0(\cp{X}{\varphi})$ follows from Lemma \ref{lem:degZbound}.
  By the projection formula we have $\Delta_r(\cp{X}{\varphi}) =
  \deg{\pi_1|_{\cp{X}{\varphi}}} (\bigO{1}^r
      [\pi_1(\cp{X}{\varphi})])$.
Recall that $\cp{X}{\varphi}\subset\IP^n\times\IP^r$ is
the Zariski closure of the graph $\varphi|_X:X\rightarrow\IGm^r$.
This implies $  \deg{\pi_1|_{\cp{X}{\varphi}}}=1$ and
$\pi_1(\cp{X}{\varphi})=\overline{X}$. So 
$\Delta_r(\cp{X}{\varphi}) = \deg{\overline X} > 0$.

We use Lemma \ref{lem:degZbound} to bound $\kappa  =
\kappa(\cp{X}{\varphi})$
 from above and
below. Indeed, suppose $1\le i\le r$ with 
$\Delta_i(\cp{X}{\varphi}) > 0$ and
$\kappa=(\Delta_0(\cp{X}{\varphi})/\Delta_i(\cp{X}{\varphi}))^{1/i}$.

Then $\kappa\le \Delta_0(\cp{X}{\varphi})^{1/i}\le
\Delta_0(\cp{X}{\varphi})$
since $\Delta_0(\cp{X}{\varphi}) \ge 1$ and 
$\Delta_i(\cp{X}{\varphi})\ge 1$. The desired upper bound for $\kappa$
follows from (\ref{eq:Delta0}).

For the lower bound,  (\ref{eq:Deltai}) and (\ref{eq:Delta0}) give
\begin{equation*}
  \kappa =
  \left(\frac{\Delta_0(\cp{X}{\varphi})}{\Delta_i(\cp{X}{\varphi})}\right)^{1/i}
\ge |\varphi|_\infty
\left(\frac{\deg{\varphi|_X}}{(4n)^r|\varphi|_\infty^r
    \deg{X}}\right)^{1/i}
\ge |\varphi|_\infty
\frac{\deg{\varphi|_X}}{(4n)^r|\varphi|_\infty^r
    \deg{X}},
\end{equation*}
in the second inequality we used (\ref{eq:Delta0}). 
\end{proof}

In Section \ref{sec:heights} we introduced a height function
$\heightsS : \IGm^n(\IQbar)\rightarrow [0,\infty)$ with sup-norm at
  the infinite places. 


We now bound the height of $\cp{X}{\varphi}$.
\begin{lemma}
\label{lem:hZbound}
 If $\varphi\not=0$, then
  $\height{\cp{X}{\varphi}} \le 
(4n)^{n+2} |\varphi|_\infty^{r+1}
\left(\height{X} + \frac{\deg{X}}{|\varphi|_\infty}\right)$.
\end{lemma}
\begin{proof}
  For any $p\in\IGm^n( \IQbar)$ we have
$\heights{\varphi(p)}\le nr |\varphi|_\infty \heights{p}$
by the discussion around (\ref{eq:heightcompare}).
Therefore, $\height{\varphi(p)}\le \frac 12 \log(n+1) +
nr|\varphi|_\infty\heights{p} \le n + nr|\varphi|_\infty\height{p}$
from the comparison estimates between the two heights.
Let $\epsilon > 0$.
We may find a Zariski dense set of $p\in X(\IQbar)$ 
with
$\height{p}\le
\essmin{X}+\epsilon$.
Recall that $s$ is the Segre morphism.
We have $\height{s(p,\varphi(p))} =
\height{p}+\height{\varphi(p)}\ge \height{p}$, so
\begin{equation*}
  \height{s(p,\varphi(p))} \le
\essmin{X}+\epsilon + n+
nr|\varphi|_\infty\height{p}
\le n + (1+nr|\varphi|_\infty)(\essmin{X}+\epsilon).
\end{equation*}
The resulting set of points $(p,\varphi(p))$ lies Zariski dense in $\cp{X}{\varphi}$.
 So
 \begin{equation*}
   \essmin{s(\cp{X}{\varphi})} \le n+2nr |\varphi|_\infty
   (\essmin{X}+\epsilon) 
 \end{equation*}
and
 letting $\epsilon$ tend to zero gives
\begin{equation}
\label{eq:essminbound}
  \essmin{s(\cp{X}{\varphi})} 
\le n + 2nr|\varphi|_\infty\essmin{X}. 
\end{equation}

Zhang's inequality (\ref{eq:Zhang}) lets us compare essential minimum with the
height and degree of a variety. More precisely, we have
\begin{equation*}
  \essmin{s(\cp{X}{\varphi})} \ge \frac{\height{\cp{X}{\varphi}}}{(1+r) \deg{\cp{X}{\varphi}}} 
\quad\text{and}\quad
  \essmin{X} \le \frac{\height{X}}{\deg{X}}.
\end{equation*}
Combining these two with (\ref{eq:essminbound})  gives
\begin{equation*}
  \frac{\height{\cp{X}{\varphi}}}{(1+r)\deg{\cp{X}{\varphi}}}
  \le n+ 2{nr}|\varphi|_\infty \frac{\height{X}}{\deg{X}}
\end{equation*}
and so
\begin{equation*}
  \height{\cp{X}{\varphi}} \le 
n^2\deg{\cp{X}{\varphi}} + 
2n^3|\varphi|_\infty \frac{\deg{\cp{X}{\varphi}}}{\deg{X}} \height{X}
\end{equation*}
after using $r\le n-1$.
Lemma \ref{lem:degZbound} implies $\deg{\cp{X}{\varphi}}\le (4n|\varphi|_\infty)^r
\deg{X}$ and this completes the proof.
\end{proof}

\input tropical.tex

\section{Generically Bounded Height}
\label{sec:genbh}

As usual, let $r,s,n\in\IN$.
For brevity we introduce a constant 
\begin{equation*}
  \mu(r,s,n) =
 (r^2+r+1)\left(1+ \frac {n-r}{2} \left(r+1+(r+3)(sn+1)(2r)^{sn}\right)\right) + r.
\end{equation*}
It is convenient to allow $r$ or $s$ to be zero and  set
$\mu(r,s,n)=0$ in this case.

For example in the case of curves $r=1$ and if $s=1$, then 
\begin{equation*}
  \mu(1,1,n) = 6(n^2-1)2^n + 3n+1.
\end{equation*}

\begin{theorem}
\label{thm:genhb}
Let $n\ge 1$ and
let $X\subset \IGm^n$ be an irreducible closed subvariety defined over $\IQbar$
and say $s \ge 0$ is an integer.
Then one of the following two cases holds.
\begin{enumerate}
\item [(i)] There exists an algebraic subgroup $H\subset\IGm^n$
  such that
  \begin{equation*}
    \dim_p X \cap pH \ge \max\{1, s + \dim H - n + 1\}
  \end{equation*}
for all $p\in X(\IQbar)$. 
\item [(ii)] 
We set
$C = 2^{200 n^7 (2n)^{n^2}}$ and $\mu =  \mu(\dim X,s,n)$.
There is a non-empty Zariski open subset $U\subset X$  such that
if $p\in U(\IQbar)\cap \sgu{s}$, then 
\begin{equation*}
  \height{p}\le C \deg{X}^{\mu}(1+\height{X}).
\end{equation*}
\end{enumerate}
\end{theorem}

Say $p\in\sgu{s}$. Our approach is to find a surjective homomorphism
$\IGm^n\rightarrow\IGm^s$ from a finite set
that maps $p$ to a point of  height not too large controlled with $\height{p}$.
Recall that we identify elements of $\mat{s,n}{\IZ}$ 
with homomorphisms $\IGm^n\rightarrow\IGm^s$.

\begin{lemma}
\label{lem:approx}
Suppose $1\le s\le n$ and 
let $Q\ge 2s!$ be a real number.
If $p\in \sgu{s}$  there exists
 $\psi\in\mat{s,n}{\IZ}$ with
\begin{equation*}
  |\pi\psi|_\infty \ge Q\quad\text{for all}\quad\pi\in\Pi_{rs}, \quad
|\psi|_\infty < Q+1,\quad
\height{\psi(p)}\le \frac 12 \log(n+1)+(sn)\height{p},
\end{equation*}
and such that 
$\psi/Q$ is $1/(2s!)$-regular.
\end{lemma}
\begin{proof}
There is $\psi_0 \in\mat{s,n}{\IZ}$ of rank $s$
whose kernel contains $p$.

An $s\times s$ minor of $\psi_0$ 
whose  discriminant has maximal absolute value among all $s\times s$
minors is invertible. We let $\alpha^{-1}$ denote such a minor, then
$\alpha\in\mat{s}{\IQ}$.
By our choice $|\alpha\psi_0|_\infty =  1$ and
some $s\times s$ minor of 
$\alpha\psi_0$ is the identity matrix.

There is  $\psi\in\mat{s,n}{\IZ}$ 
with $|Q\alpha\psi_0-\psi|_\infty< 1$. We may arrange that
 the sup-norm of each row of $\psi$ is at
least $Q$. So $\psi$ satisfies the first property. 


The second property follows from 
 $|\psi|_\infty \le 
|Q\alpha\psi_0-\psi|_\infty +  |Q\alpha\psi_0|_\infty
< Q + 1$.

To prove the third claim we set $\delta = \alpha\psi_0-\psi/Q$ and
fix $\lambda\in\IN$ with $\lambda\alpha,\lambda\delta\in\mat{s,n}{\IZ}$. If
$u_1,\dots, u_s\in\IZ^n$ are the rows of $-\lambda\delta$, then
basic height properties  imply
\begin{alignat*}1
  \heights{(-\lambda\delta)(p)} &= \heights{p^{u_1},\dots, p^{u_s}}
\le \heights{p^{u_1}}+\cdots + \heights{p^{u_s}}\\
&\le n(|u_1|_\infty + \cdots + |u_s|_\infty) \heights{p} 
\le sn |-\lambda\delta|_\infty\heights{p}
= sn\lambda |\delta|_\infty \heights{p}.
\end{alignat*}
By construction we have $|\delta|_\infty \le 1/Q$ and hence
\begin{equation}
  \label{eq:heightub}
\heights{(-\lambda\delta)(p)} \le sn\frac{\lambda}{Q} \heights{p}.
\end{equation}
But $\lambda\psi = Q\lambda\alpha\psi_0 - Q\lambda\delta$
and $\psi_0(p)=1$, 
so $\lambda\heights{\psi(p)}=\heights{(\lambda\psi)(p)} = Q\heights{(-\lambda\delta)(p)}$
by basic height properties. 
Using (\ref{eq:heightub}) we obtain
$\heights{\psi(p)}
\le sn \heights{p}$.
The third property follows from
\begin{equation*}
  \height{\psi(p)}\le \frac 12 \log(n+1) + \heights{\psi(p)}
\le \frac 12 \log(n+1)+sn \heights{p}\le \frac 12 \log(n+1) + sn \height{p}.
\end{equation*}

We turn to the final property. 
Let $\phi\in\mat{s,n}{\IR}$ 
satisfy $|\phi|_\infty \le \epsilon$ with $\epsilon \le 1/2$
and $\rank{\psi/Q+\phi}< s$.
The matrix $\psi/Q+\phi = \alpha\psi_0-(\delta - \phi)$ has an $s\times s$-minor which equals
$1-(\delta'-\phi')$ where $1$ is the unit matrix and 
$\delta'$ and $\phi'$ are  $s\times s$-minors of $\delta$ and $\phi$, respectively.
We have $|\delta'-\phi'|_\infty\le |\delta|_\infty +|\phi|_\infty \le
1/Q +\epsilon \le 1$.
Using the Leibniz expansion for the determinant
we conclude $|\det{(1-(\delta'-\phi'))}-1|_\infty\le s! |\delta'-\phi'|_\infty\le
s!(1/Q+\epsilon)
\le 1/2 + s!\epsilon$ because $Q\ge 2s!$.
If $\epsilon < 1/(2s!)$, then $\det{(1-(\delta'-\phi'))}\not=0$ and in this case
$\psi/Q+\phi$ has full rank, contradicting our assumption on
$\phi$. 
Thus we must have $\epsilon \ge 1/(2s!)$. It follows that $\psi/Q$
is $1/(2s!)$-regular as desired.
\end{proof}

\begin{proposition}
\label{prop:genheightbound}
  Let $X\subsetneq\IGm^n$ be an irreducible closed subvariety defined
  over $\IQbar$ of
  dimension $r\ge 1$. Let $s$ be an integer with $r\le s \le n$, for
  brevity we write $\mu = \mu(r,s,n)$.
Then we are either in alternative (i) of Proposition
\ref{prop:degreelb} or the following statement holds.
There is a finite set of $\{V_1,\dots,V_N\}$ of irreducible closed
subvarieties of $X$ with
\begin{alignat*}1
\dim{V_i}&= r-1,\\
\deg{V_1}+\cdots + \deg{V_N}&\le 2^{200n^7(2n)^{n^2}}
  \deg{X}^{(r+1)\left(1+r\frac{\mu-r}{r^2+r+1}\right)}, \\
 \height{V_i}&\le  2^{300 n^7 (2n)^{n^2}} \deg{X}^{
2r + 1+  (2r^2+2r+1)\frac{\mu-r}{r^2+r+1}}(1+\height{X}),
\end{alignat*}
and 
such that if $p\in (X\ssm (V_1\cup\cdots\cup V_N))(\IQbar)\cap\sgu{s}$ then
\begin{equation*}
  \height{p}\le 2^{200n^7(2n)^{n^2}}
\deg{X}^{\mu}(1+\height{X}).
\end{equation*}
\end{proposition}
\begin{proof}
  We begin by choosing the parameter which appears in Lemma \ref{lem:approx} as
  \begin{equation*}
    Q = 2^{100 n^5 (2n)^{n^2}} \deg{X}^\chi 
  \end{equation*}
with 
\begin{equation*}
\chi = 
 1+ \frac {n-r}{2} \left(r+1+(r+3)(sn+1)(2r)^{sn}\right)
=\frac{\mu-r}{r^2+r+1} \ge 1.
\end{equation*}
Elementary estimates show $Q\ge 2n! \ge 2s!$.

Let $p$ be as in the hypothesis and let $\psi$ be given by Lemma
\ref{lem:approx} 
applied to $p$.
We have $|\psi|_\infty < Q+1\le 2Q$. 

We apply Proposition \ref{prop:degreelb} with $\epsilon = 1/(2s!)$ to
$\psi/Q$. Hence there is $\pi\in\Pi_{rs}$ with
\begin{equation*}
  \deg{\varphi/Q|_X} \ge 2^{-50n^5(2n)^{n^2}} (4s!)^{-sn(2r)^{sn}}
  \deg{X}^{-\chi + 1}
\end{equation*}
for $\varphi = \pi\psi$.
We bound $(4s!)^{sn(2r)^{sn}}\le (4n!)^{n^2(2n)^{n^2}} \le
2^{(n^2 + 2)n^2 (2n)^{n^2}} 
\le 2^{3n^4(2n)^{n^2}}$ using $2n!\le 2^{n^2}$. So
\begin{equation}
\label{eq:degreelb1}
  \deg{\varphi/Q|_X} \ge 2^{-53n^5(2n)^{n^2}}  \deg{X}^{-\chi + 1}.
\end{equation}

We set $Z=\cp{X}{\varphi}$ as in Section \ref{sec:compact}. This is an
irreducible closed subvariety of $\IP^n\times\IP^r$  defined over $\IQbar$
with dimension $r$. Futher
down we will apply Proposition \ref{prop:heightlb2} to $Z$. But first
we bound the
various
 quantities associated to $Z$. Indeed, Lemma \ref{lem:kappabound}
and $Q\le |\varphi|_{\infty}\le 2Q$ give
\begin{equation*}
  \deg{\varphi|_X}\ge \kappa \ge \frac{Q}{(4n)^r \deg{X}} \frac{\deg{\varphi|_X}}{(2Q)^r}=
\frac{Q}{(8n)^r \deg{X}} \deg{\varphi/Q|_X}
\end{equation*}
where $\kappa = \kappa(Z)$.
Using the  crude bound $(8n)^r \le (8n)^n\le 2^{4n^2} \le
2^{2n(2n)^{n^2}}$ together with
(\ref{eq:degreelb1}) we obtain
$\kappa \ge 2^{-55 n^5 (2n)^{n^2}}Q \deg{X}^{-\chi}$.
On inserting our choice of $Q$ we get
\begin{equation}
\label{eq:kappalb}
  \kappa \ge 2^{45 n^5 (2n)^{n^2}}.
\end{equation}
 Lemma \ref{lem:degZbound} enables us to bound
 \begin{equation}
   \label{eq:delta0bound}
 \Delta_0 = \Delta_0(Z) \le (4n)^r (2Q)^r \deg{X}
\le (8n)^n Q^r \deg{X} \le 2^{4n^2}Q^r\deg{X}.
 \end{equation}
Degree and height of $Z$ are bounded by Lemmas \ref{lem:degZbound} and \ref{lem:hZbound}, we
have
\begin{equation*}
  \deg{Z}\le (8n)^n Q^r \deg{X}
\le 2^{4n^2} Q^r \deg{X}
\end{equation*}
and
\begin{equation*}
\height{Z}\le (8n)^{n+2} (Q^{r+1}\height{X} + Q^r \deg{X})\le
2^{12n^2} (Q^{r+1}\height{X} + Q^r \deg{X}),
\end{equation*}
so
\begin{equation*}
  \height{Z}+\deg{Z}\le 2^{12n^2} Q^{r+1} (\height{X}+2\deg{X}/Q)
  \le 2^{12n^2} Q^{r+1} (1+\height{X})
\end{equation*}
since $Q\ge 2\deg{X}$.

It is not difficult to verify
$\kappa \ge 17n^2 \ge 17rn$. So the hypothesis on $\kappa$ in
Proposition \ref{prop:heightlb2} applied to $Z$ is satisfied. 
We let $V_1,\dots,V_N\subset X$ denote 
the obstruction varieties given by this
proposition that actually meet
 $\IGm^n$.

Say $k_0 = \max\{17\cdot 3^r r! n \Delta_0(Z)^{r-1},\deg{Z}\}$ is
as in Proposition \ref{prop:heightlb2}. We recall
(\ref{eq:delta0bound}) and the bound for $\deg{Z}$ to obtain
\begin{alignat}1
  \label{eq:k0bound}
k_0 &\le \max\{2^{5+2n}n!n 2^{4n^3} Q^{r(r-1)} \deg{X}^{r-1}, (8n)^n
Q^r \deg{X}
\} \\
\nonumber
&\le 2^{13 n^3} \max\{Q^{r(r-1)} \deg{X}^{r-1}, Q^r\deg{X} \}
\\
\nonumber
&\le 2^{13 n^3} Q^{r^2} \deg{X}^r.
\end{alignat}

Now $\sum_{i=1}^N \deg{V_i}\le k_0 \deg{Z}$, so 
\begin{equation*}
\sum_{i=1}^N \deg{V_i }\le 2^{4n^2}   Q^r \deg{X} k_0
\le 2^{17 n^3} Q^{r^2 + r} \deg{X}^{r+1}.
\end{equation*}
Our choice of $Q$ and $r^2+r \le (n-1)^2 + n - 1 \le n^2$ implies
\begin{equation*}
 \sum_{i=1}^N \deg{V_i}\le 2^{17n^3 + 100n^7(2n)^{n^2}} \deg{X}^{(r+1)(1+r\chi)}. 
\end{equation*}
This bound leads quickly to the desired estimate for sum over the
degrees of the $V_i$.

We let $V$ denote one of the obstruction varieties $V_i$.
We must bound $\height{V}$ too. Proposition \ref{prop:heightlb2} and
the bounds for $\height{Z}+\deg{Z}$ and $\deg{Z}$ give
\begin{alignat*}1
  \height{V}&\le
2^9 n^4 \max\left\{1,\frac{\kappa^{r+1}}{\Delta_0}\right\}
(\height{Z}+\deg{Z})\deg{Z} k_0 \\
&\le 
2^{30n^2}  Q^{2r+1} \Delta_0^{r} \deg{X} (1+\height{X})k_0
\end{alignat*}
in the last equality we used $\kappa \le \Delta_0$
which follows from Lemmas \ref{lem:kappabound} and  
\ref{lem:degZbound}. Furthermore,
$\Delta_0$ is bounded by (\ref{eq:delta0bound}); this gives 
\begin{alignat*}1
\height{V}&\le 2^{40n^3}Q ^{r^2+2r+1} \deg{X}^{r+1}
(1+\height{X})k_0.
\end{alignat*}
 With
(\ref{eq:k0bound})   we get
\begin{alignat*}1
  \height{V}&\le 2^{60n^3} Q^{2r^2+2r+1} \deg{X}^{2r+1}(1+\height{X}).
\end{alignat*}
Since $r\le n-1$ we have $2r^2+2r+1 \le 2n^2$, so with our choice of
$Q$ we find
\begin{equation*}
\height{V}\le 2^{300 n^7 (2n)^{n^2}} \deg{X}^{2r + 1+  (2r^2+2r+1)\chi}(1+\height{X}),
\end{equation*}
as desired.

To complete the proof it remains to bound the height of $p$.
Clearly, we may assume that $\height{p}\ge 1$.
No projective coordinate of $p$ vanishes. If $q=\varphi(p)$ then $(p,q)$ is
isolated in $\pi_1|_Z$; indeed,  locally at $(p,q)$
 the variety $Z$ is the graph of
$\varphi:\IGm^n\rightarrow \IGm^r$. Hence we are in case (iii) of
Proposition \ref{prop:heightlb2}. Thus, on using $r\le n$ and $\kappa\ge 1$
we have
\begin{alignat*}1
  \height{p}
&\le \frac{2^5 n^2}{\kappa} \height{q} 
+ 2^{15} n^5 \max\left\{1,\frac{\kappa^{r+1}}{\Delta_0}\right\}
(\height{Z}+\deg{Z}) \\
&\le \frac{2^{7n}}{\kappa} \height{q} 
+ 2^{20n} \Delta_0^{r}
(\height{Z}+\deg{Z})
\end{alignat*}
here we used $\kappa \le \Delta_0$ again.
Upper bounds for $\Delta_0$
 and $\height{Z}+\deg{Z}$ were obtained above; we conclude
 \begin{alignat*}1
 \height{p}&\le \frac{2^{7n}}{\kappa}\height{q} 
+2^{20n + 4n^2r + 12n^2}
Q^{r^2+r+1}\deg{X}^{r}(1+\height{X}) \\
&
\le \frac{2^{7n}}{\kappa}\height{q} 
+2^{40n^3}
Q^{r^2+r+1}\deg{X}^{r}(1+\height{X})
 \end{alignat*}
Our choice of $Q$ and $r^2 +r+ 1 \le (n-1)^2 +(n-1)+ 1 \le n^2$ implies
\begin{alignat*}1
\height{p}  &\le \frac{2^{7n}}{\kappa} \height{q}
+ 2^{150n^7 (2n)^{n^2}} \deg{X}^{(r^2+r+1)\chi+r}(1+\height{X}).
\end{alignat*}
We note that $(r^2+r+1)\chi + r=\mu$, so
\begin{equation*}
\height{p}  \le \frac{2^{7n}}{\kappa} \height{q} +  2^{150n^7(2n)^{n^2}}
\deg{X}^{\mu}\left(1+\height{X}\right).
\end{equation*}

 Lemma \ref{lem:approx} tells us that
$\height{q}\le\height{\psi(p)}\le \frac 12 \log(n+1) + n^2\height{p} \le
n+n^2\height{p}\le 2n^2\height{p} \le 2^{3n}\height{p}$ since $\height{p}\ge 1$.
 We obtain
\begin{equation*}
\left(1-\frac{2^{10n}}{\kappa}\right)\height{p}  \le   2^{150n^7(2n)^{n^2}}
\deg{X}^{\mu}\left(1+\height{X}\right).
\end{equation*}

Certainly, (\ref{eq:kappalb}) implies
$\kappa \ge 2^{11n}$. So $1-{2^{10n}}{\kappa}^{-1}\ge
 1/2$ and the proposition follows.
\end{proof}

\begin{proof}[Proof of Theorem \ref{thm:genhb}]
We may suppose $s\le n$ since $\sgu{s}=\emptyset$ otherwise.
If $r=0$, then part (ii) holds with $U=X$ because $\mu(0,s,n)=0$ and
since the height of a point is its height considered as a variety. So we
assume $r\ge 1$.
If $s<r$, then we are  in case (i) with $H=\IGm^n$. So we may suppose
$r\le s$.
If $s=n$, then $\sgu{s}$ is the set of torsion points of
$\IGm^n$.
The height of a torsion point is $\frac 12 \log(n+1)\le n$ and so the height
bound in part (ii) holds
with $U=X$. So let us assume $s\le n-1$.
Then $1\le r\le n-1$ and the theorem follows from the previous
proposition. 
\end{proof}

\section{Bounded Height}
\label{sec:descent}
In this section we prove Theorem \ref{thm:effbhc}.

The first few reduction steps are similar as in the proof of Theorem
\ref{thm:genhb}.

We may suppose $s\le n$ since $\sgu{s}=\emptyset$ otherwise.
If $r=0$, then part (ii) holds with $Z=\emptyset$ because $\mu(0,s,n)=0$ and
since the height of a point is its height considered as a variety. So we
assume $r\ge 1$.
If $s<r$, then $\oap{X}{s}=\emptyset$ and the theorem follows with
 $Z=X$. So we may suppose
$r\le s$.
If $s=n$, then $\sgu{s}$ is the set of torsion points of
$\IGm^n$.
The height of a torsion point is $\frac 12 \log(n+1)\le n$ and so the height
bound in part (iii) holds
with $Z=\emptyset$. So let us assume $s\le n-1$.

We have reduced to the case $1\le r\le n-1$.



For brevity, we set $\mu=\mu(r,s,n)$.
 Elementary estimates lead to
\begin{alignat}1
\label{eq:mubound}
  \mu &
\le n^2\left(1+ \frac {n-1}{2}
\left(n+n^2(n+2)(2n)^{n^2-n}\right)\right) + n-1
\le  n^6(2n)^{n^2}.
\end{alignat}
We remark $\mu \ge 2$ and $r^2+r+1\ge 2r$. 
 We use $1\le r\le n-1$ and (\ref{eq:mubound}) to bound the exponents 
\begin{alignat}1
\label{eq:mubound1}
(r+1)\left(1+r \frac{\mu-r}{r^2+r+1}\right) &\le 
n \left(1+\frac\mu 2\right) \le n \mu \le n^7(2n)^{n^2},\\
\label{eq:mubound2}
2r+1 + (2r^2 + 2r + 1)\frac{\mu-r}{r^2+r+1}
&\le 2n + 5r^2\frac{\mu}{2r}
\le 5n \mu \le 5n^7 (2n)^{n^2}.
\end{alignat}
that appear in the degree 
and height bound of Proposition \ref{prop:genheightbound}, respectively. 

The remainder of the proof is a somewhat technical descent argument. 
For $0\le i\le \dim X$ we will inductively construct  irreducible
closed subvarieties 
$V^{(i)}_{j_1,\dots,j_i}$ of $X$ which satisfy certain properties to
be described below;  here 
$(j_1,\dots,j_i)$ runs over a finite index set
\begin{equation}
\label{eq:indexset}
  \{ (j_1,\dots,j_i);\,\, 1\le j_1 \le N^{(0)},\, 1\le j_2\le
  N^{(1)}_{j_1},
\dots, 1\le j_i \le N^{(i-1)}_{j_1,\dots, j_{i-1}} \}.
\end{equation}
The properties for the subvarieties are as follows
\begin{alignat}1
\label{eq:dimprop}
  \dim V^{(i)}_{j_1,\dots,j_i} &= \dim X - i, \\
\label{eq:degprop}
 \deg{ V^{(i)}_{j_1,\dots,j_{i-1},
       1}}+\cdots 
+ \deg{ V^{(i)}_{j_1,\dots,j_{i-1},
       N^{(i-1)}_{j_1,\dots,j_{i-1}}}}
     &\le 2^{(300 n^7 (2n)^{n^2})^i} \deg{X}^{(n^7(2n)^{n^2})^i},  \\
\label{eq:heightprop}
\height{V^{(i)}_{j_1,\dots,j_i}} &\le 2^{(600 n^7 (2n)^{n^2})^i}
\deg{X}^{(6n^7(2n)^{n^2})^i}(1+\height{X}), \\
\label{eq:inclprop}
V^{(i)}_{j_1,\dots,j_i} &\subset V^{(i-1)}_{j_1,\dots,j_{i-1}}
\quad\text{if}\quad i\ge 1. 
\end{alignat}

For $i=0$ we take $V^{(0)}=X$. Clearly, (\ref{eq:dimprop}),
(\ref{eq:degprop}), and (\ref{eq:heightprop}) are satisfied.
So we suppose $i\ge 1$ and that  $V^{(i-1)}_{j_1,\dots,j_{i-1}}$ 
has been constructed.

If $V^{(i-1)}_{j_1,\dots,j_{i-1}}$ is as in alternative (i) of
Proposition \ref{prop:degreelb} then we set $N^{(i-1)}_{j_1,\dots,j_{i-1}}=0$ and
the construction stops. 
Otherwise, we take $N^{(i-1)}_{j_1,\dots,j_{i-1}}=N$ as in Proposition
\ref{prop:genheightbound} and
$V^{(i)}_{j_1,\dots,j_{i-1},1},\ldots,
V^{(i)}_{j_1,\dots,j_{i-1},N^{(i-1)}_{j_1,\dots,j_{i-1}}}$ to be the subvarieties $V_1,\dots,V_N$
constructed there. 

Properties (\ref{eq:dimprop}) and
(\ref{eq:inclprop}) follow immediately from this proposition. 
 It remains to prove degree and height bounds; this is done by an
 elementary calculation. 

We begin by verifying (\ref{eq:degprop}). To do this, we keep
(\ref{eq:mubound1}) in mind. By Proposition
\ref{prop:genheightbound}, the left-hand side of (\ref{eq:degprop}) is
at most
\begin{equation*}
  2^{200n^7(2n)^{n^2}} \deg{V^{(i-1)}_{j_1,\dots,j_{i-1}}}^{n^7
      (2n)^{n^2}}
\le 
2^{200 n^7(2n)^{n^2} + (300n^7(2n)^{n^2})^{i-1} n^7 (2n)^{n^2}}
\deg{X}^{(n^7(2n)^{n^2})^i}
\end{equation*}
on using the induction hypothesis. Elementary estimates and $i\ge 1$ give
\begin{equation*}
  200 n^7(2n)^{n^2} + (300n^7(2n)^{n^2})^{i-1} n^7 (2n)^{n^2}
\le (300 n^7 (2n)^{n^2})^i
\end{equation*}
so  (\ref{eq:degprop}) follows.

Now we shall bound the height. We recall (\ref{eq:mubound2}), so 
\begin{equation*}
\height{V^{(i)}_{j_1,\dots,j_i}}\le  2^{300n^7(2n)^{n^2}} 
\deg{V^{(i-1)}_{j_1,\dots,j_{i-1}}}^{5n^7 (2n)^{n^2}} (1+\height{V^{(i-1)}_{j_1,\dots,j_{i-1}}})
\end{equation*}
by Proposition \ref{prop:genheightbound}. 
We insert both degree and height bounds from the induction hypothesis
to find that $\height{V^{(i)}_{j_1,\dots,j_i}}$ is at most 
\begin{equation*}
  2^{300n^7(2n)^{n^2} + (300n^7(2n)^{n^2})^{i-1} 5n^7(2n)^{n^2}
+(600 n^7 (2n)^{n^2})^{i-1}}
\deg{X}^{(n^7 (2n)^{n^2})^{i-1}5n^7(2n)^{n^2} +
  (6n^7(2n)^{n^2})^{i-1}}
(1+\height{X}).
\end{equation*}
Basic estimates and $i \ge 1$ lead to 
\begin{equation*}
  300n^7(2n)^{n^2} + (300n^7(2n)^{n^2})^{i-1} 5n^7(2n)^{n^2}
+(600 n^7 (2n)^{n^2})^{i-1}\le (600 n^7 (2n)^{n^2})^i
\end{equation*}
as well as
\begin{equation*}
  (n^7 (2n)^{n^2})^{i-1}5n^7(2n)^{n^2} +
  (6n^7(2n)^{n^2})^{i-1} \le (6n^7(2n)^{n^2})^i.
\end{equation*}
Hence claim (\ref{eq:heightprop}) holds true. 

We define 
\begin{equation*}
  Z = \bigcup_{i=0}^r \bigcup_{\atopx{V^{(i)}_{j_1,\dots,
        j_i}\text{ as in alt. (i)}}{\text{of Proposition \ref{prop:degreelb}}}} V^{(i)}_{j_1,\dots, j_i}. 
\end{equation*}
It is a Zariski closed subset of $X$. 
This will be the Zariski closed set refered to in the assertion.

We have $X=Z$ if and only if $V^{(0)}$ appears in the union above. 
The degree and height bound for the irreducible components of $Z$
follows from (\ref{eq:degprop}) and (\ref{eq:heightprop}).

We now verify that the bound for the number of irreducible components
of $Z$ is correct. 
If $V^{(0)}=X$ is an irreducible component of $Z$, then our bound
certainly holds.
Otherwise, an irreducible component is of the form
$V^{(i)}_{j_1,\dots,j_i}$ for some $1\le i\le r$. 
For fixed $i$ and $j_1,\dots,j_{i-1}$, the number of possible 
$j_i$ is at most $N^{(i-1)}_{j_1,\dots,j_{i-1}}$. This quantity is
bounded from above by (\ref{eq:degprop}). 
Recalling (\ref{eq:indexset})
 we find that for fixed $i$ there are at most
 \begin{equation*}
   2^{300n^7 (2n)^{n^2} + \cdots +(300n^7 (2n)^{n^2})^i}
\deg{X}^{n^7(2n)^{n^2} + \cdots + (n^7(2n)^{n^2})^i}
\le 
2^{r(300n^7 (2n)^{n^2})^r}
\deg{X}^{r(n^7(2n)^{n^2})^r}
 \end{equation*}
possible $V^{(i)}_{j_1,\dots,j_i}$. To get a bound for all possible
irreducible components we must sum over $1\le i\le r$. This gives us
the bound
\begin{equation*}
r2^{r(300n^7 (2n)^{n^2})^r}
\deg{X}^{r(n^7(2n)^{n^2})^r}
\le  
2^{r+r(300n^7 (2n)^{n^2})^r}
\deg{X}^{r(n^7(2n)^{n^2})^r}.
\end{equation*}
Part (ii) of the theorem follows since
$r+r(300 n^7(2n)^{n^2})^r \le (600n^7 (2n)^{n^2})^r$
and $r(n^7(2n)^{n^2})^r \le (2n^7(2n)^{n^2})^r$. 

We claim that $\oap{X}{s}\subset X\ssm Z$. 
This will imply part (i) of the theorem.
So suppose $p\in  Z(\IQbar)$. 
By definition of $Z$ there  are $i,j_1,\dots,j_i$ and 
 an algebraic subgroup $H\subset\IGm^n$ 
such that $\dim_p V^{(i)}_{j_1,\dots, j_i}\cap pH \ge \max \{1, s+
\dim H - n + 1\}$. 
Certainly, this local dimension is at most
$\dim_p X \cap pH$. So $p\not\in \oap{X}{s}$ and our claim is
established. 

Suppose $p \in (X\ssm Z)(\IQbar)$ with $p\in \sgu{s}$.
It remains to bound $\height{p}$ as in (iii) of the theorem.

By construction $X=V^{(0)}$. We may choose $0\le i\le r$ maximal, such that
there are indices $j_1,\dots,j_i$ with 
$p\in  V^{(i)}_{j_1,\dots,j_i}(\IQbar)$.
We split up into two cases. 

First, let us assume $i=r$. This case is easy, since 
$V^{(i)}_{j_1,\dots,j_i}=\{p\}$ holds by (\ref{eq:dimprop}). 
Now $\height{p}$ is the height of the variety $\{p\}$, so the desired
height bound follows from (\ref{eq:heightprop}). 

Now, we suppose $0\le i \le r-1$. We remark that $p\not\in Z(\IQbar)$
implies that $V^{(i)}_{j_1,\dots,j_i}$ is not as in alternative (i) of
Proposition \ref{prop:degreelb}. In particular, 
Proposition \ref{prop:genheightbound} 
gives a height bound for $p$ providing it does not lie on
\begin{equation*}
V^{(i+1)}_{j_1,\dots,j_i,1}\cup\cdots\cup
V^{(i+1)}_{j_1,\dots,j_i,N^{(i)}_{j_1,\dots,j_i}}.
\end{equation*}
But $p$ cannot lie in this union because of the maximality of $i$. 
Recalling (\ref{eq:mubound}) gives us 
\begin{equation*}
  \height{p}\le 2^{200n^7 (2n)^{n^2}}
  \deg{V^{(i)}_{j_1,\dots,j_i}}^{n^6(2n)^{n^2}} 
(1+\height{V^{(i)}_{j_1,\dots,j_i}}).
\end{equation*}
Our degree bound (\ref{eq:degprop}) and height bound (\ref{eq:heightprop})
give
\begin{equation*}
  2^{200 n^7 (2n)^{n^2} + (300 n^7 (2n)^{n^2})^i n^6(2n)^{n^2}
+(600n^7(2n)^{n^2})^i}
\deg{X}^{(n^7(2n)^{n^2})^i n^6 (2n)^{n^2} + (6n^7(2n)^{n^2})^i}
(1+\height{X})
\end{equation*}
as a bound for $\height{p}$. The exponent of $2$ is
\begin{equation*}
  200 n^7 (2n)^{n^2} + (300 n^7 (2n)^{n^2})^i n^6(2n)^{n^2}
+(600n^7(2n)^{n^2})^i
\le 
(200 + 300^{r-1}+600^{r-1}) (n^7 (2n)^{n^2})^r.
\end{equation*}
because $i\le r-1$. Hence it is at most $(600 n^7(2n)^{n^2})^r$.
The exponent of $\deg{X}$ is 
\begin{equation*}
  (n^7(2n)^{n^2})^i n^6 (2n)^{n^2} + (6n^7(2n)^{n^2})^i
\le (1 + 6^{r-1})(n^7(2n)^{n^2})^r
\le (6n^7(2n)^{n^2})^r.
\end{equation*}
Therefore, 
$\height{p}\le 2^{(600 n^7 (2n)^{n^2})^r}\deg{X}^{(6n^7(2n)^{n^2})^r}
(1+\height{X})$ and part (iii) of the theorem holds. 
\qed

\appendix

\section{The Case of Abelian Varieties}

An abelian variety $A$ defined over $\IQbar$ together with an
ample symmetric line bundle determines a height function
 called the N\'eron-Tate or
  canonical height. It is a quadratic form and vanishes precisely on
  the torsion points of $A$. We refer to Chapter 9 of Bombieri and
  Gubler's book \cite{BG} for the necessary background.

The history of height upper bounds on subvarieties of abelian
varieties runs parallel to the history of height bounds on the
algebraic torus. But we will only give a brief account of what is
known in the projective case. 

An initial result was obtained by Viada \cite{Viada} who proved the
following analog of
Theorem \ref{thm:BMZ}.

\begin{theorem}[Viada \cite{Viada}]
\label{thm:viada}
  Let $A=E^g$ where $E$ is an elliptic curve defined over $\IQbar$ and
  suppose that we have fixed a symmetric and line bundle bundle, and
  thus a N\'eron-Tate height, on $A(\IQbar)$. 
Let $C\subset A$ be an irreducible algebraic curve that is not
contained in the translate of a proper algebraic subgroup of $A$. 
Then the height of 
 points on $C$ that are contained in a proper algebraic subgroup
 is bounded from above uniformly.
\end{theorem}

R\'emond \cite{RemondInterI} gave a systematic approach for passing
from height upper bounds in the spirit of Theorem \ref{thm:viada}
to finiteness result using Lehmer-type and relative Lehmer-type
height lower bounds.  These
inequalities remain  conjectural for many abelian varieties. For
example,  no sufficiently strong Lehmer-type height inequality is
known on a power of an elliptic curve without complex multiplication
to tackle the abelian analog of Theorem \ref{thm:BMZfinite} using
R\'emond's approach.
 Viada \cite{Viada} did obtain a finiteness result akin to
 Theorem \ref{thm:BMZfinite} when the elliptic curve in question has
 complex multiplication. In this setting the sufficiently strong
  Lehmer-type
 height lower bounds are available thanks to work of David and Hindry
\cite{DavidHindry}.


The analog of Maurin's Theorem for curves inside a power of an
elliptic curve was obtain already in 2003 by R\'emond and Viada
\cite{RV}.
Again the elliptic curve was assumed to have complex multiplication.  
As Maurin's Theorem, R\'emond and Viada's Theorem relies on R\'emond's
Generalized Vojta Inequality 
\cite{Remond:Vojtagen}.

Advances made primarily by Galateau \cite{Galateau:Bogomolov,Galateau} 
on Bogomolov-type height lower bounds have had a catalytic effect
on the finiteness problems.
His results hold for a wide class of abelian
varieties, including abelian surfaces and arbitrary powers of elliptic curves.
 Viada \cite{Viada08} used them to
prove the analog of Maurin's Theorem for curves defined over $\IQbar$
inside a power of an arbitrary elliptic curve defined over $\IQbar$.

Partial results for subvarieties of arbitrary dimension in an abelian
variety
are known as well. 
Here the definition of $\oa{X}$ for a subvariety $X\subset A$
is verbatim to the toric case. 
The union $\sguG{A}{s}$ of all algebraic subgroups of $A$ of codimension
at least $s$
also makes perfect sense.
For example, the author  proved \cite{abvar} 
the full analog of   Theorem \ref{thm:BHC}.
That is, the N\'eron-Tate height is uniformly bounded from above
on $\oa{X}(\IQbar)\cap \sguG{A}{\dim X}$.

R\'emond's Generalized Vojta Inequality is powerful enough to treat,
along with varying algebraic subgroups, the division closure of a
finite rank subgroup $\Gamma \subset A(\IQbar)$.
Indeed, he considers points on $X$ contained in 
\begin{equation*}
 \sguG{A}{s}+\Gamma = \{ h+\gamma;\,\, h\in
 \sguG{A}{s}\quad\text{and}\quad
\gamma\in\Gamma\}.
\end{equation*}

\begin{theorem}[R\'emond \cite{RemondInterIII}, cf. \cite{RemondInterII}]
  Let $A$ be an abelian variety defined over $\IQbar$ 
and suppose we have fixed a N\'eron-Tate height on $A$. 
Let $\Gamma \subset A(\IQbar)$ be the division closure of a finitely
generated subgroup of $A$. Moreover, let $X\subset A$ be an
irreducible closed subvariety defined over $\IQbar$. 
Then the height of 
 points in $\oa{X}(\IQbar)\cap (\sguG{A}{1+\dim X}+\Gamma)$
 is bounded from above uniformly.
\end{theorem}

Maurin's Theorem \ref{thm:maurin2} is the toric version of
this result. 

\section{Height Bounds in  Shimura Varieties}

The well-known analogy between  semi-abelian varieties and
Shimura varieties which is underlined by the conceptual similarity
of the  Conjectures of
Manin-Mumford and Andr\'e-Oort   tempts us  to formulate 
a Bounded Height Conjecture in the moduli-theoretic setting. 
The role of algebraic subgroups on the  abelian or toric side
is played by the special subvarieties on the side of Shimura
varieties. Furthermore, the torsion points are replaced by special
points. The sweeping conjecture of Pink \cite{Pink} on mixed Shimura
varieties
covers
the Conjectures of Manin-Mumford and Andr\'e-Oort. By the comments in
Bombieri, Masser and Zannier's appendix \cite{BMZUnlikely},
 Pink's Conjecture implies 
Zilber's Conjecture 1 \cite{Zilber}
and Bombieri, Masser, and Zannier's Torsion Finiteness Conjecture
\cite{BMZGeometric}.

However, a too literal generalization of the Bounded Height Conjecture to  
 Shimura varieties is false. 
 Bombieri, Masser, and Zannier \cite{BMZGeometric} showed
that the height of the $j$-invariant of an elliptic curve with complex
multiplication can be arbitrary large. 
These $j$-invariants, which we call singular moduli,  are algebraic integers.
The Shimura variety in question is $Y(1)$, the modular curve whose
complex points correspond to the $j$-invariants of elliptic curves
defined over $\IC$. As a variety $Y(1)$ equals the affine line. 
Unboundedness of height already follows from an earlier
more general result of Colmez \cite{Colmez} who proved
the following estimate.

The discriminant of a singular moduli is the discriminant of the
endomorphism ring of the corresponding elliptic curve.

\begin{theorem}[Colmez \cite{Colmez}]
  There exists an absolute constant $c>0$ with the following
  property. If $j$ is a singular moduli whose discriminant $\Delta$ is a
  fundamental discriminant, then 
\begin{equation*}
  \height{j} \ge -c^{-1} + c \log |\Delta|. 
\end{equation*}
\end{theorem}

Polynomial upper bounds in the discriminant are available 
 using classical estimates in
 analytic number theory. Pila and the author proved the following inequality.

 \begin{lemma}[Lemma 4.3 \cite{hp:beyondAO}]
\label{lem:phjbound}
For any $\epsilon > 0$
  there exists a constant $c>0$ with the following
  property. If $j$ is a singular moduli with discriminant $\Delta$, then 
\begin{equation*}
  \height{j} \le  c  |\Delta|^\epsilon. 
\end{equation*}
 \end{lemma}

If the Generalized Riemann Hypothesis (GRH) is true, then one may
replace $|\Delta|^\epsilon$ in this upper bound by $\log |\Delta|$.
We refer to 
 Lemmas 3 and 5 \cite{habegger:wbhc} for an even better estimate.

The special subvarieties of the 
 product $Y(1)^2$ are known. 
Points whose coordinates are both singular moduli
are precisely the special points and $Y(1)^2$ itself is the only
two-dimensional special subvariety.
Among the special curves  we find the vertical
 and horizontal lines where the fixed coordinate is
a singular moduli. 
The remain ones are $Y_0(N)\subset Y(1)^2$ and given by 
 the  zero-sets of the classical $N$th  modular transformation
polynomial for $N\in\IN$.
A complex point on such a special curve corresponds to a pair of
elliptic curves that are linked by an isogeny of degree $N$ with
cyclic kernel. 

If $C \subset Y(2)^2$ is not a special curve then the
author proved \cite{habegger:wbhc} that there is a constant $c>0$
such that  $C\cap Y_0(p)$ contains a point of height at least $c\log p$
for  all primes  $p\ge c^{-1}$.
Hence not even $C\cap \bigcup_{N\ge 1} Y_0(N)$ has bounded height.
 The situation already looks dire  in a product of two modular
curves. 

As there can be no such thing as a Bounded Height Conjecture in the
Shimura setting
we must content ourselves with something less. In the particular case of
curves in $Y(1)^2$ the author formulated the following conjecture.

If $S\subset Y(1)^2$ is an irreducible curve defined over $\IQbar$
then we let $\degS_\IQ{S}$ denote the degree of the union of all
conjugates of $S$ over $\IQ$.
We observe that any point $p \in Y(1)^2$ is contained in a uniquely
determined \emph{minimal} special subvariety ${\mathcal S}(p)\subset
Y(2)^2$.
For example, if $p=(j,*)$ is a special point, then $\mathcal{S}(p) = \{p\}$
and 
$\degS_\IQ \mathcal{S}(p) = [\IQ(p):\IQ] \ge [\IQ(j):\IQ]$.
If $\epsilon >0$ then,  by the Siegel-Brauer Theorem, $[\IQ(j):\IQ]$ grows at least of the
order
$|\Delta|^{1/2-\epsilon}$ where $\Delta$ is the discriminant of the
singular moduli $j$.

\begin{conjecture*}[Weakly Bounded Height Conjecture for $Y(1)^2$ \cite{habegger:wbhc}]
Let $C\subset Y(1)^2$ be an irreducible algebraic curve defined over
$\IQbar$ that is not special. There exists a constant $c>0$ with the following
property. Suppose $p\in C$ and  $\dim {\mathcal S}(p) \le 1$, then
\begin{equation}
\label{eq:whb}
  \height{p} \le c \log(1+\degS_\IQ S).
\end{equation}
\end{conjecture*}

By Corollary 1.2(ii) \cite{habegger:wbhc} 
the GRH implies this conjecture for a class
of curves satisfying a geometric restriction in addition to being
non-special.
Part (i) of this corollary implies  that we can still obtain 
a logarithmic height bound (\ref{eq:whb}) 
without assuming the GRH
if we restrict to points satisfying $\mathcal S(p)=Y_0(N)$ for some $N\in\IN$.
 GRH is used solely to obtain an  upper  for the height of a
singular moduli that is logarithmic in terms of its 
discriminant. But as we have seen in Lemma \ref{lem:phjbound}, polynomials upper bounds with
arbitrarily small exponent hold unconditionally.


The following, even weaker,  bounded height conjecture could bail us
out should the GRH default.

\begin{conjecture*}[Super Weakly Bounded Height Conjecture for
    $Y(1)^2$]
Let $\epsilon > 0$ and 
let $C\subset Y(1)^2$ be an irreducible algebraic curve defined over
$\IQbar$ that is not special. There exists a constant $c>0$ with the following
property. Suppose $p\in C$ and  $\dim {\mathcal S}(p) \le 1$, then 
\begin{equation*}
  \height{p} \le c (\degS_\IQ S)^{\epsilon}.
\end{equation*}
\end{conjecture*}

Already this conjecture has implications in direction of
 Pink's Conjecture. Let us suppose for the moment that it holds.
Using 
arguments laid out in the author's joint work with Pila \cite{hp:beyondAO}
one can show the following finiteness statement. 
Suppose $C\subset Y(1)^n$ is an irreducible  algebraic curve defined over
$\IQbar$ that is not contained in a proper special subvariety of
$Y(1)^n$. Then $C$ contains only finitely many points that are inside
a special subvariety of $Y(1)^n$ of codimension at least $2$.

Theorem 1 \cite{hp:beyondAO} implies this finiteness result unconditionally for curves
satisfying an additional geometric hypothesis.

\bibliographystyle{amsplain}
\bibliography{literatur,habeggerlit}

\vfill
\address{
\noindent
Philipp Habegger,
Johann Wolfgang Goethe-Universit\"at,
Robert-Mayer-Str. 6-8,
60325 Frankfurt am Main,
Germany,
{\tt habegger@math.uni-frankfurt.de}
}

\end{document}

%% file: macros.tex
\newcounter{thmcounter}
\newcounter{lemmacounter}
\newcounter{propcounter}
\newcounter{corcounter}

\newcounter{tmpcounter}
\newcounter{remcounter}
\newcounter{quecounter}

\numberwithin{equation}{section}
\numberwithin{lemmacounter}{section}
\numberwithin{propcounter}{section}
\numberwithin{remcounter}{section}

\theoremstyle{plain}
\newtheorem{theorem}[thmcounter]{Theorem}       
\newtheorem*{theorem*}{Theorem}
\newtheorem{lemma}[lemmacounter]{Lemma}
\newtheorem{lemma*}[tmpcounter]{Lemma}

\newtheorem*{conjecture*}{Conjecture}
\newtheorem{proposition}[propcounter]{Proposition}
\newtheorem{corollary}[corcounter]{Corollary}

\theoremstyle{definition}
\newtheorem*{example*}{Example}

\theoremstyle{remark}
\newtheorem{question}[quecounter]{Question}       
\newtheorem{remark}[remcounter]{Remark}

\def\IG{\mathbf G}
\def\IGm{\IG_{m}}

\def\IP{\mathbf P}

\def\IR{\mathbf R}
\def\IC{\mathbf C}
\def\IQ{\mathbf Q}
\def\IQbar{{\overline{\mathbf Q}}}
\def\IZ{\mathbf Z}
\def\IN{\mathbf N}

\newcommand{\heightS}{h}
\newcommand{\heightsS}{h_s}

\newcommand{\height}[1]{\heightS(#1)}
\newcommand{\heights}[1]{\heightsS(#1)}

\newcommand{\lh}[2]{\heightS_{#1}({#2})}
\newcommand{\matheight}[2]{\heightS_{#1}({#2})}
\newcommand{\vsheight}[1]{\heightS_{\rm Ar}({#1})}

\newcommand{\hpd}[1]{\beta({#1})}  

\newcommand{\pl}[1]{M_{#1}} 
\newcommand{\degS}{{\rm deg}}
\renewcommand{\deg}[1]{\degS({#1})}

\newcommand{\bigo}[1]{\mathcal{O}(#1)}

\newcommand{\mat}[2]{{\rm Mat}_{#1}({#2})}
\newcommand{\rk}[1]{\text{\rm rk}({#1})}

\newcommand{\oa}[1]{{#1}^{\rm oa}}
\newcommand{\oap}[2]{{#1}^{{\rm oa},[{#2}]}}

\newcommand{\essmin}[1]{{\mu}^{\rm ess}({#1})}
\def\codim{\mathop{\rm codim}\nolimits}

\newcommand{\orth}[1]{#1^{\bot}}

\newcommand{\trans}[1]{{#1}^{\mathsf{T}}}
\newcommand{\ordS}{\textrm{ord}}
\newcommand{\ord}[1]{\ordS({#1})}
\newcommand{\rank}[1]{{\rm rank}({#1})}

\newcommand{\atopx}[2]{\genfrac{}{}{0pt}{}{#1}{#2}}

\newcommand{\trop}[1]{\mathcal{T}({#1})}
\newcommand{\tropreg}[1]{\mathcal{T}^0({#1})}

\newcommand{\sgu}[1]{{(\IG_m^n)}^{[{#1}]}}
\newcommand{\sguG}[2]{{#1}^{[{#2}]}}

\newcommand{\ssm}{\smallsetminus}

%% file: tropical.tex
\section{Tropical Geometry and Degree Lower Bounds}
\label{sec:tropical}

Recall that $n$ and $r$ are positive integers. 
As usual, we identify elements of $\mat{r,n}{\IZ}$ 
with homomorphisms of algebraic groups
$\IGm^n\rightarrow \IGm^r$.
If $\varphi$ is such a matrix and $X\subset\IGm^n$ is an irreducible
closed subvariety defined over $\IC$ of dimension $r$, then $\deg{\varphi|_X}$ is the degree of the
restriction
$\varphi|_X:X\rightarrow\IGm^r$.
In this section we use techniques from Tropical Geometry to evaluate
$\deg\varphi|_X$ in terms of $\varphi$ and the so-called
tropicalization
 of $X$.

 Although a
$\varphi\in \mat{r,n}{\IQ}$ need not determine a homomorphism of algebraic groups
we can make sense of 
$\deg{\varphi|_X}$ by killing denominators,
cf. Lemma 3.1(iii) \cite{BHC}.
In this reference the author showed 
\begin{equation*}
  \deg{\lambda \varphi|_X} = |\lambda|^r \deg{\varphi|_X}
\end{equation*}
for all $\varphi\in\mat{r,n}{\IQ}$ and $\lambda\in\IQ$. 

Let $s$ be an integer with $r\le s\le n$.
Let $\epsilon > 0$, then $\varphi_0\in\mat{s,n}{\IR}$
is called $\epsilon$-regular if for any $\varphi\in\mat{s,n}{\IR}$
with $\rank{\varphi}<s$ we have $|\varphi_0-\varphi|_\infty \ge
\epsilon$.
In other words, the distance of $\varphi_0$ to the set of matrices of
non-full rank is at least $\epsilon$.

Let $\Pi_{rs}$ denote the
set of $r\times s$ matrices which represent projections
$\IQ^s\rightarrow\IQ^r$ onto $r$ distinct coordinates
of $\IQ^s$.

\begin{proposition}
\label{prop:degreelb}
Let $X\subsetneq\IGm^n$ be an irreducible closed subvariety 
defined over $\IC$ with
$\dim X = r \ge 1$.
Let $s$ be an integer with $r\le s \le n$.
One of the following two cases holds.
\begin{enumerate}
  \item [(i)] There exists an algebraic subgroup $H\subset\IGm^n$ such that 
$\dim_p X \cap pH \ge \max\{1,s+\dim H -n +1\}$
for all $p\in X(\IC)$.
\item[(ii)]
  If $\epsilon \in (0,1]$ and if $\varphi\in\mat{s,n}{\IQ}$ is
  $\epsilon$-regular,  there is $\pi\in\Pi_{rs}$ with
  \begin{equation*}
    \deg{\pi \varphi|_X}\ge 2^{-50n^5 (2n)^{n^2}}
\deg{X}^{-\frac 12 (n-r) (r+1+(r+3)(sn+1)(2r)^{sn})}
 \left(\frac{\epsilon}{\max\{1,|\varphi|_\infty\}}\right)^{sn(2r)^{sn}}.
  \end{equation*}
\end{enumerate}
\end{proposition}

\subsection{Preliminaries on Tropical Geometry}

Let $X\subset\IGm^n$ be an irreducible closed subvariety defined
over $\IC$ of
dimension $r\ge 1$. One can associate to $X$ 
a set $\trop{X}\subset \IQ^n$, called the tropicalization of $X$, 
in the follow manner.

Let $K$ be the field of puiseux series with complex
coefficients. This is an algebraically closed field
equipped with a  surjective 
valuation $\ordS :K\rightarrow \IQ \cup\{+\infty\}$.
We also use $\ordS:K^n\rightarrow (\IQ\cup\{+\infty\})^n$ to denote the
$n$-fold product. 
Let $X_K$ denote $X$ considered as a variety over $K$.
We set
\begin{equation}
\label{eq:definetrop}
  \trop{X} = \{\ord{x};\,\, x\in
  X_K(K) \}\subset \IQ^n.
\end{equation}

The closure of  $\trop{X}$ in $\IR^n$ coincides with the so-called Bieri-Groves set of $X$,
see the work of Einsiedler, Kapranov, and Lind \cite{EKL} for a
 proof. The Bieri-Groves set of $X$ is known \cite{BG:valuations} to be
a rational polyhedral set  of pure dimension $r$. 
It follows from the argument
given by Einsiedler, Kapranov, and Lind that $\trop{X}$ is the
intersection of a rational polyhedral set with $\IQ^n$. 

In our particular situation, the variety $X$ is defined over $\IC$. 
Because $\ordS$ is trivial on $\IC$ 
we find that $\trop{X}$ is a finite union of 
rational polyhedral cones of pure
dimension $r$. In other words, $\trop{X}$ is a finite union of
sets
\begin{equation*}
  \{(x_1,\dots,x_n)\in \IQ^n;\,\,
 a_{i1} x_1 + \cdots + a_{in} x_n \ge 0 \quad \text{for}\quad
{1\le i\le n}\}
\end{equation*}
with $a_{ij}\in \IZ$. Pure dimension $r$ means that the
  vector subspace of $\IQ^n$ spanned by each of the sets 
  has dimension $r$. 


We call $v\in\trop{X}$ regular if  some neighborhood $v$ in
$\trop{X}$ coincides with  the neighborhood
 of a vector subspace of $\IQ^n$.
In this case the vector subspace is uniquely determined and we denote
it by $L_v$. We define $\tropreg{X}\subset \trop{X}$ to be the set of
regular points.  
It follows that $\dim L_v=\dim X$ for  $v\in\tropreg{X}$.
We define the finite set
\begin{equation*}
 \Sigma(X) = \{L_v;\,\, v\in\tropreg{X}\}.
\end{equation*}

There is more information attached to the tropicalization of $X$.
One can define a locally constant
 function $m_X:\tropreg{X}\rightarrow \IN$ called the
multiplicity function, cf. Definition 3.7 \cite{SturmfelsTevelev}.


\begin{remark}
\label{rem:multiplicities}
We have $\trop{\IGm^n} = \IQ^n$, this is clear from our
characterization (\ref{eq:definetrop}). Hence $\tropreg{\IGm^n}
= \IQ^n$ too.  
If $v\in \IQ^n$, then $m_{\IGm^n}(v) = 1$
by Corollary 3.15 \cite{SturmfelsTevelev}.
\end{remark}


Our main tool from tropical geometry is a special case of a 
result of Sturmfels and
Tevelev \cite{SturmfelsTevelev}. It allows us to determine $\deg{\varphi|_X}$ for 
 varying  $\varphi$ (and fixed $X$). 

In the formulation below we regard $\varphi\in\mat{r,n}{\IZ}$
simultaneously as a homomorphism $\IGm^n\rightarrow\IGm^r$ and a
linear
map $\IQ^n\rightarrow\IQ^r$.

If we for the moment drop the assumption that $X$ has dimension $r$,
then 
\begin{equation}
\label{eq:tropimage}
  \trop{\overline{\varphi(X)}} = \varphi(\trop{X})
\end{equation}
holds  by Remark 2.1 \cite{SturmfelsTevelev}.

\begin{theorem}[Sturmfels, Tevelev]
\label{thm:ST}
  Let $X,n,\Sigma(X)=\{L_v\},$ and $r$  be as above and let $\varphi\in\mat{r,n}{\IZ}$ such
  that the restriction $\varphi|_X:X\rightarrow\IGm^r$ is
  dominant with degree $\deg{\varphi|_X}$.
Then $\varphi|_{\trop{X}}:\trop{X}\rightarrow \IQ^r$
is surjective. Moreover, assume $w\in\IQ^r$ such that
$\varphi|_{\trop{X}}^{-1}(w)$ is a finite subset of
$\tropreg{X}$, then
\begin{equation}
\label{eq:degree}
\deg{\varphi|_X} = 
  \sum_{v\in \varphi|_{\trop{X}}^{-1}(w) } m_{X}(v)
[\IZ^r:\varphi(L_v\cap\IZ^n)].
\end{equation}
\end{theorem}
\begin{proof}
This
  follows from Theorem  3.12 \cite{SturmfelsTevelev}. We remark  that
  $\varphi|_X:X\rightarrow \IGm^r$ is generically finite since it is
  dominant and   $\dim X = r$. Moreover, 
$m_{\IGm^n}(w)=1$ for all $w\in\IQ^n$ by Remark \ref{rem:multiplicities}.
\end{proof}

\begin{remark}
  Let $v$ be as in the sum (\ref{eq:degree}) and $v'\in L_v\cap\IZ^n$ with
$\varphi(v')=0$. Then $v+\lambda v'\in \tropreg{X}$ for $\lambda\in\IQ$
sufficiently small. Since $\varphi|_{\trop{X}}^{-1}(w)$ is
finite we must have $v'=0$. Hence $\varphi|_{L_v\cap\IZ^n}$ is
injective and therefore $\varphi(L_v\cap\IZ^n)$ has rank equal to
$\dim L_v = \dim X = r$. In particular, $\varphi(L_v\cap\IZ^n)$ has
finite index in $\IZ^r$, and the sum above is well-defined.

A similar argument shows that if the sum in (\ref{eq:degree}) contains
terms corresponding to two different $v$ and $v'$, then $L_v \not= L_{v'}$.
\end{remark}

In order to get explicit estimates from this theorem we need to
 get a grip on $\Sigma(X)$ and the multiplicity function for a given
 $X$.

We recall that $\deg{X}$ is the degree of the Zariski closure of $X$
in $\IP^n$. 

\begin{lemma}
\label{lem:multbound}
  Let $v\in\tropreg{X}$, then $m_{X}(v)\le \deg{X}$.
\end{lemma}
\begin{proof}
  Since $m_X$ is locally constant on $\tropreg{X}$ it suffices to
  prove the inequality for some $v'\in\tropreg{X}$ sufficiently close
  to $v$. We note also that $L_v = L_{v'}$ has dimension $r$. 
Let $\varphi\in\mat{r,n}{\IZ}$ be a projection
 onto $r$ distinct coordinates with
$\varphi|_{L_v}$ bijective.

Let $Y$ be the Zariski closure of $\varphi(X)$ in $\IGm^r$.
By (\ref{eq:tropimage}) the set $\trop{Y}$ contains the image of a
neighborhood of $v$ in $L_v$. Hence $\dim Y =\dim\trop{Y} =  r$ and we
conclude that $Y=\IGm^r$. Therefore, $\varphi|_X$ is dominant. 

We may apply Theorem \ref{thm:ST} to $\varphi$ and an
appropriate $w\in\IQ^r$ which we proceed to choose. 
The fiber $\varphi|_{L_v}^{-1}(w)$ is
finite by our choice of $\varphi$ regardless of $w$.
We recall that $\trop{X}$ is contained in a finite union of vector
subspaces of $\IQ^n$.
Hence, after replacing $v$ by a sufficiently close $v'\in
L_v\cap\tropreg{X}$
 we may assume that
$\varphi|_{\trop{X}}^{-1}(w)$ satifies the necessary
conditions. Now (\ref{eq:degree}) implies
$\deg{\varphi|_X} \ge m_{X}(v) [\IZ^r:\varphi(L_v\cap\IZ^n)]
\ge m_{X}(v)$. An application of B\'ezout's Theorem leads to
$\deg{\varphi|_X}\le\deg{X}$ and the lemma follows.
\end{proof}

\begin{lemma}
\label{lem:findLv}
We have $\#\Sigma(X) \le 
2^{5n^3} \deg{X}^{(r+1)(n-r)}$
 and
if $L\in\Sigma(X)$  there is  $\psi\in\mat{n-r,n}{\IZ}$
with $L = \ker \psi \subset\IQ^n$ and 
$|\psi|_\infty \le  2n\deg{X}$.
\end{lemma}
\begin{proof}
We may assume $r=\dim X \le n-1$, otherwise $X=\IGm^n$ and then
$\Sigma(X)$  contains only $\IQ^n$.

For an index set $I=\{i_1 <\cdots <i_{r+1} \}\subset \{1,\dots, n\}$
we let $\varphi_I:\IGm^n\rightarrow \IGm^{r+1}$
denote the projection onto the coordinates $i_1,\dots,i_{r+1}$.
The Zariski closure of $\varphi_I(X)$ in $\IGm^{r+1}$ has
dimension at most  $r$. 
Moreover, its degree is at most $\deg{X}$.
Hence it is in the zero set of a non-constant polynomial $f_I$ in $r+1$ variables and degree at
most $n\deg{X}$; for the latter statement we refer to Chardin's
Corollaire 2, cf. Exemple 1 \cite{Chardin}.

Say $v\in\trop{X}$. So
there is $x\in X_K(K)$ with $v=\ord{x}$. 

We have
$f_I(\varphi_I(x)) = 0$.
Because the coefficients of $f_I$
have valuation zero we must have $\ord{x^{v_I}}=0$
with $v_I\in\IZ^{r+1}\ssm\{0\}$ the difference of two
 distinct elements in the support of $f_I$.
On letting $I$ vary we find $n-r$ independent $v_I\in\IZ^n$
with $\ord{x^{v_I}}=0$. 
These define the rows of $\psi\in\mat{\IZ}{n-r,n}$ with rank $r$ whose kernel
contains $\ord{x}$.
In total there are ${n \choose r+1}$ possibilities for $I$ and there
are ${{n \choose r+1} \choose n-r}\le {n\choose r+1}^{n-r} \le
2^{n(n-r)}$ possibilities to choose $n-r$
different $I$. 
The sup-norm of the difference of two elements in the support
of $f_I$ is bounded by $2\deg{f_I} \le 2n\deg{X}$.
This leaves us with at most 
$(1+4n\deg{X})^{r+1}$
possibilities for $v_I$.
In total we obtain at most $2^{n(n-r)} (1+4n\deg{X})^{(r+1)(n-r)}$
 different $\psi$.
Using $r\ge 1$ and $r\le n-1$ we estimate
\begin{alignat*}1
 2^{n(n-r)}(1+4n\deg{X})^{(r+1)(n-r)} 
&\le 2^{n(n-r)} (8n\deg{X})^{(r+1)(n-r)}  \\
&\le 2^{4n(n-r)} n^{(r+1)(n-r)} \deg{X}^{(r+1)(n-r)} \\
&
\le  2^{4n(n-r)+n(r+1)(n-r)}\deg{X}^{(r+1)(n-r)} \\
&\le 2^{5n^3} \deg{X}^{(r+1)(n-r)}.
\end{alignat*}

Since $v\in\trop{X}$ was arbitrary we conclude that $\trop{X}$ is
contained in a union of at most 
$2^{5n^3} \deg{X}^{(r+1)(n-r)}$
vector subspaces of $\IQ^n$ determined by a $\psi$ as above. The lemma follows
since each $L\in \Sigma(X)$ has the same dimension as the kernel of
 a $\psi$.
\end{proof}

If $L\in\Sigma(X)$, then $L\cap\IZ^n\subset \IZ^n$ is a subgroup  of rank $r$.
In the next lemma we
 find a lattice basis for $L\cap\IZ^n$ with controlled entries.

 \begin{lemma}
\label{lem:defBL}
  For $L\in\Sigma(X)$ there is $B_L\in\mat{n,r}{\IZ}$ whose columns
$v_1,\dots,v_r$
are a basis for $L\cap\IZ^n$ with 
$|v_1|_\infty\cdots|v_r|_\infty\le 2^{n} r! n^{2n} \deg{X}^{n-r}$.
 \end{lemma}
 \begin{proof}
Lemma \ref{lem:findLv} supplies us with $\psi\in\mat{n-r,n}{\IZ}$ of
rank $r$ such that $\ker \psi = L\subset\IQ^n$ and $|\psi|_\infty\le 2n\deg{X}$.
   By Corollary 2.9.9 \cite{BG:valuations} there are independent 
   $v'_1,\dots,v'_r\in\IZ^n$  with $\psi(v'_1) = \cdots =
   \psi(v'_r)=0$
and $\prod_{k=1}^r |v'_k|_\infty \le n^{(n-r)/2} |\psi|_\infty^{n-r}$.
We note that our reference works with the multiplicative
projective height. Since our coefficients are integers, we may
compare this height to the norm $|\cdot|_\infty$. This is possible
since we 
suppose, as we may, that the entries of each $v'_k$ are coprime.
After permuting the $v'_k$ we may suppose
$|v'_1|_\infty\le\cdots \le |v'_r|_\infty$.

By a result of Mahler there is a basis
 $(v_1,\dots,v_k)$ of $\ker \psi$ with 
$|v_k|_\infty \le k |v'_k|_\infty$. This statement is similar
to  Lemma 3.2.11 \cite{BG}; said reference works
with a different norm and gives a somewhat different statement. 
But its proof adapts easily to yield the statement
above.

We conclude $\prod_{k=1}^r |v_k|_\infty \le r! n^{(n-r)/2}
|\psi|_\infty^{n-r} \le 2^{n} r! n^{2n} \deg{X}^{n-r}$
as desired.
 \end{proof}

In the following proposition we use $B_L$ as given by the previous lemma.

 \begin{proposition}
\label{prop:degree}
   Let $\varphi\in\mat{r,n}{\IZ}$.
   \begin{enumerate}
   \item [(i)] 
We have
     \begin{equation}
\label{eq:degreeub}
       \deg{\varphi|_X} \le  \deg{X} \max_{L\in\Sigma(X)}
\{ |\det{\varphi B_L}| \}\#\Sigma(X).
     \end{equation}
    \item[(ii)] 
If
 $L\in \Sigma(X)$ then
      \begin{equation*}
        \deg{\varphi|_X} \ge |\det{\varphi B_L}|.
      \end{equation*}
   \end{enumerate}
 \end{proposition}
 \begin{proof}
   In the proof of the proposition we will use
   \begin{equation*}
     [\IZ^r: \varphi(L\cap\IZ^n)] = |\det {\varphi B_L}|
   \end{equation*}
for $L\in\Sigma(X)$ where we define the index to be $0$ if $\varphi(L\cap\IZ^n)$
has rank less than $r$.

We begin with part (i). 
It suffices to assume that $\varphi|_X$ is dominant.
We note that $m_X(v)\le\deg{X}$ for all
$v\in\tropreg{X}$ by Lemma \ref{lem:multbound}. Inequality
(\ref{eq:degreeub})
follows from Sturmfels and Tevelev's result if there exists
$w\in\IQ^r$ such that $\varphi|_{\trop{X}}^{-1}(w)$ 
is a finite subset of $\tropreg{X}$. 
Indeed, for $w$ outside a finite union of proper vector subspace of
$\IQ^r$ the fiber $\varphi|_{\trop{X}}^{-1}(w)$ does not meet
any $L'\in \Sigma(X)$ with $\varphi|_{L'}$ non-injective or any points
of
$\trop{X}\ssm\tropreg{X}$


Let us now prove part (ii). 
Without loss of generality we may assume $\det{\varphi B_L}\not=0$. In
other words, $\varphi|_{L}:L\rightarrow\IQ^r$ is bijective.

We remark that the tropicalization of the Zariski closure of
$\varphi(X)$ 
is all $\IQ^r$
by (\ref{eq:tropimage}).  In particular, $\varphi|_X$ is dominant.


In view of the Sturmfels and Tevelev's
result, it suffices to show that there is $w\in\IQ^r$ 
such that  $\varphi|_{\trop{X}}^{-1}(w)$ is a finite subset of
$\tropreg{X}$. This is the case by the same argument we gave in the
proof of part (i).
 \end{proof}

\subsection{Ax's Theorem and Degree Lower Bounds}


The following proposition is a  variant of the author's Proposition
7.3 \cite{BHC}. Its proof relies on Ax's Theorem \cite{Ax}.

\begin{proposition}
\label{prop:alternative}
Let $X$ be an irreducible closed subvariety of $\IGm^n$
defined over $\IC$  of
dimension $1\le r \le n-1$. Let $s$ be an integer with $r\le s\le n$.
  Assume $\varphi_0\in\mat{s,n}{\IR}$ has rank $s$. Then one of the
  following two cases holds.
  \begin{enumerate}
  \item [(i)] There exists an algebraic subgroup $H\subset\IGm^n$ such that 
$\dim_p X \cap pH \ge \max\{1,s+\dim H -n +1\}$
for all $p\in X(\IC)$.
  \item[(ii)] There exist $\epsilon >0$ and an open neighborhood $U$ of
    $\varphi_0$
in $\mat{s,n}{\IR}$ such that 
$  \max_{\pi\in\Pi_{rs}} \deg{\pi \varphi|_X} \ge \epsilon $ for
all $\varphi\in U\cap\mat{s,n}{\IQ}$.
  \end{enumerate}
\end{proposition}
\begin{proof}
  Take $\mathcal{K} = \{\varphi_0\}$ in Proposition 7.3 \cite{BHC}.
\end{proof}

Using the results from tropical geometry in the previous subsection we
will eventually transform this qualitative lower bound into a quantitative one.
At first we show a non-vanishing result. 

Throughout this subsection, and if not stated otherwise, we keep the notation of Proposition \ref{prop:alternative}.

\begin{lemma}
\label{lem:extendtoR}
Let us assume part (i) of Proposition \ref{prop:alternative}
does not hold for $X$.
If $\varphi_0\in\mat{s,n}{\IR}$ has rank $s$, then
there exist $L\in\Sigma(X)$ and $\pi\in\Pi_{rs}$ with 
\begin{equation*}
  \det(\pi \varphi_0 B_L) \not =0.
\end{equation*}
\end{lemma}
\begin{proof}
Let $U\subset\mat{s,n}{\IR}$ be the neighborhood of $\varphi_0$ from
Proposition \ref{prop:alternative}. Let $\varphi\in
U\cap\mat{s,n}{\IQ}$
and $\lambda\in\IN$ with $\lambda\varphi\in\mat{s,n}{\IZ}$.
We have $\max_{\pi\in\Pi_{rs}}\deg{\lambda\pi\varphi|_X} = \lambda^r
\max_{\pi\in\Pi_{rs}}\deg{\pi\varphi|_X} \ge
\lambda^r \epsilon$.
On the other hand, Proposition \ref{prop:degree}(i)
implies 
\begin{alignat*}1
 \deg{\lambda\pi\varphi|_X} &\le \deg{X} \max_{L\in\Sigma(X)} 
|\det(\lambda \pi \varphi B_L)| \#\Sigma(X)\\
&=\deg{X}\lambda^r \max_{L\in\Sigma(X)} 
|\det(\pi\varphi B_L)|\#\Sigma(X)
\end{alignat*}
for all $\pi\in\Pi_{rs}$.
We cancel $\lambda$ and  conclude 
\begin{equation*}
\max_{\pi\in\Pi_{rs}} 
 \max_{L\in\Sigma(X)}
|\det(\pi\varphi B_L)| \ge \epsilon > 0
\end{equation*}
for all $\varphi\in U\cap\mat{r,n}{\IQ}$. 

By continuity and since
$U\cap\mat{r,n}{\IQ}$ lies dense in $U$ we have the
same inequality for $\varphi_0$. 
In particular, $\det(\pi\varphi_0 B_L)\not=0$
for some $L$ and $\pi$.
\end{proof}

The following corollary is an consequence of the lemma above. It will
be of no further relevance for the current article. Although it would
be interesting to know if it could be proved using methods from
tropical geometry instead of Ax's Theorem.

We let $\overline{\trop{X}}$ be the closure of $\trop{X}\subset\IQ^n$
in $\IR^n$. 

Let us now prove Corollary \ref{cor:tropical} which is 
stated in Section
\ref{sec:effbhc}. 

\begin{proof}
  For $\varphi_0$ as in the hypothesis,  Lemma
  \ref{lem:extendtoR} implies that we are in  alternative (i) of Proposition
  \ref{prop:alternative}. Let $H$ be the algebraic subgroup mentioned
  there. We fix a surjective homomorphism  $\varphi':\IG_m^n
  \rightarrow \IG_m^{n-\dim H}$ with kernel $H$. 

Say $n-\dim H \ge r$. Then we compose $\varphi'$ with the projection
onto the first $r$ coordinates of $\IG_m^{n-\dim H}$ and obtain a
surjective homomorphism $\varphi:\IG_m^n\rightarrow \IG_m^r$. 
Each fiber of $\varphi|_X$ has dimension at least
$\max\{1,s+\dim H - n + 1\}$. By the Fiber Dimension Theorem,
the Zariski closure of $\varphi(X)$ has dimension at most
$\dim X - \max\{1,s+\dim H - n + 1\} <\dim X$. 
If we consider $\varphi$ as a matrix in $\mat{r,n}{\IQ}$, then it has
rank $r$. Moreover, by (\ref{eq:tropimage}) we have
 $\varphi(\trop{X}) \not= \IQ^r$.

Now let us assume $n-\dim H < r$. 
We remark that 
\begin{equation*}
 s + \dim H - n + 1 \ge r + \dim H - n + 1 > 1. 
\end{equation*}
We take the product of $\varphi'$
with some homomorphism
 $\IG_m^n \rightarrow \IG_m^{r -(n-\dim H)}$ in general position
to obtain a surjective homomorphism $\varphi:\IG_m^n\rightarrow
\IG_m^r$. By intersection theory, the fibers of $\varphi|_X$ have
dimension at least
\begin{equation*}
  s+\dim H - n + 1 - (r - (n - \dim H))
=s-r+1 \ge 1.
\end{equation*}
As before, $\varphi(X)$ does not lie Zariski dense in $\IG_m^r$ and we
 also conclude  $\varphi(\trop{X})\not=\IQ^r$. 
\end{proof}

\begin{example*}
  The previous corollary shows that the collection of $\IQ$-vector spaces
  $\{L_v\}$ associated to  $X$ satisfy a
  certain rationality condition. 
Indeed, it can be reformulated as follows. 
Let $\overline{L_v}$ be the closure of $L_v$ in $\IR^n$ or
equivalently, the vector subspace of $\IR^n$ generated by $L_v$.
If there exists a vector
subspace $V_0\subset\IR^n$ of dimension $n-r$ with 
\begin{equation*}
  \overline{L_v}\cap V_0 \not= 0  \quad\text{for all}\quad v
\end{equation*}
then there exists a vector subspace $V\subset \IQ^n$ of dimension
$n-r$ with
\begin{equation*}
  {L_v}\cap V \not= 0  \quad\text{for all}\quad v.
\end{equation*}

We exhibit four $2$-dimensional 
vector subspaces in $\IQ^4$ that do not satisfy this rationality
condition. The four bases
\begin{alignat*}1
&\big(
(1,0,1,0), (0,-2,0,1) 
\big),  \\
&\big(
(1,-1,0,0), (0,0,1,1) 
\big), \\
&\big(
(0,1,0,0), (0,0,1,0) 
\big),  \quad\text{and} \\
&\big(
(1,0,0,0), (0,0,0,1) 
\big)
\end{alignat*}
determine four planes $L_1,L_2,L_3,$ and $L_4$ in $\IQ^4$,
respectively.

Say $V\subset \IQ^4$ is an arbitrary $2$-dimensional vector subspace of
$\IQ^4$. One can verify that $V\cap L_i=0$ for some $i\in \{1,2,3,4\}$. 

On the other hand, each $L_i$ meets
the vector subspace of $\IR^4$ with basis
\begin{equation*}
 \big((0,\sqrt{2},1,0), (-\sqrt{2},0,0,1) \big)
\end{equation*}
non-trivially. 

In particular, $\Sigma(X)\not=\{L_1,L_2,L_3,L_4\}$ for all 
irreducible surfaces $X\subset\IG_m^4$.
Sam Payne has pointed out to the author that this also follows from
the fact that the tropicalization of a variety is connected in
codimension one.
Indeed, the pairwise intersection of the $L_i$ is trivial. 
\end{example*}

Let $M_{ij}$ be independent variables for $1\le i \le s$ and $1\le
j\le n$. They are entries of an $s\times n$ matrix
 $M=(M_{ij})$. We define the polynomial
\begin{equation*}
  D_X = \sum_{\pi\in\Pi_{rs}} \sum_{L\in\Sigma(X)} 
\det(\pi M B_L)^2 \in \IZ[M_{ij}].
\end{equation*}
Lemma \ref{lem:extendtoR} tells us that if $X$ is not as in
alternative (i) of Proposition \ref{prop:alternative} then
\begin{equation}
\label{eq:zeroset}
\varphi_0\in\mat{s,n}{\IR}\quad\text{with}\quad
  D_X(\varphi_0) = 0\quad\text{implies}\quad
\rank{\varphi_0} < s.
\end{equation}

This polynomial is closely related to the degree map
$\varphi\mapsto\deg{\varphi|_X}$. It has the advantage that it can
 be evaluated for  $\varphi$ with real entries. Recall that
 $\deg{\varphi|_X}$ has no  straightforward interpretation for
 irrational $\varphi$.

Since we can contain the vanishing locus
of $D_X$ we  will 
 obtain an explicit  lower bound for its values.
To do this we must first  determine basic properties of $D_X$.

For a polynomial $P$ in any number of variables and integer
coefficients we let $|P|_\infty$ denote the largest absolute value of
any coefficient.

\begin{lemma}
\label{lem:DXprop}
  \begin{enumerate}
  \item [(i)] The polynomial $D_X$ is homogeneous of degree $2r$.
  \item[(ii)] We have $|D_X|_\infty 
\le 2^{20n^3} \deg{X}^{(r+3)(n-r)}$.
  \end{enumerate}
\end{lemma}
\begin{proof}
  Part (i) follows from properties of the determinant. 
 
We turn to Part (ii). First we remark that if $P_1,\dots,P_k$ are 
polynomials in any number of variables, then 
$|P_1+\cdots+P_k|_\infty \le |P_1|_\infty+\cdots+|P_k|_\infty$
by the triangle inequality. 
The product can bounded using
 Lemma 1.6.11 \cite{BG} as $|P_1\cdots P_k|_\infty \le 2^{d}
|P_1|_\infty\cdots|P_k|_\infty$
where $d$ is the sum of the partial degrees of $P_1\cdots P_k$.

Say $L\in\Sigma(X)$ and $\pi\in\Pi_{rs}$. 
Let $v_1,\dots,v_r\in\IZ^n$ be the columns of $B_L$;
latter is given by Lemma \ref{lem:defBL}.
An entry in the $j$-th
column of $\pi M B_L$ is a linear form with coefficients
among the coefficients of $v_j$.
Each term in the Leibniz formula for $\det(\pi M B_L)$
is a  polynomial in at most $nr$ distinct $M_{ij}$ with each partial degree at most $1$.
Since there are $r!$ such terms we obtain
$|\det(\pi M B_L)|_\infty \le  2^{nr}r! |v_1|_\infty \cdots
|v_r|_\infty$.
Each partial degree of $\det(\pi M B_L)^2$ is at most $2$ and at most $nr$
variables appear, so
$|\det(\pi M B_L)^2|_\infty \le 2^{2nr}|\det(\pi M B_L)|^2_\infty
\le 2^{4nr} r!^2 |v_1|^2_\infty\cdots |v_r|_\infty^2$.
The bound from Lemma \ref{lem:defBL}
leads to 
\begin{equation}
\label{eq:oneterm}
|\det(\pi M B_L)^2|_\infty \le
2^{4nr} r!^2 (2^nr! n^{2n}\deg{X}^{n-r})^2.
\end{equation}

Now $D_X$ consists
of $\#\Pi_{rs} \#\Sigma = {s \choose r} \#\Sigma\le 2^s\#\Sigma$
terms for the form $\det(\pi M B_L)^2$.
We use the bound for $\#\Sigma \le 2^{5n^3} \deg{X}^{(r+1)(n-r)}$ 
from  Lemma \ref{lem:findLv} and
(\ref{eq:oneterm}) to deduce
\begin{alignat*}1
  |D_X|_\infty &\le
2^s 2^{5n^3}\deg{X}^{(r+1)(n-r)} 2^{4nr} r!^2 (2^nr!
n^{2n}\deg{X}^{n-r})^2\\
&\le 
2^{s+5n^3+4nr + 2n}
n^{4n}
r!^4 \deg{X}^{(r+1)(n-r) + 2(n-r)} \\
&\le
2^{12n^3} n^{8n} \deg{X}^{(r+3)(n-r)}\\
&\le
2^{12n^3} 2^{8n^2} \deg{X}^{(r+3)(n-r)}
\end{alignat*}
where we used $r,s\le n-1$ and $n\le 2^n$.
\end{proof}

We now prove Proposition \ref{prop:degreelb}.

Let us assume that we are not in alternative (i). Then we are not in
alternative (i) of Proposition \ref{prop:alternative}
and the conclusion of Lemma \ref{lem:extendtoR} holds.

By (\ref{eq:zeroset}) we know that the set of zeros of $D_X$ is
contained in  the set of  matrices of $\mat{s,n}{\IR}$ with non-full rank.

  Let $\varphi$ and $\epsilon$ be as in the hypothesis. 
We apply R\'emond's explicit \L ojasiewicz inequality
\cite{Remond:Lojasiewicz} to bound $D_X(\varphi)$ from below.

Recall that $D_X$ is a polynomial of  degree $2r$ 
in $sn$ variables $M_{ij}$. 
R\'emond's inequality and  Lemma \ref{lem:DXprop}(i) imply
\begin{alignat*}1
 D_X(\varphi) = |D_X(\varphi)| &\ge (e^{4rsn} |D_X|_\infty)^{-(sn+1)
   (2r)^{sn}}\delta 
\ge 2^{-12rs^2n^2(2r)^{sn}} |D_X|_\infty^{-(sn+1)(2r)^{sn}} \delta
\end{alignat*}
where for brevity we set 
\begin{equation*}
\delta=
\left(\frac{\epsilon}{\max\{1,|\varphi|_\infty\}}\right)^{sn(2r)^{sn}} \le 1.
\end{equation*}

Next we use the bound for $|D_X|_\infty$ given by Lemma
\ref{lem:DXprop}(ii); using $n^{n}\le 2^{n^2}$ we find
\begin{alignat*}1
D_X(\varphi)&\ge 
 2^{-12 r s^2 n^2 (2r)^{sn}-20 n^3(sn+1)(2r)^{sn}}
 \deg{X}^{-(sn+1)(2r)^{sn}(r+3)(n-r)}\delta.
\end{alignat*}
The exponent of $2^{-1}$
is 
\begin{equation*}
  12rs^2n^2(2r)^{sn} + 20n^3(sn+1)(2r)^{sn}
\le
 ( 12 + \frac{3}{2}20)n^5(2n)^{n^2}
\le 50 n^5 (2n)^{n^2}
\end{equation*}
where we used $s,r\le n$ and $n\ge 2$.
We conclude
\begin{equation}
\label{eq:DXvarphilb}
  D_X(\varphi) \ge 2^{-50n^5 (2n)^{n^2}} 
 \deg{X}^{-(sn+1)(2r)^{sn}(r+3) (n-r)}\delta.
\end{equation}

By definition of $D_X$ there exist $L\in\Sigma(X)$ and
$\pi\in\Pi_{rs}$ with 
\begin{equation*}
   D_X(\varphi)  \le \#\Pi_{rs} \#\Sigma(X)
\det(\pi \varphi B_L)^2. 
\end{equation*}
 Lemma \ref{lem:findLv} gives
\begin{alignat*}1
D_X(\varphi)&\le 
2^s 2^{5n^3} \deg{X}^{(r+1)(n-r)}
\det(\pi \varphi B_L)^2 
\le  2^{10n^3}\deg{X}^{(r+1)(n-r)}
\det(\pi \varphi B_L)^2.
\end{alignat*}
Together with (\ref{eq:DXvarphilb})  we conclude
\begin{equation}
\label{eq:detlb}
  \det(\pi\varphi B_L)^2 \ge 
2^{-100n^5 (2n)^{n^2}} 
\deg{X}^{-(r+1)(n-r)-(sn+1)(2r)^{sn}(r+3)(n-r)}\delta.
\end{equation}

Say $\lambda\in\IN$ such that $\lambda\varphi\in\mat{r,n}{\IZ}$.
We apply Proposition \ref{prop:degree}(ii) to get
$\deg{\lambda \pi \varphi|_X} \ge |\det{\lambda \pi \varphi B_L}| 
= \lambda^r |\det{\pi\varphi B_L}|$.
Since $\deg{\lambda\pi\varphi|_X} = \lambda^r \deg{\pi\varphi|_X}$
we have $\deg{\pi\varphi|_X} \ge |\det{\pi\varphi B_L}|$.

The proposition now follow from (\ref{eq:detlb}).
\qed